\documentclass{amsart}
\usepackage{latexsym,amsmath,amsthm,verbatim,ifthen,amssymb,graphicx,color}
\usepackage[all]{xy}
\usepackage{color}
\usepackage{tikz}
\usepackage{hyperref}
\usepackage{xurl}

\usetikzlibrary{tikzmark, calc}

\tikzset{every picture/.style={remember picture}}

\usetikzlibrary{3d} 
\usetikzlibrary{decorations.pathmorphing}
\usetikzlibrary{matrix}
\usepackage{caption}
\usepackage{enumitem}

\SelectTips{cm}{}

\newtheorem{theorem}{Theorem}[section]

 \newtheorem{lemma}[theorem]{Lemma}
 
 \theoremstyle{definition}
 
 \theoremstyle{remark}
 \newtheorem{remark}[theorem]{Remark}
 
 \numberwithin{equation}{section}

\newcommand{\restoreparindent}{\setlength{\parindent}{\normalparindent}}

\DeclareRobustCommand{\stirling}{\genfrac\{\}{0pt}{}}

\definecolor{darkgreen}{rgb}{0,0.8,0}

\begin{document}

\title[Higher dimensional visual proofs]{Higher dimensional visual proofs, Nicomachus' 4D Theorem and the mysterious irreducible factor  $(3n^2+3n-1)$  in the sum of fourth powers}

\author[U. Buijs]{Urtzi Buijs}
\address{Departamento de Algebra, Geometr\'{\i}a y Topolog\'{\i}a, Universidad
de M\'alaga, Ap. 59, 29080 M\'alaga, Spain}
\email{ubuijs@uma.es}

\begin{abstract}
Sums of powers $S_p(n)=\sum_{k=1}^n k^p$ can be described by Faulhaber's formula in terms of the Bernoulli numbers. The first cases of this formula admit visual proofs of various kinds, which lead to factorized Faulhaber polynomials.

In this article we present a technique that yields higher-dimensional visual proofs for these factorized formulas, providing a geometric interpretation of the roots that appear.

In particular, we prove Nicomachus’s Theorem in four dimensions, and we visually explain the appearance, in dimension five, of the irreducible factor $(3n^2 +3n-1)$ in the polynomial ring over the rational numbers.
.\end{abstract}

\maketitle

\section{Introduction}
Mathematicians have long been fascinated by figurate numbers. The Pythagoreans, as early as the 6th century BC, showed a clear interest in classifying numbers according to geometric patterns. In doing so, they introduced triangular, square, and other polygonal numbers \cite[p. 76]{Heat}.

We still use some figurate terminology today, such as square numbers,
$1, 4, 9, \dots$, which can literally be arranged in square patterns using stones or tiles.
Although these geometric representations cannot normally match the versatility of algebra, they do make it possible to carry out certain proofs visually and with remarkable simplicity. One of the most familiar examples is the visual proof showing that the sum of consecutive odd numbers is always a square.

If we represent each odd number 
$2n+1$ as a corner-shaped arrangement formed by two segments of length 
$n$ meeting at a single square, then the diagram shows that $1+3+5+\cdots+(2n-1) = n^2$.
\vskip .5cm

\begin{tikzpicture}
\draw [fill=green!50, very thick] (0,0) rectangle (0.5,0.5);

\draw [fill=orange!40, very thick] (2,0) rectangle (2.5,0.5);
\draw [fill=orange!40, very thick] (1.5,-0.5) rectangle (2,0);
\draw [fill=orange!40.0, very thick] (2,-0.5) rectangle (2.5,0);

\foreach  \y in {0, 0.5}{
\draw [fill=blue!40, very thick] (4.5,-\y) rectangle (5,-\y+0.5);
}
\draw [fill=blue!40.0, very thick] (3.5+1,-1) rectangle (4+1,-0.5);

\foreach  \x in {0, 0.5}{
\draw [fill=blue!40, very thick] (3.5+\x,-1) rectangle (4+\x,-0.5);
}

\foreach  \y in {0, 0.5, 1}{
\draw [fill=red!40, very thick] (7.5,-\y) rectangle (8,-\y+0.5);
}
\draw [fill=red!40.0, very thick] (6+1.5,-1.5) rectangle (6.5+1.5,-1);
\foreach  \x in {0, 0.5, 1}{
\draw [fill=red!40, very thick] (6+\x,-1.5) rectangle (6.5+\x,-1);
}

\foreach  \y in {0, 0.5, 1}{
\draw [fill=red!40, very thick] (7.5+3.5,-\y) rectangle (8+3.5,-\y+0.5);
}
\draw [fill=red!40.0, very thick] (6+5,-1.5) rectangle (6.5+5,-1);
\foreach  \x in {0, 0.5, 1}{
\draw [fill=red!40, very thick] (6.5+3+\x,-1.5) rectangle (7+3+\x,-1);
}

\foreach  \y in {0, 0.5}{
\draw [fill=blue!40, very thick] (4.5+6,-\y) rectangle (5+6,-\y+0.5);
}
\draw [fill=blue!40.0, very thick] (3.5+7,-1) rectangle (4+7,-0.5);

\foreach  \x in {0, 0.5}{
\draw [fill=blue!40, very thick] (4+5.5+\x,-1) rectangle (4.5+5.5+\x,-0.5);
}

\draw [fill=orange!40, very thick] (1+9,0) rectangle (1.5+9,0.5);
\draw [fill=orange!40.0, very thick] (1.5+8.5,-0.5) rectangle (2+8.5,0);
\draw [fill=orange!40, very thick] (2+7.5,-0.5) rectangle (2.5+7.5,0);

\draw [fill=green!50, very thick] (0+9.5,0) rectangle (0.5+9.5,0.5);

\node at (1,0.25) {$+$};
\node at (3,0.25) {$+$};
\node at (5.5,0.25) {$+$};
\node at (8.5,0.25) {$=$};
\end{tikzpicture}
\vspace{.5cm}
\captionof{figure}{Sum of consecutive odd numbers}
\label{Odd numbers}

\hskip 15pt The sum of any number of consecutive natural numbers $1, 2, 3, \dots $
defines a triangular number, with a standard geometric interpretation.

\begin{tikzpicture}
\draw [fill=blue!55, very thick] (0,0) rectangle (0.5,0.5);
\node[below] at (-0.25,-0.3) {$\hskip 1cm t_1=1$};
\draw [fill=blue!55, very thick] (2,0) rectangle (2.5,0.5);
\node[below] at (2.5,-0.3) {$t_2=1+2$};
\draw [fill=blue!55, very thick] (2,0.5) rectangle (2.5,1);
\draw [fill=blue!55, very thick] (2.5,0) rectangle (3,0.5);
\draw [fill=blue!55, very thick] (4.5,0.5) rectangle (5,1);
\node[below] at (5.25,-0.3) {$t_3=1+2+3$};
\draw [fill=blue!55, very thick] (4.5,1) rectangle (5,1.5);
\draw [fill=blue!55, very thick] (5,0.5) rectangle (5.5,1);
\draw [fill=blue!55, very thick] (4.5,0) rectangle (5,0.5);
\draw [fill=blue!55, very thick] (5,0) rectangle (5.5,0.5);
\draw [fill=blue!55, very thick] (5.5,0) rectangle (6,0.5);
\draw [fill=blue!55, very thick] (7.5,1) rectangle (8,1.5);
\node[below] at (8.5,-0.3) {$t_4=1+2+3+4$};
\draw [fill=blue!55, very thick] (7.5,1.5) rectangle (8,2);
\draw [fill=blue!55, very thick] (8,1) rectangle (8.5,1.5);
\draw [fill=blue!55, very thick] (7.5,0.5) rectangle (8,1);
\draw [fill=blue!55, very thick] (8,0.5) rectangle (8.5,1);
\draw [fill=blue!55, very thick] (8.5,0.5) rectangle (9,1);
\draw [fill=blue!55, very thick] (7.5,0) rectangle (8,0.5);
\draw [fill=blue!55, very thick] (8,0) rectangle (8.5,0.5);
\draw [fill=blue!55, very thick] (8.5,0) rectangle (9,0.5);
\draw [fill=blue!55, very thick] (9,0) rectangle (9.5,0.5);

\end{tikzpicture}
\captionof{figure}{Triangular numbers}
\restoreparindent
Although finding the explicit value of a triangular number may at first seem daunting, because it involves adding many consecutive terms, this task can be greatly simplified by a classic visual argument. The key observation is that, while a single triangular number \(t_n\) may be awkward to compute directly, the doubled quantity \(2t_n\) is much easier to handle. By rotating and translating a copy of the triangular arrangement in the plane, we can fit the two copies together into a rectangle, reorganising the same information in a more efficient way and making the value of $t_n$ easy to read off.
\vskip .1cm
\begin{tikzpicture}[scale=0.9]
 \draw[line width=0pt] (0,0) -- (0,0);
\begin{scope}[shift={(-6.4,0)}]
    \draw[->,>=stealth, line width=0.5mm] (13,2) .. controls (13,3.75) and (10.5,3.75) .. (10.5,2);
  \draw[->,>=stealth, line width=0.5mm] (10,2) .. controls (9.5,2) and (9.5,2) .. (9,2);   
\draw[<->] (7.25,0) -- (7.25,2.5) node[midway, left] {$n$};
\draw [fill=blue!55, very thick] (7.5,1.5) rectangle (8,2);
\node[below] at (10.25,1.5) {$+$};
\node[below] at (14.25,1.5) {$=$};
\draw [fill=blue!55, very thick] (7.5,2) rectangle (8,2.5);
\draw [fill=blue!55, very thick] (8,1.5) rectangle (8.5,2);
\draw [fill=blue!55, very thick] (7.5,1) rectangle (8,1.5);
\draw [fill=blue!55, very thick] (8,1) rectangle (8.5,1.5);
\draw [fill=blue!55, very thick] (8.5,1) rectangle (9,1.5);
\draw [fill=blue!55, very thick] (7.5,0.5) rectangle (8,1);
\draw [fill=blue!55, very thick] (8,0.5) rectangle (8.5,1);
\draw [fill=blue!55, very thick] (8.5,0.5) rectangle (9,1);
\draw [fill=blue!55, very thick] (9,0.5) rectangle (9.5,1);
\draw [fill=blue!55, very thick] (7.5,0) rectangle (8,0.5);
\draw [fill=blue!55, very thick] (8,0) rectangle (8.5,0.5);
\draw [fill=blue!55, very thick] (8.5,0) rectangle (9,0.5);
\draw [fill=blue!55, very thick] (9,0) rectangle (9.5,0.5);
\draw [fill=blue!55, very thick] (9.5,0) rectangle (10,0.5);
\draw [fill=red!55, very thick] (11,1.5) rectangle (11.5,2);
\draw [fill=red!55, very thick] (11,2) rectangle (11.5,2.5);
\draw [fill=red!55, very thick] (11.5,1.5) rectangle (12,2);
\draw [fill=red!55, very thick] (11,1) rectangle (11.5,1.5);
\draw [fill=red!55, very thick] (11.5,1) rectangle (12,1.5);
\draw [fill=red!55, very thick] (12,1) rectangle (12.5,1.5);
\draw [fill=red!55, very thick] (11,0.5) rectangle (11.5,1);
\draw [fill=red!55, very thick] (11.5,0.5) rectangle (12,1);
\draw [fill=red!55, very thick] (12,0.5) rectangle (12.5,1);
\draw [fill=red!55, very thick] (12.5,0.5) rectangle (13,1);
\draw [fill=red!55, very thick] (11,0) rectangle (11.5,0.5);
\draw [fill=red!55, very thick] (11.5,0) rectangle (12,0.5);
\draw [fill=red!55, very thick] (12,0) rectangle (12.5,0.5);
\draw [fill=red!55, very thick] (12.5,0) rectangle (13,0.5);
\draw [fill=red!55, very thick] (13,0) rectangle (13.5,0.5);
\draw[<->] (7.25+8,0) -- (7.25+8,2.5) node[midway, left] {$n$};
\draw [fill=blue!55, very thick] (7.5+8,1.5) rectangle (8+8,2);
\draw [fill=blue!55, very thick] (7.5+8,2) rectangle (8+8,2.5);
\draw [fill=blue!55, very thick] (8+8,1.5) rectangle (8.5+8,2);
\draw [fill=blue!55, very thick] (7.5+8,1) rectangle (8+8,1.5);
\draw [fill=blue!55, very thick] (8+8,1) rectangle (8.5+8,1.5);
\draw [fill=blue!55, very thick] (8.5+8,1) rectangle (9+8,1.5);
\draw [fill=blue!55, very thick] (7.5+8,0.5) rectangle (8+8,1);
\draw [fill=blue!55, very thick] (8+8,0.5) rectangle (8.5+8,1);
\draw [fill=blue!55, very thick] (8.5+8,0.5) rectangle (9+8,1);
\draw [fill=blue!55, very thick] (9+8,0.5) rectangle (9.5+8,1);
\draw [fill=blue!55, very thick] (7.5+8,0) rectangle (8+8,0.5);
\draw [fill=blue!55, very thick] (8+8,0) rectangle (8.5+8,0.5);
\draw [fill=blue!55, very thick] (8.5+8,0) rectangle (9+8,0.5);
\draw [fill=blue!55, very thick] (9+8,0) rectangle (9.5+8,0.5);
\draw [fill=blue!55, very thick] (9.5+8,0) rectangle (10+8,0.5);
\draw [fill=red!55, very thick] (7.5+8.5,2) rectangle (8+8.5,2.5);
\draw [fill=red!55, very thick] (7.5+9,2) rectangle (8+9,2.5); 
\draw [fill=red!55, very thick] (7.5+9.5,2) rectangle (8+9.5,2.5);
\draw [fill=red!55, very thick] (7.5+10,2) rectangle (8+10,2.5);
\draw [fill=red!55, very thick] (7.5+10.5,2) rectangle (8+10.5,2.5);

\draw [fill=red!55, very thick] (7.5+9,1.5) rectangle (8+9,2); 
\draw [fill=red!55, very thick] (7.5+9.5,1.5) rectangle (8+9.5,2);
\draw [fill=red!55, very thick] (7.5+10,1.5) rectangle (8+10,2);
\draw [fill=red!55, very thick] (7.5+10.5,1.5) rectangle (8+10.5,2);

\draw [fill=red!55, very thick] (7.5+9.5,1) rectangle (8+9.5,1.5);
\draw [fill=red!55, very thick] (7.5+10,1) rectangle (8+10,1.5);
\draw [fill=red!55, very thick] (7.5+10.5,1) rectangle (8+10.5,1.5);

\draw [fill=red!55, very thick] (7.5+10,0.5) rectangle (8+10,1);
\draw [fill=red!55, very thick] (7.5+10.5,0.5) rectangle (8+10.5,1);

\draw [fill=red!55, very thick] (7.5+10.5,0) rectangle (8+10.5,0.5);
\draw[<->] (15.5,-0.25) -- (18.5,-0.25) node[midway, below] {$n+1$};

\end{scope}
\end{tikzpicture}
 
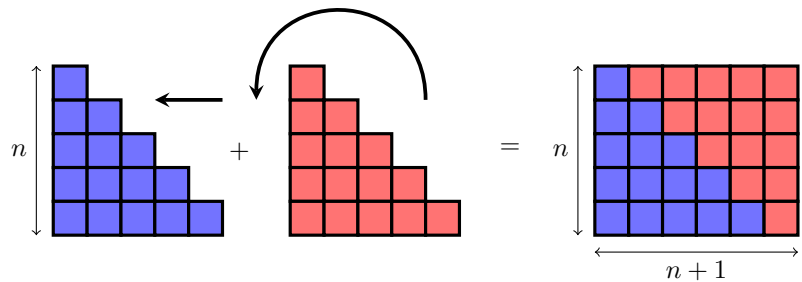
\captionof{figure}{Visual proof for the formula for triangular numbers}\label{Gauss}

\restoreparindent Since the number of tiles in a rectangular arrangement is simply the product of its base and height, we obtain $2t_n =n(n+1)$ and therefore the well-known formula for triangular numbers: 
\begin{equation} \label{Triangular}
t_n = 1+2+3+\cdots + n = \frac{n(n+1)}{2}.
\end{equation}

\restoreparindent The problem of finding a simple formula for the sum of the first 
$n$ squares was already studied by Archimedes (287--212 BC). In his work {\em On Spirals}, in connection with the area under the Archimedean spiral, he derived the following formula 

 \begin{equation}\label{Squares}
 1^2 + 2^2 + 3^3 + \cdots + n^2 = \frac{n(n+1)(2n+1)}{6}.
 \end{equation}
 
\restoreparindent This formula follows from the identity \cite[Proposition 10, p. 162]{Heat2}:
\begin{equation}\label{ArchimedesFormula}
(1+2+3+\cdots + n)+(n+1)n^2 = 3(1^2+2^2+\cdots + n^2).
\end{equation}
Using the expression $1+2+3+\cdots +n =\frac{n(n+1)}{2}$ and simplifying, one readily obtains formula $(\ref{Squares})$. This identity also admits a visual interpretation \cite{Pengelley}, illustrated in Figure $\ref{ArchimedesProof}$. 

\restoreparindent The sum of the first $n$ cubes was already discussed by Nicomachus of Gerasa in the 1st century AD in his {\em Introduction to Arithmetic} \cite{Nicomachus}, where he described the pattern underlying the first cases of the formula. The result, which may at first seem almost miraculous, is in fact the identity
\begin{equation}\label{Nicomachus}
1^3+2^3+3^3+\cdots + n^3 = (1+2+3+\cdots +n)^2.
\end{equation}

\restoreparindent This result was later discussed by other authors, notably the Indian mathematician Aryabhata in his Aryabhatiya (499 AD) \cite{Aryabhata}.

\restoreparindent A visual proof of the identity also exists \cite{Gulley}, shown in Figure $\ref{ClassicalNicomachus}$.
\vskip 0.5cm


{\LARGE $t_n + (n+1)n^2$\ \ \ $=$\ \ \ \ \ $3(1^2 +2^2+3^2+\cdots +n^2$)}
\captionof{figure}{Visual proof of Archimedes’ identity}\label{ArchimedesProof}
\vskip 0.2cm
{\huge $1^3$\ \ $+$\ \ $2^3$\ \ \ \ $+$\ \ \ \ $3^3$\ \ \ $+$\ \ \ $\cdots$\ \ \ $+$\ \ \ \ \ $n^3$}
\vskip .1cm
%

\captionof{figure}{Visual proof of Nicomachus’ Theorem}\label{ClassicalNicomachus}

\restoreparindent The visual proofs for the sums of squares and cubes shown in Figures~\ref{ArchimedesProof} and~\ref{ClassicalNicomachus} are undoubtedly striking and aesthetically pleasing, yet they are not free from criticism. Archimedes’ argument relies on the idea that counting tiles in the $n$ squares of sides $1, 2, 3, \dots, n$ is cumbersome, and that arranging three copies of this configuration allows one to count the same total more easily as $3(1^2 + 2^2+ 3^2+\cdots +n^2)$. However, the visual proof in Figure~\ref{ArchimedesProof} does not seem to follow the same simple and intuitive pattern as the one in Figure~\ref{Gauss}.

In contrast, the construction used in Figure~\ref{ClassicalNicomachus} represents a more substantial conceptual shift. Rather than simplifying the count by combining pieces, it proceeds in the opposite direction: it breaks the configuration into smaller parts and rearranges them to form a new figure.

The main difficulty is that these visual arguments do not seem to generalise in any straightforward way to obtain formulas for the sums of $p$-th powers, $S_p(n)=\sum_{k=1}^n k^p$.

The search for formulas for the sums of fourth, fifth, and higher powers has a long history involving many distinguished mathematicians.

In 1631, Johann Faulhaber (1580--1635) published explicit formulas for the sums of powers up to the 17th case, although he did not give a general expression \cite{Faulhaber1631,Knuth1993}. A few years later, in 1636, Pierre de Fermat (1601--1665) introduced recurrence relations based on figurate numbers that expressed any such sum in terms of earlier ones, but the resulting computations soon became unmanageable \cite{Mahoney1973, BeeryFermat}.

As a consequence of his study of the arithmetic triangle, Blaise Pascal (1623--1662) discovered in 1654 a more practical formula for any sum, although still in recursive form \cite{Pascal1654,Katz2009}.

The first closed formula for all sums of powers was obtained by Jacob Bernoulli (1654--1705) in his posthumous work \emph{Ars Conjectandi} \cite{BernoulliArs}.
 There he introduced the numbers that now bear his name and showed that the sequence $B_0, B_1, B_2, \dots$ provides a uniform formula for these sums. He wrote with evident excitement:

\vspace{6pt}
\textit{“…it took me less than half of a quarter of an hour to find that the tenth power of the first 1000 numbers being added together would yield the sum 91,409,924,241,424,243,424,241,924,242,500.”}
\vspace{6pt}

The closed formula he obtained is the following:
\begin{equation} \label{FaulhaberFormula}
S_p(n)=\sum_{k=1}^n k^p 
= \frac{1}{p+1}\sum_{j=0}^p \binom{p+1}{j}B_j\,n^{p+1-j},
\end{equation}
where the numbers $\{ B_n\}_{n=0}^\infty $ are defined recursively by:
\begin{equation}\label{Bernoulli numbers}
\sum_{i=0}^m \binom{m+1}{i}B_i = m+1,\qquad m=0,1,2,\dots 
\end{equation}

The first values of this numerical sequence are listed in the following table:

\begin{table}[h]
\centering
\renewcommand{\arraystretch}{1.5} 
\begin{tabular}{ccccccccccccccccc}
\hline
\(n\) & 0 & 1 & 2 & 3 & 4 & 5 & 6 & 7 & 8 & 9 & 10 & 11 & 12 & 13 & 14 & 15  \\ \hline
\(B_n\) & 1 & \(\frac{1}{2}\) & \(\frac{1}{6}\) & 0 & \(-\frac{1}{30}\) & 0 & \(\frac{1}{42}\) & 0 & \(-\frac{1}{30}\) & 0 & \(\frac{5}{66}\) & 0 & \(-\frac{691}{2730}\) & 0 & \(\frac{7}{6}\) & 0 \\ \hline
\end{tabular}
\caption{Values of the first Bernoulli numbers.}
\label{tab:bernoulli_ext}
\end{table}

Thus, the formulas for the cases $p=1,2,3,4$ are as follows:
\begin{equation}\label{FirstCases}
\begin{aligned} 
\sum_{k=1}^n k^{\ } &=\frac{1}{2}(n^2 + n) = \frac{n(n+1)}{2},\\
\sum_{k=1}^n k^2 &=\frac{1}{3}(n^3+\frac{3}{2}n^2 + \frac{1}{2}n) = \frac{n(n+1)(2n+1)}{6},\\
\sum_{k=1}^n k^3 &=\frac{1}{4}(n^4 + 2n^3+n^2) = \frac{n^2(n+1)^2}{4},\\
\sum_{k=1}^n k^4 &=\frac{1}{5}(n^5 + \frac{5}{2}n^4+\frac{5}{3}n^3-\frac{1}{6}n) = \frac{n(n+1)(2n+1)(3n^2+3n-1)}{30}.
\end{aligned}
\end{equation}
\vskip .1cm

However, the proof of formula~\eqref{FaulhaberFormula} for sums of powers is purely algebraic (see, for example, \cite{Arakawa}). In this work we aim to provide geometric insight into these formulas by means of visual proofs.

Having concluded our historical overview of the sums of powers, we now return to the need for a simple and generalisable proof in the case of the sums of squares. The following proof of formula~\eqref{Squares} is very recent despite its simplicity \cite{Siu, Nelsen}.

In this setting we count unit cubes. Thus, summing the first 
$n$ squares amounts to counting a single unit cube, then the cubes in a square of edge length $2$, then those in a square of edge length $3$, and so on up to a square of edge length $n$.

These cubes can be arranged to form a square-based pyramid of height $n$. Thus, computing $1^2 + 2^2+ 3^2+\cdots +n^2$
is equivalent to counting the number of unit cubes in this pyramid.
\vskip .6cm
{\huge \ $1^2$\ \ \ \ $+$\ \ \ \ $2^2$\ \ \ \ \ $+$\ \ \ \ \ $3^2$\ \ \ $+$\ \ \ $\cdots$\ \ \ $+$\ \ \ \ $n^2$}
\vskip .3cm
\begin{tikzpicture}[scale=0.5]
    \tikzset{cube/.style={very thin,draw=black, opacity=1}}
    \begin{scope}[rotate around x=0, rotate around y=-35]
       
        {

                 \draw[cube] (0,0,0) -- (1,0,0) -- (1,0,1) -- (0,0,1) -- cycle; 

                \draw[cube] (0,0,0) -- (0,0,1) -- (0,1,1) -- (0,1,0) -- cycle; 
               
                \draw[cube] (0,0,0) -- (1,0,0) -- (1,1,0) -- (0,1,0) -- cycle; 
                
                 \draw[cube, fill=orange!55] (0,1,0) -- (1,1,0) -- (1,1,1) -- (0,1,1) -- cycle; 
                \draw[cube, fill=orange!55!black] (0,0,1) -- (1,0,1) -- (1,1,1) -- (0,1,1) -- cycle; 
                \draw[cube, fill=orange!30] (1,0,0) -- (1,1,0) -- (1,1,1) -- (1,0,1) -- cycle; 

        }
    \end{scope}

\begin{scope}[shift={(4,0)}]
 \begin{scope}[rotate around x=0, rotate around y=-35]
        \foreach \x in {0,1}
        \foreach \y in {0}
        \foreach \z in {1,0}
        {
            \begin{scope}[shift={(\x,\y,-\z)}]
                
                 \draw[cube] (0,0,0) -- (1,0,0) -- (1,0,1) -- (0,0,1) -- cycle; 

                \draw[cube] (0,0,0) -- (0,0,1) -- (0,1,1) -- (0,1,0) -- cycle; 
               
                \draw[cube] (0,0,0) -- (1,0,0) -- (1,1,0) -- (0,1,0) -- cycle; 
                
                 \draw[cube, fill=orange!55] (0,1,0) -- (1,1,0) -- (1,1,1) -- (0,1,1) -- cycle; 
              \draw[cube, fill=orange!55!black] (0,0,1) -- (1,0,1) -- (1,1,1) -- (0,1,1) -- cycle; 
                \draw[cube, fill=orange!30] (1,0,0) -- (1,1,0) -- (1,1,1) -- (1,0,1) -- cycle; 
                
            \end{scope}
        }
    \end{scope}

\end{scope}

\begin{scope}[shift={(10,0)}]
 \begin{scope}[rotate around x=0, rotate around y=-35]
        \foreach \x in {0,1,2}
        \foreach \y in {0}
        \foreach \z in {2,1,0}
        {
            \begin{scope}[shift={(\x,\y,-\z)}]
                
                 \draw[cube] (0,0,0) -- (1,0,0) -- (1,0,1) -- (0,0,1) -- cycle; 

                \draw[cube] (0,0,0) -- (0,0,1) -- (0,1,1) -- (0,1,0) -- cycle; 
               
                \draw[cube] (0,0,0) -- (1,0,0) -- (1,1,0) -- (0,1,0) -- cycle; 
                
                 \draw[cube, fill=orange!55] (0,1,0) -- (1,1,0) -- (1,1,1) -- (0,1,1) -- cycle; 
                \draw[cube, fill=orange!55!black] (0,0,1) -- (1,0,1) -- (1,1,1) -- (0,1,1) -- cycle; 
                \draw[cube, fill=orange!30] (1,0,0) -- (1,1,0) -- (1,1,1) -- (1,0,1) -- cycle; 
                
            \end{scope}
        }
    \end{scope}

\end{scope}

\begin{scope}[shift={(17.5,0)}]
 \begin{scope}[rotate around x=0, rotate around y=-35]
        \foreach \x in {0,1,2,3}
        \foreach \y in {0}
        \foreach \z in {3,2,1,0}
        {
            \begin{scope}[shift={(\x,\y,-\z)}]
                
                 \draw[cube] (0,0,0) -- (1,0,0) -- (1,0,1) -- (0,0,1) -- cycle; 

                \draw[cube] (0,0,0) -- (0,0,1) -- (0,1,1) -- (0,1,0) -- cycle; 
               
                \draw[cube] (0,0,0) -- (1,0,0) -- (1,1,0) -- (0,1,0) -- cycle; 
                
                 \draw[cube, fill=orange!55] (0,1,0) -- (1,1,0) -- (1,1,1) -- (0,1,1) -- cycle; 
                \draw[cube, fill=orange!55!black] (0,0,1) -- (1,0,1) -- (1,1,1) -- (0,1,1) -- cycle; 
                \draw[cube, fill=orange!30] (1,0,0) -- (1,1,0) -- (1,1,1) -- (1,0,1) -- cycle; 
                
            \end{scope}
        }
    \end{scope}

\end{scope}

\draw[<->] (19,-1.5) -- (22.7,-0.2); \node[right] at (20.7, -1.3) {$n$};
\end{tikzpicture}
\begin{tikzpicture}
    \draw[dashed, thick] (0,0) -- (0,0);
    \draw[dashed, thick] (0,-0.5) -- (12.5,-0.5);
    \draw[dashed, thick] (0,-1) -- (0,-1);
\end{tikzpicture}
\begin{tikzpicture}[scale=0.5]
    \tikzset{cube/.style={very thin,draw=black, opacity=1}}

\begin{scope}[shift={(0,-3)}]
 \begin{scope}[rotate around x=0, rotate around y=-35]
        \foreach \x in {0,1,2,3}
        \foreach \y in {0}
        \foreach \z in {3,2,1,0}
        {
            \begin{scope}[shift={(\x,\y,-\z)}]
                
                 \draw[cube] (0,0,0) -- (1,0,0) -- (1,0,1) -- (0,0,1) -- cycle; 

                \draw[cube] (0,0,0) -- (0,0,1) -- (0,1,1) -- (0,1,0) -- cycle; 
               
                \draw[cube] (0,0,0) -- (1,0,0) -- (1,1,0) -- (0,1,0) -- cycle; 
                
                 \draw[cube, fill=orange!55] (0,1,0) -- (1,1,0) -- (1,1,1) -- (0,1,1) -- cycle; 
                \draw[cube, fill=orange!55!black] (0,0,1) -- (1,0,1) -- (1,1,1) -- (0,1,1) -- cycle; 
                \draw[cube, fill=orange!30] (1,0,0) -- (1,1,0) -- (1,1,1) -- (1,0,1) -- cycle; 
                
            \end{scope}
        }
    \end{scope}

\end{scope}

\begin{scope}[shift={(0,-2)}]
 \begin{scope}[rotate around x=0, rotate around y=-35]
        \foreach \x in {0,1,2}
        \foreach \y in {0}
        \foreach \z in {2,1,0}
        {
            \begin{scope}[shift={(\x,\y,-\z)}]
                
                 \draw[cube] (0,0,0) -- (1,0,0) -- (1,0,1) -- (0,0,1) -- cycle; 

                \draw[cube] (0,0,0) -- (0,0,1) -- (0,1,1) -- (0,1,0) -- cycle; 
               
                \draw[cube] (0,0,0) -- (1,0,0) -- (1,1,0) -- (0,1,0) -- cycle; 
                
                 \draw[cube, fill=orange!55] (0,1,0) -- (1,1,0) -- (1,1,1) -- (0,1,1) -- cycle; 
                \draw[cube, fill=orange!55!black] (0,0,1) -- (1,0,1) -- (1,1,1) -- (0,1,1) -- cycle; 
                \draw[cube, fill=orange!30] (1,0,0) -- (1,1,0) -- (1,1,1) -- (1,0,1) -- cycle; 
                
            \end{scope}
        }
    \end{scope}

\end{scope}

\begin{scope}[shift={(0,-1)}]
 \begin{scope}[rotate around x=0, rotate around y=-35]
        \foreach \x in {0,1}
        \foreach \y in {0}
        \foreach \z in {1,0}
        {
            \begin{scope}[shift={(\x,\y,-\z)}]
                
                 \draw[cube] (0,0,0) -- (1,0,0) -- (1,0,1) -- (0,0,1) -- cycle; 

                \draw[cube] (0,0,0) -- (0,0,1) -- (0,1,1) -- (0,1,0) -- cycle; 
               
                \draw[cube] (0,0,0) -- (1,0,0) -- (1,1,0) -- (0,1,0) -- cycle; 
                
                 \draw[cube, fill=orange!55] (0,1,0) -- (1,1,0) -- (1,1,1) -- (0,1,1) -- cycle; 
              \draw[cube, fill=orange!55!black] (0,0,1) -- (1,0,1) -- (1,1,1) -- (0,1,1) -- cycle; 
                \draw[cube, fill=orange!30] (1,0,0) -- (1,1,0) -- (1,1,1) -- (1,0,1) -- cycle; 
                
            \end{scope}
        }
    \end{scope}

\end{scope}

 \begin{scope}[rotate around x=0, rotate around y=-35]
       
        {

                 \draw[cube] (0,0,0) -- (1,0,0) -- (1,0,1) -- (0,0,1) -- cycle; 

                \draw[cube] (0,0,0) -- (0,0,1) -- (0,1,1) -- (0,1,0) -- cycle; 
               
                \draw[cube] (0,0,0) -- (1,0,0) -- (1,1,0) -- (0,1,0) -- cycle; 
                
                 \draw[cube, fill=orange!55] (0,1,0) -- (1,1,0) -- (1,1,1) -- (0,1,1) -- cycle; 
                \draw[cube, fill=orange!55!black] (0,0,1) -- (1,0,1) -- (1,1,1) -- (0,1,1) -- cycle; 
                \draw[cube, fill=orange!30] (1,0,0) -- (1,1,0) -- (1,1,1) -- (1,0,1) -- cycle; 

        }
    \end{scope}

\draw[<->] (-1.25,0.7) -- (-1.25,-3.3); \node[right] at (-2.25, -1.5) {$n$};
 
\end{tikzpicture}
\vspace{0.5cm}
\captionof{figure}{Sum of squares and unit cubes in a pyramid}\label{Pyramid}
\vskip .5cm

\restoreparindent Let us denote by $p_n$ the number of unit cubes in this pyramid. As with triangular numbers, although $p_n$ may be difficult to compute directly, it is somewhat paradoxical that $3p_n$ turns out to be much easier to determine.

\restoreparindent Indeed, if we take three identical pyramids, we can rotate and translate them in three-dimensional space to assemble them into a single solid block.

\vskip .5cm
{\huge \hskip 2cm $3\ (\ 1^2\ +\ 2^2\ +\ 3^2\ +\ \cdots \ +\ n^2\  )$}
\vskip .7cm


\vskip .3cm
\captionof{figure}{Three pyramids forming an almost rectangular block}\label{Threepyramids}

\restoreparindent Since the block has a stepped top surface, we need a small geometric “DIY step” —a brief adjustment to smooth the top layer— in order to obtain a perfect parallelepiped.

\restoreparindent  The resulting parallelepiped (see Figure $\ref{DIY}$) has $n+1$ units in length, $n$ in width, and  $n+\frac{1}{2}$ in height. It follows that the total number of unit cubes in the three square-based pyramids of height $n$ is $n(n+1)(n+\frac{1}{2})$.

\restoreparindent Thus, we obtain the formula:
\begin{equation}\label{SumofSquaresDenominator3}
p_n = 1^2 + 2^2+ 3^2 +\cdots + n^2 = \frac{n(n+1)(n+\frac{1}{2})}{3}.
\end{equation}


\vskip .2cm
\captionof{figure}{Geometric DIY step to level the top layer}\label{DIY}

\restoreparindent Although formulas~\eqref{Squares} and~\eqref{SumofSquaresDenominator3} are essentially the same, the way they are written plays an important role in introducing the fundamental idea of this work.

\restoreparindent The visual proofs in Figure~\ref{Gauss} for triangular numbers $1+2+3+\cdots +n$, and in Figures~\ref{Threepyramids} and~\ref{DIY} for pyramidal numbers $1^2 + 2^2 + 3^2+\cdots +n^2$ are entirely analogous. In both cases the key idea is the same: combining two triangular arrangements to form a rectangle, or combining three square-based pyramids to obtain a parallelepiped. As a consequence, the resulting formulas naturally have a denominator equal to the number of pieces used, which in turn matches the dimension in which the visual proof takes place.

\restoreparindent The number of factors in the numerator also reflects the dimension. These factors correspond precisely to the measurements of the resulting block, and the numbers appearing in them —the roots of the polynomial in $n$— arise from the geometric adjustments required to obtain a perfect block.
\vskip .3cm
\begin{tikzpicture}
\draw[-, line width=0mm] (0,0) -- (0, 0);

\begin{scope}[shift={(3.5,-2)}]
    \node (frac) at (0,0) {\huge $\displaystyle \frac{n(n+1)\left(n+\tikz[baseline=(char.base)]\node[draw, dashed, line width=0.5pt, inner sep=4pt, rectangle](char){$\frac{1}{2}$};\right)}{3}$};

    
      \draw[->, line width=0.3mm] (1.7,1) -- (1.7, 1.5) node[midway, above, yshift=5pt]{Term produced by the geometric ‘DIY step’.};
      \draw[->, line width=0.3mm] (2.5,0.4) -- (3, 0.4) node[right,  align=left]{\hskip .2cm As many factors as indicated \\\hskip .2cm  by the denominator.};
       \draw[->, line width=0.3mm] (1.5,-0.7) -- (3, -0.7) node[right,  align=left]{\hskip .2cm The dimension in which  \\\hskip .2cm  the visual proof takes place.};
     
\end{scope}
\end{tikzpicture}
\captionof{figure}{Structure of the formula for the sum of powers}\label{Structure}
\vskip .5cm

\restoreparindent Although the visual proof of Nicomachus’ Theorem (see Figure~\ref{ClassicalNicomachus}) does not follow the same pattern as the previous two, the resulting identity
$1^3 +2^3+3^3+ \cdots +n^3 = (1+2+3+\cdots +n )^2$ nonetheless exhibits the same structural features described in Figure~\ref{Structure}. Since 
$1+2+3+\cdots + n = \frac{n(n+1)}{2}$, we obtain 
\begin{equation}\label{formula4D}
\begin{aligned}
1^3 +2^3+3^3+ \cdots +n^3&=(1+2+3+\cdots +n )^2\\
&=\Bigl( \frac{n(n+1)}{2}\Bigr)^2\\
&=\frac{n^2(n+1)^2}{4}.
\end{aligned} 
\end{equation}

This expression also appears in \eqref{FirstCases} via Faulhaber’s formula.

\restoreparindent In other words, the formula has four factors in the numerator and a $4$ in the denominator; yet the visual proof itself does not take place in four–dimensional space. This naturally leads to the following question:
\vskip .3cm
\textbf{Is it possible to give a visual proof of Nicomachus’ Theorem in 4–dimensions by assembling four pyramids of hypercubes so that the product of the dimensions of the resulting 
4–dimensional block yields formula~\eqref{formula4D}?}
\vskip .3cm

\restoreparindent Once this objective has been achieved, we can be more ambitious. If we examine the fourth identity in~\eqref{FirstCases} —Faulhaber’s formula for the sum of fourth powers— we see that it can be rewritten in the structural form shown in Figure~\ref{Structure}:
\begin{equation}\label{enigmatic}
1^4+2^4+3^4+\cdots +n^4 = \frac{n(n+1)(n+\frac{1}{2})(n^2+n-\frac{1}{3})}{5}.
\end{equation}
The term $\frac{1}{2}$ has appeared before, and we understand the type of geometric adjustment that gives rise to it. But what is the meaning of the irreducible factor $n^2+n-\frac{1}{3}$?

This leads to a second question:

\vskip .3cm
\textbf{Can we extend the previous visual constructions and obtain a visual proof in 5D of formula~\eqref{enigmatic} that reveals the geometric nature of the irreducible factor
$n^2+n-\frac{1}{3}$?}
\vskip.3cm

\restoreparindent In this work we address both questions. The article is organised as follows.

\noindent Section~\ref{LowerDimensional} introduces the preliminary tools needed for our constructions.

\noindent Section~\ref{LowerSections} presents the technique we will use to visualize higher–dimensional objects. In particular, we will see how different lower–dimensional sections of the same object provide different answers to the same question, yielding identities immediately. More importantly, this approach will give us flexibility when assembling the two–dimensional puzzles that arise later.

\noindent Section~\ref{Sumin2D} applies this technique to give a two–dimensional visual proof of formula~\eqref{SumofSquaresDenominator3}, corresponding to Figures~\ref{Threepyramids} and~\ref{DIY}.

\noindent Section~3 is devoted to Nicomachus’ Theorem in four dimensions.

\noindent Section~3.1 explains how to work with three–dimensional sections of four–dimensional hypercubes to provide a visual proof of formula~\eqref{formula4D}.

\noindent Section~3.2 reduces this argument by one dimension, translating the proof into a two–dimensional puzzle. This reduction will be essential when extending the method to higher dimensions.

\noindent Section~4 presents a visual proof of formula~\eqref{enigmatic} for the sum of fourth powers.

\noindent Section~4.1 assembles four 5D pyramids and displays the resulting configuration as a 2D puzzle.

\noindent Section~4.2 incorporates the fifth 5D pyramid into the same puzzle.

\noindent Section~4.3 introduces a first cut-and-paste operation that accounts for the factors $n(n+1)$.

\noindent Section~4.4 is devoted to understanding the second cut-and-paste operation, which explains the remaining factor $n^2+n-\tfrac{1}{3}$.

\noindent Section~4.5 derives the factor $n+\tfrac{1}{2}$ by clarifying the construction carried out in the previous section.


\begin{remark}[Previous Work on Visual Proofs in Higher Dimensions]
Before beginning our journey through visual proofs in higher dimensions, we briefly mention several earlier works that have explored related ideas from different perspectives:
\begin{itemize}[leftmargin=.5cm]
\item[(1)] One of the fundamental references on figurate numbers is undoubtedly {\em The Book of Numbers} by John H. Conway and Richard K. Guy \cite{Conway}. In Chapter 2 ({\em Figures from Figures Doing Arithmetic and Algebra by Geometry}), the section on the third dimension shows, in a striking way, that hex pyramids are in fact cubes, using the three-dimensional analogue of the proof in Figure~\ref{Odd numbers}. In the section devoted to the fourth dimension, the authors remark: “Although it’s hard to visualize jigsaw puzzles in four dimensions, it can be done!”, and they employ the technique of assembling blocks and multiplying their dimensions to obtain expressions for the {\em pentatope numbers}.

\item [(2)] In \cite{Sasho}, the author begins with a visual proof of the identity $1+2+3+\cdots +n=\frac{n(n+1)}{2}$ different from the one shown in Figure~\ref{Gauss}, and generalizes this construction to three dimensions in order to obtain a visual proof of identity~\eqref{Squares} for the sum of squares. The transition to the next dimension is carried out using visual representations of hypercubes, and the author even manages to give a visual demonstration of identity~\eqref{formula4D} for the sum of cubes. In the Epilogue, however, he acknowledges the following obstacle: 
\vskip .2cm
\emph{
“...there should be a five-dimensional block made of five-dimensional
parallelepipeds in such a way that the corresponding volumes yield the
above identity. However, something new intervenes in five dimensions:
the polynomial \(3n^{2}+3n-1\) is irreducible over the integers.
I do not know if that significantly affects the procedure, but
visualization will certainly be obstructed by the technical problem of
having to draw \(5\)-dimensional objects in \(2\) dimensions.”
}
\end{itemize}
\end{remark}

\section{Lower dimensional sections and the sum of squares} \label{LowerDimensional}
In this section, we introduce the technique of using sections to visualize a three-dimensional figure through its lower-dimensional slices. To illustrate this method, we will revisit the visual proof of formula~\eqref{SumofSquaresDenominator3}, given in Figures~\ref{Threepyramids} and~\ref{DIY}, by examining its planar sections. In this way, the original solid-geometry problem is transformed into a two-dimensional puzzle.

\subsection{Lower-dimensional sections} \label{LowerSections}
One way to visualise a three-dimensional figure is through its planar X-rays. For example, if we wish to view the square-based pyramid in Figure~\ref{Pyramid} as if we were beings from {\em Flatland} \cite{Abbott}, we can take suitable two-dimensional sections of the figure. Since our goal is to count the number of unit cubes efficiently using images, it suffices to choose X-rays in which each cube is scanned exactly once, allowing us to count squares in the planar sections instead of cubes in the original solid.

In the figure we see that the number of squares appearing in the X-rays is $1^2+2^2+3^2+\cdots +n^2$, , which is to be expected: the number of unit cubes in the pyramid is precisely the same sum.

We will refer to the planar sections of the pyramid whose X-rays produce sums of powers (in this case, squares) as the {\em Main Sections}.
\vskip .5cm
\begin{tikzpicture}[scale=0.47]

    \tikzset{cube/.style={very thin,draw=black, opacity=1}}

\begin{scope}[shift={(1.5,1)}]
\begin{scope}[shift={(0,-3)}]
 \begin{scope}[rotate around x=0, rotate around y=-35]
        \foreach \x in {0,1,2,3}
        \foreach \y in {0}
        \foreach \z in {3,2,1,0}
        {
            \begin{scope}[shift={(\x,\y,-\z)}]
                
                 \draw[cube] (0,0,0) -- (1,0,0) -- (1,0,1) -- (0,0,1) -- cycle; 

                \draw[cube] (0,0,0) -- (0,0,1) -- (0,1,1) -- (0,1,0) -- cycle; 
               
                \draw[cube] (0,0,0) -- (1,0,0) -- (1,1,0) -- (0,1,0) -- cycle; 
                
                 \draw[cube, fill=orange!55] (0,1,0) -- (1,1,0) -- (1,1,1) -- (0,1,1) -- cycle; 
                \draw[cube, fill=orange!55!black] (0,0,1) -- (1,0,1) -- (1,1,1) -- (0,1,1) -- cycle; 
                \draw[cube, fill=orange!30] (1,0,0) -- (1,1,0) -- (1,1,1) -- (1,0,1) -- cycle; 
                
            \end{scope}
        }
    \end{scope}

\end{scope}

\begin{scope}[shift={(0,-2)}]
 \begin{scope}[rotate around x=0, rotate around y=-35]
        \foreach \x in {0,1,2}
        \foreach \y in {0}
        \foreach \z in {2,1,0}
        {
            \begin{scope}[shift={(\x,\y,-\z)}]

                \draw[cube, fill=orange!55] (0,0.5,0) -- (1,0.5,0) -- (1,0.5,1) -- (0,0.5,1) -- cycle; 
                \draw[cube, fill=orange!55!black] (0,0,1) -- (1,0,1) -- (1,0.5,1) -- (0,0.5,1) -- cycle; 
                \draw[cube, fill=orange!30] (1,0,0) -- (1,0.5,0) -- (1,0.5,1) -- (1,0,1) -- cycle; 
                
            \end{scope}
        }
    \end{scope}

\end{scope}

\begin{scope}[shift={(2.5,-0.8)}]
\begin{scope}[rotate around x=0, rotate around y=-35]
  \draw[cube, fill=white!95, opacity=0.75] (0,0.5,0) -- (6.5,0.5,0) -- (6.5,0.5,5.8) -- (0,0.5,5.8) -- cycle; 
  \end{scope}
\end{scope}

\begin{scope}[shift={(0,-1.5)}]
 \begin{scope}[rotate around x=0, rotate around y=-35]
        \foreach \x in {0,1,2}
        \foreach \y in {0}
        \foreach \z in {2,1,0}
        {
            \begin{scope}[shift={(\x,\y,-\z)}]

                \draw[cube, fill=orange!55] (0,0.5,0) -- (1,0.5,0) -- (1,0.5,1) -- (0,0.5,1) -- cycle; 
                \draw[cube, fill=orange!55!black] (0,0,1) -- (1,0,1) -- (1,0.5,1) -- (0,0.5,1) -- cycle; 
                \draw[cube, fill=orange!30] (1,0,0) -- (1,0.5,0) -- (1,0.5,1) -- (1,0,1) -- cycle; 
                
            \end{scope}
        }
    \end{scope}

\end{scope}

\begin{scope}[shift={(0,-1)}]
 \begin{scope}[rotate around x=0, rotate around y=-35]
        \foreach \x in {0,1}
        \foreach \y in {0}
        \foreach \z in {1,0}
        {
            \begin{scope}[shift={(\x,\y,-\z)}]
                
                 \draw[cube] (0,0,0) -- (1,0,0) -- (1,0,1) -- (0,0,1) -- cycle; 

                \draw[cube] (0,0,0) -- (0,0,1) -- (0,1,1) -- (0,1,0) -- cycle; 
               
                \draw[cube] (0,0,0) -- (1,0,0) -- (1,1,0) -- (0,1,0) -- cycle; 
                
                 \draw[cube, fill=orange!55] (0,1,0) -- (1,1,0) -- (1,1,1) -- (0,1,1) -- cycle; 
              \draw[cube, fill=orange!55!black] (0,0,1) -- (1,0,1) -- (1,1,1) -- (0,1,1) -- cycle; 
                \draw[cube, fill=orange!30] (1,0,0) -- (1,1,0) -- (1,1,1) -- (1,0,1) -- cycle; 
                
            \end{scope}
        }
    \end{scope}

\end{scope}

 \begin{scope}[rotate around x=0, rotate around y=-35]
       
        {

                 \draw[cube] (0,0,0) -- (1,0,0) -- (1,0,1) -- (0,0,1) -- cycle; 

                \draw[cube] (0,0,0) -- (0,0,1) -- (0,1,1) -- (0,1,0) -- cycle; 
               
                \draw[cube] (0,0,0) -- (1,0,0) -- (1,1,0) -- (0,1,0) -- cycle; 
                
                 \draw[cube, fill=orange!55] (0,1,0) -- (1,1,0) -- (1,1,1) -- (0,1,1) -- cycle; 
                \draw[cube, fill=orange!55!black] (0,0,1) -- (1,0,1) -- (1,1,1) -- (0,1,1) -- cycle; 
                \draw[cube, fill=orange!30] (1,0,0) -- (1,1,0) -- (1,1,1) -- (1,0,1) -- cycle; 

        }
    \end{scope}
\end{scope} 

\draw[-, line width=0.5mm] (-0.3,1.5) -- (-0.3,-1.2);
\draw[-, line width=0.5mm, opacity=0.3] (-0.3,-1.2) -- (-0.3,-1.45);
\draw[->,>=stealth, line width=0.5mm] (-0.3,-1.45) -- (-0.3,-2.5);

\draw[->, decorate, decoration={snake, amplitude=1.2mm, segment length=5mm}] (7,0) -- (10,0) node[midway, above, yshift=5pt, align=center] {\small X-rays \\ of the \\ pyramid};


\draw [fill=white!40, thick] (10.75,2.75) rectangle (24,-2.25);

\draw [fill=orange!55, line width=0.3mm] (11.75,1.75) rectangle (12.5,1);


\begin{scope}[shift={(13.5,0.25)}]
\foreach \x in {0, 0.75} {
        \foreach \y in {0, 0.75} {
         
            \draw [fill=orange!55, line width=0.3mm] (\x,\y) rectangle (\x+0.75,\y+0.75);
        }
    }
    \end{scope}


\begin{scope}[shift={(16.25,-0.5)}]
\foreach \x in {0, 0.75, 1.5} {
        \foreach \y in {0, 0.75, 1.5} {
         
            \draw [fill=orange!55, line width=0.3mm] (\x,\y) rectangle (\x+0.75,\y+0.75);
        }
    }
    \end{scope}


\begin{scope}[shift={(19.75,-1.25)}]
\foreach \x in {0, 0.75, 1.5, 2.25} {
        \foreach \y in {0, 0.75, 1.5, 2.25} {
         
            \draw [fill=orange!55, line width=0.3mm] (\x,\y) rectangle (\x+0.75,\y+0.75);
        }
    }
    \end{scope}
   
\end{tikzpicture}

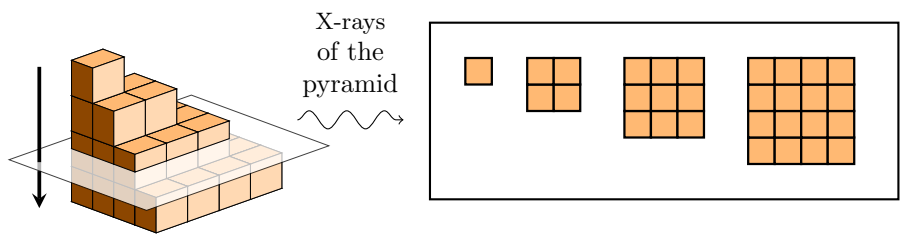
\captionof{figure}{Main sections of the pyramid.}\label{MainSections}
\vskip .5cm
\hskip 15pt The advantage of the sections method is that there is no unique way to take an X-ray of a pyramid. In fact, we may also take profile X-rays and observe that these sections contain as many squares as successively truncated triangular numbers.
\vskip .5cm
\begin{tikzpicture}[scale=0.47]
\draw[line width=0pt] (0,0) -- (0,0);
\begin{scope}[shift={(1.2,0)}]
    \tikzset{cube/.style={very thin,draw=black, opacity=1}}

\begin{scope}[shift={(1.5,1)}]

\begin{scope}[shift={(4.3,-1.9)}]
\begin{scope}[rotate around x=0, rotate around y=-35]
  \draw[cube, fill=white!95, opacity=0.75] (0.5,0,0) -- (0.5,6.5-1,0) -- (0.5,6.5-1,5.8+1) -- (0.5,0,5.8+1) -- cycle; 
  \end{scope}
\end{scope}

\begin{scope}[shift={(0,-3)}]
 \begin{scope}[rotate around x=0, rotate around y=-35]
        \foreach \x in {0,1,2,3}
        \foreach \y in {0}
        \foreach \z in {3,2,1,0}
        {
            \begin{scope}[shift={(\x,\y,-\z)}]
                
                 \draw[cube] (0,0,0) -- (1,0,0) -- (1,0,1) -- (0,0,1) -- cycle; 

                \draw[cube] (0,0,0) -- (0,0,1) -- (0,1,1) -- (0,1,0) -- cycle; 
               
                \draw[cube] (0,0,0) -- (1,0,0) -- (1,1,0) -- (0,1,0) -- cycle; 
                
                 \draw[cube, fill=orange!55] (0,1,0) -- (1,1,0) -- (1,1,1) -- (0,1,1) -- cycle; 
                \draw[cube, fill=orange!55!black] (0,0,1) -- (1,0,1) -- (1,1,1) -- (0,1,1) -- cycle; 
                \draw[cube, fill=orange!30] (1,0,0) -- (1,1,0) -- (1,1,1) -- (1,0,1) -- cycle; 
                
            \end{scope}
        }
    \end{scope}

\end{scope}

\begin{scope}[shift={(0,-2)}]
 \begin{scope}[rotate around x=0, rotate around y=-35]
        \foreach \x in {0,1,2}
        \foreach \y in {0}
        \foreach \z in {2,1,0}
        {
            \begin{scope}[shift={(\x,\y,-\z)}]

                \draw[cube, fill=orange!55] (0,1,0) -- (1,1,0) -- (1,1,1) -- (0,1,1) -- cycle; 
                \draw[cube, fill=orange!55!black] (0,0,1) -- (1,0,1) -- (1,1,1) -- (0,1,1) -- cycle; 
                \draw[cube, fill=orange!30] (1,0,0) -- (1,1,0) -- (1,1,1) -- (1,0,1) -- cycle; 
                
            \end{scope}
        }
    \end{scope}

\end{scope}

\begin{scope}[shift={(0,-1)}]
 \begin{scope}[rotate around x=0, rotate around y=-35]
        \foreach \x in {0,1}
        \foreach \y in {0}
        \foreach \z in {1,0}
        {
            \begin{scope}[shift={(\x,\y,-\z)}]
                
                 \draw[cube] (0,0,0) -- (1,0,0) -- (1,0,1) -- (0,0,1) -- cycle; 

                \draw[cube] (0,0,0) -- (0,0,1) -- (0,1,1) -- (0,1,0) -- cycle; 
               
                \draw[cube] (0,0,0) -- (1,0,0) -- (1,1,0) -- (0,1,0) -- cycle; 
                
                 \draw[cube, fill=orange!55] (0,1,0) -- (1,1,0) -- (1,1,1) -- (0,1,1) -- cycle; 
              \draw[cube, fill=orange!55!black] (0,0,1) -- (1,0,1) -- (1,1,1) -- (0,1,1) -- cycle; 
                \draw[cube, fill=orange!30] (1,0,0) -- (1,1,0) -- (1,1,1) -- (1,0,1) -- cycle; 
                
            \end{scope}
        }
    \end{scope}

\end{scope}

 \begin{scope}[rotate around x=0, rotate around y=-35]
       
        {

                 \draw[cube] (0,0,0) -- (1,0,0) -- (1,0,1) -- (0,0,1) -- cycle; 

                \draw[cube] (0,0,0) -- (0,0,1) -- (0,1,1) -- (0,1,0) -- cycle; 
               
                \draw[cube] (0,0,0) -- (1,0,0) -- (1,1,0) -- (0,1,0) -- cycle; 
                
                 \draw[cube, fill=orange!55] (0,1,0) -- (1,1,0) -- (1,1,1) -- (0,1,1) -- cycle; 
                \draw[cube, fill=orange!55!black] (0,0,1) -- (1,0,1) -- (1,1,1) -- (0,1,1) -- cycle; 
                \draw[cube, fill=orange!30] (1,0,0) -- (1,1,0) -- (1,1,1) -- (1,0,1) -- cycle; 

        }
    \end{scope}
\end{scope}

\draw[->, decorate, decoration={snake, amplitude=1.2mm, segment length=5mm}] (7,0) -- (10,0) node[midway, above, yshift=5pt, align=center] {\small X-rays \\ of the \\ pyramid};


\draw [fill=white!40, thick] (10.75,2.75) rectangle (25.25,-2.25);

\begin{scope}[shift={(11.25,-1.25)}]
\foreach \x in {0, 0.75, 1.5, 2.25} {
            \draw [fill=orange!55, line width=0.3mm] (\x,0) rectangle (\x+0.75, 0.75);
    }
    
    \foreach \x in {0, 0.75, 1.5} { 
            \draw [fill=orange!55, line width=0.3mm] (\x,0.75) rectangle (\x+0.75, 1.5);
        }
        
         \foreach \x in {0, 0.75} { 
            \draw [fill=orange!55, line width=0.3mm] (\x,1.5) rectangle (\x+0.75, 2.25);
        }
    
    \draw [fill=orange!55, line width=0.3mm] (0,2.25) rectangle (0.75, 3);
\end{scope}


\begin{scope}[shift={(14.75,-1.25)}]
\foreach \x in {0, 0.75, 1.5, 2.25} {
            \draw [fill=orange!55, line width=0.3mm] (\x,0) rectangle (\x+0.75, 0.75);
    }
    
    \foreach \x in {0, 0.75, 1.5} { 
            \draw [fill=orange!55, line width=0.3mm] (\x,0.75) rectangle (\x+0.75, 1.5);
        }
        
         \foreach \x in {0, 0.75} { 
            \draw [fill=orange!55, line width=0.3mm] (\x,1.5) rectangle (\x+0.75, 2.25);
        }

    \end{scope}


\begin{scope}[shift={(18.25,-1.25)}]
\foreach \x in {0, 0.75, 1.5, 2.25} {
            \draw [fill=orange!55, line width=0.3mm] (\x,0) rectangle (\x+0.75, 0.75);
    }
    
    \foreach \x in {0, 0.75, 1.5} { 
            \draw [fill=orange!55, line width=0.3mm] (\x,0.75) rectangle (\x+0.75, 1.5);
        }

    \end{scope}


\begin{scope}[shift={(21.75,-1.25)}]
\foreach \x in {0, 0.75, 1.5, 2.25} {
            \draw [fill=orange!55, line width=0.3mm] (\x,0) rectangle (\x+0.75, 0.75);
    }

    \end{scope}
    
\begin{scope}[shift={(1.2,-2.52)}]
\begin{scope}[rotate around x=0, rotate around y=-35]
  \draw[cube, fill=white!95, opacity=0.75,  draw=none] (0.5,0,0) -- (0.5,6.5-1.03,0) -- (0.5,6.5-1.03,5.8-4.18) -- (0.5,0,5.8-4.18) -- cycle; 
  \end{scope}
\end{scope}

\begin{scope}[shift={(5.465,1.97)}]
\begin{scope}[rotate around x=0, rotate around y=-35, draw=none]
  \draw[cube, fill=white!95, opacity=0.75,  draw=none ] (0.5,0,0) -- (0.5,6.5-4,0) -- (0.5,6.5-4,5.8-1) -- (0.5,0,5.8-1) -- cycle; 
  \end{scope}
\end{scope}

\begin{scope}[shift={(5.78,1.1)}]
\begin{scope}[rotate around x=0, rotate around y=-35]
  \draw[cube, fill=white!95, opacity=0.75, draw=none ] (0.5,0,0) -- (0.5,6.5-5.5,0) -- (0.5,6.5-5.5,5.8-3) -- (0.5,0,5.8-3) -- cycle; 
  \end{scope}
\end{scope}

\begin{scope}[shift={(5.76,-0.35)}]
\begin{scope}[rotate around x=0, rotate around y=-35]
  \draw[cube, fill=white!95, opacity=0.75, draw=none ] (0.5,0,0) -- (0.5,6.5-5.5,0) -- (0.5,6.5-5.5,5.8-4.7) -- (0.5,0,5.8-4.7) -- cycle; 
  \end{scope}
\end{scope}
   

\draw[-, line width=0.5mm] (-0.6,-2.355) -- (0.05,-2.58);
\draw[-, line width=0.5mm, opacity=0.3] (0.05,-2.58) -- (0.4,-2.7);
\draw[->,>=stealth, line width=0.5mm] (0.4,-2.7) -- (2.5,-3.455);    
   
  
  \draw[line width=0.2mm] (1.5,-2.65) -- (1.5,0.35);
  \draw[line width=0.2mm] (1.5,0.35) -- (3.28,1);
  \draw[line width=0.2mm] (3.6,0.1) -- (4.15,0.31);
  \draw[line width=0.2mm] (3.6+1.47,0.1-0.95) -- (4.191+1.47,0.31-0.94);
  \end{scope}
\end{tikzpicture}

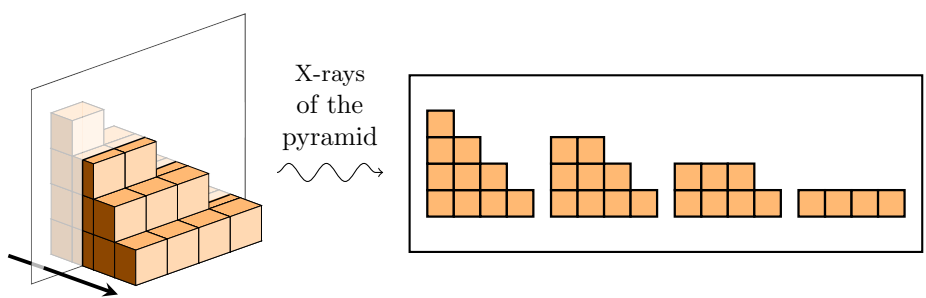
\captionof{figure}{Secondary sections of the pyramid.}\label{SecondarySections}
\vskip .3cm
\restoreparindent  These profile sections will be referred to as the {\em Secondary sections}.

\restoreparindent This observation is meaningful because, since the number of squares must coincide in both types of sections—and therefore with the number of cubes in the original pyramid—we obtain the following identity:
\begin {equation}\label{Algebraic formula}
\begin{aligned}
1^2 +2^2+ 3^2+ \cdots + n^2 &=& 1+2+3+\cdots +n \\
&& \ \ +2+3+\cdots +n \\
&& \ \ \ \ +3+\cdots +n \\
&&\hfill \vdots \ \\
&&\hfill +n
\end{aligned}
\end{equation}

\restoreparindent The visual proof of this identity is precisely given by Figures~\ref{MainSections} and~\ref{SecondarySections}. However, this identity also admits a very simple algebraic proof, and this proof generalises naturally to any dimension. Namely,

\begin{lemma} \label{Algebraic Lemma}
For $p\geq 0$ the following identity holds:
\begin {equation}\label{Algebraic formula for p}
\begin{aligned}
1^{p+1} +2^{p+1}+ 3^{p+1}+ \cdots + n^{p+1} &=& 1^p+2^p+3^p+\cdots +n^p \\
&& \ \ +2^p+3^p+\cdots +n^p \\
&& \ \ \ \ +3^p+\cdots +n^p \\
&&\hfill \vdots \ \ \ \\
&&\hfill +n^p
\end{aligned}
\end{equation}
\end{lemma}
\begin{proof}
The sum on the right-hand side can be computed in an alternative way. Instead of summing the rows sequentially, we may sum the columns and obtain the same result.
\[
\begin{tikzpicture}
    \matrix (m) [matrix of math nodes,
        row sep=.5em, column sep=.1em,
        nodes={anchor=center}] {
        1^p & + & 2^p &+ &3^p & +& \cdots & + & n^p \\
         & + & 2^p &+ &3^p & +& \cdots & + & n^p \\
         & & &+ &3^p & +& \cdots & + & n^p \\
         & &  &  & & &  &  &\vdots \\
         \ \  &  \ \  &  \ \  &  \ \  &  \ \  &  \ \  & \ \  & + & n^p \\
       1\cdot 1^p  & + & 2\cdot 2^p & + & 3\cdot 3^p & + & \cdots & + & n\cdot n^p \\
    };
    \draw[dashed] (m-5-1.south west) -- (m-5-9.south east);
\end{tikzpicture}
\]
\vskip .3cm
\captionof{figure}{Result obtained by summing by columns}
\restoreparindent And since $1\cdot 1^p + 2\cdot 2^p+3\cdot 3^p+\cdots +n\cdot n^p = 1^{p+1}+2^{p+1}+3^{p+1}+\cdots +n^{p+1}$, the lemma follows immediately.
\end{proof}
\begin{remark}
Although this article focuses on visual proofs, we have included this algebraic lemma for several reasons. First, the lemma arises naturally from the technique we use to visualize higher dimensions through sections: the case $p=1$ is precisely the equality between the sum of squares and the sum of truncated triangular numbers. Secondly, the proof of Lemma~\ref{Algebraic Lemma} is immediate once one recognizes the spatial arrangement of its algebraic elements—namely, summing by columns instead of by rows. This naturally raises the question of whether, in the end, algebra is nothing more than a symbolic way of carrying out visual arguments.

\restoreparindent
As mentioned above, identity~\eqref{Algebraic formula for p}, in the case $p=1$, expresses the equality between the sum of squares and the sum of truncated triangular numbers. For $p=2$, the identity becomes the equality between the sum of cubes and the sum of truncated pyramidal numbers. The sums of cubes correspond to the main sections of pyramids of $4$-dimensional hypercubes, whereas the truncated pyramidal numbers arise from their secondary sections. This correspondence will be used in Section~\ref{Nicomachus 4D} to establish a visual proof of Nicomachus’ Theorem in 4~dimensions.
\end{remark}

\subsection{Sum of squares in 2D} \label{Sumin2D}

Now that we have a tool for visualising higher-dimensional objects through their lower-dimensional sections, we revisit the visual proof of identity~\eqref{SumofSquaresDenominator3}. In three dimensions, this proof uses three square-based pyramids, but we will now reinterpret it from the perspective of a Flatland being by examining its planar sections. This first application of 2D sections as a 3D visualisation technique will serve as preparation for the next section, where we will proceed analogously with 3D sections of 4D objects.

To translate the three–dimensional proof from Figures~\ref{Threepyramids} and~\ref{DIY} into two dimensions, we take X–rays of the block in Figure~\ref{Threepyramids}, obtaining a kind of 2D puzzle.

Since three pyramids participate in the construction of the block, the choice of direction for the X–rays determines which pyramids will appear in the 2D puzzle through their secondary sections, and which one through its main sections.

A relevant detail is that, if in the 2D puzzle of Figure \ref{2D of 3D} we omit the pyramid for which the chosen direction produces the main section (the yellow one), the first section of the two remaining pyramids is precisely the visual proof, in one lower dimension, of the formula for the sum of powers. More specifically, in the upper row of X-rays in Figure \ref{2D of 3D}, if we remove the yellow square (the main section), what remains is exactly Figure \ref{Gauss}, consisting of two fitted triangular numbers.

In the remaining sections, if we omit the yellow squares, we obtain truncated triangles that leave an almost square gap, into which the sections of the yellow pyramid fit.
\vskip .5cm

\vskip .3cm
\captionof{figure}{2D X-ray reconstruction of the 3D visual proof}\label{2D of 3D}

\restoreparindent Next, we need to carry out a small geometric “DIY step’’ on our 2D puzzle to turn all the pieces into identical rectangles. To do this, we cut the first row in half and exchange the two halves between the first and last sections, the second and second-to-last sections, and so on.

\restoreparindent The result, shown in Figure~\ref{2D 3D DIY}, is a set of rectangles of base 
$n+1$ and height 
$n+\frac{1}{2}$, matching the length and height of the block obtained after the DIY step in Figure~\ref{DIY}.
Since the X-rays were taken along the remaining dimension, the width of the 3D block corresponds simply to the total number of sections—that is, to the number of rectangles in the 2D puzzle.

Thus, the number of cubes in the three square-based pyramids of height $n$, that is, $3(1^2 + 2^2+ 3^2+\cdots +n^2)$, is precisely 
$n(n+1)(n+\frac{1}{2})$, from which the desired formula follows.

\vskip .5cm
\hspace*{-.6cm}

\vskip .3cm
\captionof{figure}{DIY in the 2D puzzle of 3D visual proof}
\label{2D 3D DIY}

Having sharpened our ability to prove 3-dimensional theorems through 2-dimensional sections, we are now prepared to take the next step: to explore the fourth dimension and prove Nicomachus’ theorem in 4D by examining its 3D sections.

\begin{remark}[Algebra vs. Visual Arguments I]
It is worth pausing at this moment to evaluate the difference between the advantages and disadvantages of an algebraic proof versus a visual argument. In the sections of the block shown in Figure $\ref{2D of 3D}$, we have seen that two truncated triangles leave a gap that is 'almost square'.

\restoreparindent Algebraically, this statement could be written as:
\begin{equation}\label{identity almost square}
2\Bigl( m + (m+1)+\cdots + n\Bigr) + m^2 =(n+1)^2 -(n+1-m), 
\end{equation}
where $1\leq m \leq n$.
\vskip .2cm
\begin{tikzpicture}[scale=0.65]
 \draw[line width=0pt] (0,0) -- (0,0);  
\begin{scope}[shift={(-10,11)}] 
\begin{scope}[shift={(17.5,-14.25)}]
\foreach \x in {0, 0.75, 1.5, 2.25} {
            \draw [fill=blue!55, line width=0.3mm] (\x,0) rectangle (\x+0.75, 0.75);

    }
    
    \foreach \x in {0, 0.75, 1.5} { 
            \draw [fill=blue!55, line width=0.3mm] (\x,0.75) rectangle (\x+0.75, 1.5);

        }
        
         \foreach \x in {0, 0.75} { 
            
            \draw [fill=red!55, line width=0.3mm] (\x+2.25,0.75) rectangle (\x+3, 1.5);
            
              \draw [fill=red!55, line width=0.3mm] (\x +2.25,1.5) rectangle (\x+3, 2.25);
            \draw [fill=red!55, line width=0.3mm] (\x+2.25,2.25) rectangle (\x+3, 3);
        }

    \draw [fill=red!55, line width=0.3mm] (3,0) rectangle (3.75, 0.75);

     \foreach \x in {0, 0.75, 1.5} { 
     \foreach \y in {0, -0.75, -1.5}{
      \draw [fill=orange!55, line width=0.3mm] (\x,\y+3) rectangle (\x+0.75,\y+3.75);
      }
}

\end{scope}

\end{scope}

  \draw[dashed, dash pattern=on 3pt off 2pt] (9.8,0.5) -- (11.25,0.5);
    \draw[dashed, dash pattern=on 3pt off 2pt] (11.25,0.5) -- (11.25,-0.2);
\begin{scope}[shift={(0,1)}] 
\draw[<->]  (11.5, -4.2) -- (11.5,-0.5)  node[midway, right] {\hskip -.25cm $\ \ n+1$};
\draw[<->]  (7.5, -4.5) -- (11.3,-4.5)  node[midway, below] {\hskip -.25cm $\ \ n+1$};
\draw[<->]  (7.5, -0.2) -- (9.7,-0.2)  node[midway, above] {$m$};
\end{scope}
\end{tikzpicture}

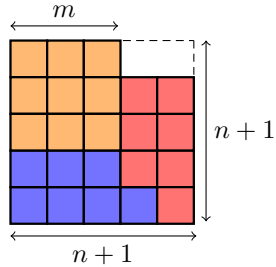
\captionof{figure}{Two truncated triangles leave an almost square gap}
\label{almost square}
\par
\hskip 15pt One can prove this identity using the expression for triangular numbers obtained in Figure $\ref{Gauss}$, that is:
\begin{align*}
 2\Bigl( m + (m+1)+\cdots + n\Bigr) + m^2&=2\Bigl( t_n-t_{m-1}\Bigr) + m^2 \\
 &=2\Bigl(\frac{n(n+1)}{2}-\frac{(m-1)m}{2}\Bigr) + m^2\\
 &=n^2+n+m =(n+1)^2-n-1+m.
\end{align*}
\par

\restoreparindent With this example, we see that the image depicted in Figure $\ref{almost square}$ is identified both with identity ($\ref{identity almost square}$) and with the proof we have seen of it. While the identity and these proofs are not without a certain complexity, the image is clear and straightforward.

\restoreparindent In the following chapters, where we tackle higher dimensions, we will encounter more situations in which visual arguments are very clear and simple, while their algebraic ``translation'' is complex and full of cumbersome notation.
\end{remark}

\section{Nicomachus' 4D Theorem for sum of cubes}
\label{Nicomachus 4D}
\restoreparindent This chapter can be considered the central chapter of this work, as it constitutes the first example in which we apply our sectioning technique to visualize higher dimensions beyond our ability to visualize them directly.

\restoreparindent It is true that the fourth dimension can, with difficulty, be visualized through representations such as the tesseract, as shown in \cite{Sasho}; however, this approach becomes unfeasible in the next step, the fifth dimension.

Thus, this chapter not only presents a clear and concise 4D proof of Nicomachus' Theorem through sections, but also paves the way to visualize the fifth dimension without additional difficulty.

To do so, once Nicomachus' Theorem is established in Section \ref{Nicomachus 4D in 3D}, we apply our sectioning technique once more to the same three-dimensional slices, reducing the construction to a 2D puzzle in the style of Figures \ref{2D of 3D} and \ref{2D 3D DIY} in Section \ref{Sumin2D}.
\vskip 2cm

\subsection{Nicomachus' 4D Theorem through 3D sections}  \label{Nicomachus 4D in 3D}

Our goal in this section is to provide a 4D visual proof of formula (\ref{formula4D}) for the sum of cubes. We begin, therefore, by representing this sum. At this point, one might naturally wonder why all the cubes have been assigned the same color here, in contrast with Figure \ref{ClassicalNicomachus}.
\vskip .2cm
{\Large \ \ \ \ \ \ \ $1^3$\ \ $+$\ \ $2^3$\ \ \ \ $+$\ \ \ \ $3^3$\ \ \ $+$\ \ \ $\cdots$\ \ \ $+$\ \ \ \ \ $n^3$}
\vskip .2cm

\begin{tikzpicture}[scale=0.45]
    \tikzset{cube/.style={very thin,draw=black, opacity=1}}
\begin{scope}[shift={(4,0)}]    
    \begin{scope}[rotate around x=0, rotate around y=-35]
       
        {
                \draw[line width=0pt] (-5,0) -- (-5,0);
                 \draw[cube, fill=green!55] (0,1,0) -- (1,1,0) -- (1,1,1) -- (0,1,1) -- cycle; 
                \draw[cube, fill=green!55!black] (0,0,1) -- (1,0,1) -- (1,1,1) -- (0,1,1) -- cycle; 
                \draw[cube, fill=green!30] (1,0,0) -- (1,1,0) -- (1,1,1) -- (1,0,1) -- cycle; 

        }
    \end{scope}
\end{scope}
\begin{scope}[shift={(7,0)}]
 \begin{scope}[rotate around x=0, rotate around y=-35]
        \foreach \x in {0,1}
        \foreach \y in {0,1}
        \foreach \z in {1,0}
        {
            \begin{scope}[shift={(\x,\y,-\z)}]
                
                 \draw[cube, fill=green!55] (0,1,0) -- (1,1,0) -- (1,1,1) -- (0,1,1) -- cycle; 
                \draw[cube, fill=green!55!black] (0,0,1) -- (1,0,1) -- (1,1,1) -- (0,1,1) -- cycle; 
                \draw[cube, fill=green!30] (1,0,0) -- (1,1,0) -- (1,1,1) -- (1,0,1) -- cycle; 
                
            \end{scope}
        }
    \end{scope}

\end{scope}

\begin{scope}[shift={(12,0)}]
 \begin{scope}[rotate around x=0, rotate around y=-35]
        \foreach \x in {0,1,2}
        \foreach \y in {0,1,2}
        \foreach \z in {2,1,0}
        {
            \begin{scope}[shift={(\x,\y,-\z)}]
                
                 \draw[cube, fill=green!55] (0,1,0) -- (1,1,0) -- (1,1,1) -- (0,1,1) -- cycle; 
                \draw[cube, fill=green!55!black] (0,0,1) -- (1,0,1) -- (1,1,1) -- (0,1,1) -- cycle; 
                \draw[cube, fill=green!30] (1,0,0) -- (1,1,0) -- (1,1,1) -- (1,0,1) -- cycle; 
                
            \end{scope}
        }
    \end{scope}

\end{scope}

\begin{scope}[shift={(19,0)}]
 \begin{scope}[rotate around x=0, rotate around y=-35]
        \foreach \x in {0,1,2,3}
        \foreach \y in {0,1,2,3}
        \foreach \z in {3,2,1,0}
        {
            \begin{scope}[shift={(\x,\y,-\z)}]
                
                 \draw[cube, fill=green!55] (0,1,0) -- (1,1,0) -- (1,1,1) -- (0,1,1) -- cycle; 
                \draw[cube, fill=green!55!black] (0,0,1) -- (1,0,1) -- (1,1,1) -- (0,1,1) -- cycle; 
                \draw[cube, fill=green!30] (1,0,0) -- (1,1,0) -- (1,1,1) -- (1,0,1) -- cycle; 
                
            \end{scope}
        }
    \end{scope}

\end{scope}

\draw[<->] (24.5,0) -- (24.5,4); \node[right] at (24.5, 2) {$n$};
\end{tikzpicture}
\vskip .3cm

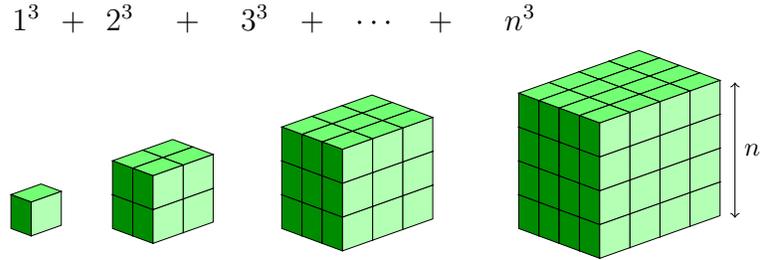
\captionof{figure}{Main sections of a 4D pyramid}
\label{Main sections 4D pyramid}

\restoreparindent A reader familiar with the strategy developed in the previous sections (or who has examined the caption of Figure \ref{Main sections 4D pyramid}) will recognize this configuration as the collection of main three-dimensional sections of a four-dimensional pyramid.

\restoreparindent In Figure \ref{ClassicalNicomachus}, the cubes were independent units that we decomposed into smaller parts to form a square tapestry. In the present setting, however, the 
$n$ green cubes constitute a single object: a pyramid in four dimensions. This is the reason for depicting them in a uniform color.

\restoreparindent As we observed in Section \ref{LowerSections}, there is more than one way to obtain the sections of a pyramid. Setting $p=2$ in Lemma \ref{Algebraic Lemma}, the left-hand side of the identity corresponds to the main sections considered above, whereas the right-hand side describes the secondary sections, which in this case take the form of truncated pyramids.

\begin {equation}\label{Algebraic formula for 3}
\begin{aligned}
1^{3} +2^{3}+ 3^{3}+ \cdots + n^{3} &=& 1^2+2^2+3^2+\cdots +n^2 \\
&& \ \ +2^2+3^2+\cdots +n^2 \\
&& \ \ \ \ +3^2+\cdots +n^2 \\
&&\hfill \vdots \ \ \ \\
&&\hfill +n^2
\end{aligned}
\end{equation}

\begin{tikzpicture}[scale=0.45]
    \tikzset{cube/.style={very thin,draw=black, opacity=1}}

\begin{scope}[shift={(0,-3)}]
 \begin{scope}[rotate around x=0, rotate around y=-35]
        \foreach \x in {0,1,2,3}
        \foreach \y in {0}
        \foreach \z in {3,2,1,0}
        {
            \begin{scope}[shift={(\x,\y,-\z)}]
                   
                 \draw[cube, fill=orange!55] (0,1,0) -- (1,1,0) -- (1,1,1) -- (0,1,1) -- cycle; 
                \draw[cube, fill=orange!55!black] (0,0,1) -- (1,0,1) -- (1,1,1) -- (0,1,1) -- cycle; 
                \draw[cube, fill=orange!30] (1,0,0) -- (1,1,0) -- (1,1,1) -- (1,0,1) -- cycle; 
                
            \end{scope}
        }
    \end{scope}

\end{scope}

\begin{scope}[shift={(0,-2)}]
 \begin{scope}[rotate around x=0, rotate around y=-35]
        \foreach \x in {0,1,2}
        \foreach \y in {0}
        \foreach \z in {2,1,0}
        {
            \begin{scope}[shift={(\x,\y,-\z)}]
                            
                 \draw[cube, fill=orange!55] (0,1,0) -- (1,1,0) -- (1,1,1) -- (0,1,1) -- cycle; 
                \draw[cube, fill=orange!55!black] (0,0,1) -- (1,0,1) -- (1,1,1) -- (0,1,1) -- cycle; 
                \draw[cube, fill=orange!30] (1,0,0) -- (1,1,0) -- (1,1,1) -- (1,0,1) -- cycle; 
                
            \end{scope}
        }
    \end{scope}

\end{scope}

\begin{scope}[shift={(0,-1)}]
 \begin{scope}[rotate around x=0, rotate around y=-35]
        \foreach \x in {0,1}
        \foreach \y in {0}
        \foreach \z in {1,0}
        {
            \begin{scope}[shift={(\x,\y,-\z)}]

                 \draw[cube, fill=orange!55] (0,1,0) -- (1,1,0) -- (1,1,1) -- (0,1,1) -- cycle; 
              \draw[cube, fill=orange!55!black] (0,0,1) -- (1,0,1) -- (1,1,1) -- (0,1,1) -- cycle; 
                \draw[cube, fill=orange!30] (1,0,0) -- (1,1,0) -- (1,1,1) -- (1,0,1) -- cycle; 
                
            \end{scope}
        }
    \end{scope}

\end{scope}

 \begin{scope}[rotate around x=0, rotate around y=-35]
       
        {

                 \draw[cube, fill=orange!55] (0,1,0) -- (1,1,0) -- (1,1,1) -- (0,1,1) -- cycle; 
                \draw[cube, fill=orange!55!black] (0,0,1) -- (1,0,1) -- (1,1,1) -- (0,1,1) -- cycle; 
                \draw[cube, fill=orange!30] (1,0,0) -- (1,1,0) -- (1,1,1) -- (1,0,1) -- cycle; 

        }
    \end{scope}

\draw[<->] (-1.25,0.7) -- (-1.25,-3.3); \node[right] at (-2.25, -1.5) {$n$};


\begin{scope}[shift={(5.5,2)}]

\begin{scope}[shift={(0,-3)}]
 \begin{scope}[rotate around x=0, rotate around y=-35]
        \foreach \x in {0,1,2,3}
        \foreach \y in {0}
        \foreach \z in {3,2,1,0}
        {
            \begin{scope}[shift={(\x,\y,-\z)}]
                   
                 \draw[cube, fill=orange!55] (0,1,0) -- (1,1,0) -- (1,1,1) -- (0,1,1) -- cycle; 
                \draw[cube, fill=orange!55!black] (0,0,1) -- (1,0,1) -- (1,1,1) -- (0,1,1) -- cycle; 
                \draw[cube, fill=orange!30] (1,0,0) -- (1,1,0) -- (1,1,1) -- (1,0,1) -- cycle; 
                
            \end{scope}
        }
    \end{scope}

\end{scope}

\begin{scope}[shift={(0,-2)}]
 \begin{scope}[rotate around x=0, rotate around y=-35]
        \foreach \x in {0,1,2}
        \foreach \y in {0}
        \foreach \z in {2,1,0}
        {
            \begin{scope}[shift={(\x,\y,-\z)}]
                            
                 \draw[cube, fill=orange!55] (0,1,0) -- (1,1,0) -- (1,1,1) -- (0,1,1) -- cycle; 
                \draw[cube, fill=orange!55!black] (0,0,1) -- (1,0,1) -- (1,1,1) -- (0,1,1) -- cycle; 
                \draw[cube, fill=orange!30] (1,0,0) -- (1,1,0) -- (1,1,1) -- (1,0,1) -- cycle; 
                
            \end{scope}
        }
    \end{scope}

\end{scope}

\begin{scope}[shift={(0,-1)}]
 \begin{scope}[rotate around x=0, rotate around y=-35]
        \foreach \x in {0,1}
        \foreach \y in {0}
        \foreach \z in {1,0}
        {
            \begin{scope}[shift={(\x,\y,-\z)}]

                 \draw[cube, fill=orange!55] (0,1,0) -- (1,1,0) -- (1,1,1) -- (0,1,1) -- cycle; 
              \draw[cube, fill=orange!55!black] (0,0,1) -- (1,0,1) -- (1,1,1) -- (0,1,1) -- cycle; 
                \draw[cube, fill=orange!30] (1,0,0) -- (1,1,0) -- (1,1,1) -- (1,0,1) -- cycle; 
                
            \end{scope}
        }
    \end{scope}

\end{scope}

 \end{scope}


\begin{scope}[shift={(11,4)}]

\begin{scope}[shift={(0,-3)}]
 \begin{scope}[rotate around x=0, rotate around y=-35]
        \foreach \x in {0,1,2,3}
        \foreach \y in {0}
        \foreach \z in {3,2,1,0}
        {
            \begin{scope}[shift={(\x,\y,-\z)}]
                   
                 \draw[cube, fill=orange!55] (0,1,0) -- (1,1,0) -- (1,1,1) -- (0,1,1) -- cycle; 
                \draw[cube, fill=orange!55!black] (0,0,1) -- (1,0,1) -- (1,1,1) -- (0,1,1) -- cycle; 
                \draw[cube, fill=orange!30] (1,0,0) -- (1,1,0) -- (1,1,1) -- (1,0,1) -- cycle; 
                
            \end{scope}
        }
    \end{scope}

\end{scope}

\begin{scope}[shift={(0,-2)}]
 \begin{scope}[rotate around x=0, rotate around y=-35]
        \foreach \x in {0,1,2}
        \foreach \y in {0}
        \foreach \z in {2,1,0}
        {
            \begin{scope}[shift={(\x,\y,-\z)}]
                            
                 \draw[cube, fill=orange!55] (0,1,0) -- (1,1,0) -- (1,1,1) -- (0,1,1) -- cycle; 
                \draw[cube, fill=orange!55!black] (0,0,1) -- (1,0,1) -- (1,1,1) -- (0,1,1) -- cycle; 
                \draw[cube, fill=orange!30] (1,0,0) -- (1,1,0) -- (1,1,1) -- (1,0,1) -- cycle; 
                
            \end{scope}
        }
    \end{scope}

\end{scope}

 \end{scope}


\begin{scope}[shift={(16.5,6)}]

\begin{scope}[shift={(0,-3)}]
 \begin{scope}[rotate around x=0, rotate around y=-35]
        \foreach \x in {0,1,2,3}
        \foreach \y in {0}
        \foreach \z in {3,2,1,0}
        {
            \begin{scope}[shift={(\x,\y,-\z)}]
                   
                 \draw[cube, fill=orange!55] (0,1,0) -- (1,1,0) -- (1,1,1) -- (0,1,1) -- cycle; 
                \draw[cube, fill=orange!55!black] (0,0,1) -- (1,0,1) -- (1,1,1) -- (0,1,1) -- cycle; 
                \draw[cube, fill=orange!30] (1,0,0) -- (1,1,0) -- (1,1,1) -- (1,0,1) -- cycle; 
                
            \end{scope}
        }
    \end{scope}

\end{scope}

\end{scope}
\end{tikzpicture}
\captionof{figure}{Secondary sections of a 4D pyramid}
\label{Secondary sections 4D pyramid}

\restoreparindent If we examine the X-rays of the block formed by the three assembled three-dimensional pyramids (see Figure \ref{2D of 3D}), we observe that two of them give rise to secondary sections in the X-rays; that is, truncated triangles that fit together, where the first of these sections coincides with the visual proof in the lower dimension (see Figure \ref{Gauss}).

\restoreparindent The third pyramid, in contrast, fits into the ‘almost’ square gaps through its main sections, which in this case are squares.

Similarly, four hypercube pyramids can be assembled so that, when taking their three-dimensional X-rays, three of them appear as secondary sections—namely, truncated pyramids—whose first section coincides with the visual proof of the previous dimension, that is, Figure \ref{Threepyramids}.\par
\vspace{5mm}
\noindent

\vskip .3cm
\captionof{figure}{Secondary sections of three assembled 4D pyramids}
\label{3 pyramids 4D}

\restoreparindent In the ‘almost cubic’ gap left by the secondary sections of the three assembled 4D pyramids fits the fourth 4D pyramid, which appears precisely through its main sections—namely, cubes (see Figure \ref{Main sections 4D pyramid}).

Thus, the three-dimensional sections of the four assembled 4D pyramids produce $n$ imperfect blocks that once again require some DIY work. In this case, however, formula (\ref{formula4D}) tells us that the small cubes do not need to be cut, only rearranged.

Examining these imperfect blocks, we observe that each of them contains an excess region and a complementary deficit. More precisely, in the last three-dimensional section there is an 
$n\times n$ square slab protruding from the block
($4\times 4$ in Figure 6 when $n=4$); in the preceding section, the protruding slab has size 
$(n-1)\times (n-1)$; and this pattern continues down to the second section, with a $2\times 2$ slab, and the first section, where the excess consists of a single green square.\par
\vspace{5mm}
\noindent

\vskip .3cm
\captionof{figure}{3D sections of four assembled 4D pyramids}
\label{4 pyramids 4D}

\restoreparindent However, as shown in Figure \ref{DIY Removing}, once these slabs are separated from the blocks, the region that must be filled is exactly a triangular gap in the first section, which gradually truncates until it becomes a linear gap of length 
$n$ in the last section.

\restoreparindent As we know from Lemma \ref{Algebraic Lemma} (and, in particular, from identity (\ref{Algebraic formula})), the sum of squares equals the sum of truncated triangular numbers. Consequently, the protruding squares fill the corresponding gaps, producing 
$n$ identical blocks, as shown in Figure \ref{DIY Filling}.\par
\vspace{5mm}
\noindent

\vskip .3cm
\captionof{figure}{DIY: Removing the Excess Slabs from the Blocks}
\label{DIY Removing}

\restoreparindent The dimensions of the resulting identical blocks are: $n+1$ in length, $n$ in width, and $n+1$ in height. Moreover, there are exactly $n$ such blocks, one for each section we considered. Hence, the total number of unitary cubes is $n^2 (n+1)^2$.

 \restoreparindent Recall that we started with four pyramids of hypercubes whose three-dimensional sections contained a total of $4(1^3+ 2^3+3^3+ \cdots +n^3)$ unitary cubes. Consequently, we obtain the desired identity:
$$1^3+ 2^3+3^3+ \cdots +n^3=\frac{n^2 (n+1)^2}{4},$$
which completes the proof of the Nicomachus’ 4D Theorem.\par
\vspace{5mm}
\noindent

\vskip .3cm
\captionof{figure}{DIY: Filling the Gaps of the Blocks}
\label{DIY Filling}

\subsection{Nicomachus' 4D Theorem in 2D sections}

Although we are three-dimensional beings capable of visualizing in 3D, we must acknowledge that this dimension presents certain difficulties.

\restoreparindent Indeed, the visual proof in Figures~\ref{Threepyramids} and \ref{DIY} for the sum of squares, which takes place in 3D, is not without challenges. The pyramids and the block they form have a rear part that we do not see and must, in some way, imagine.

\restoreparindent This is why the same proof, viewed in its 2D sections in Figures~\ref{2D of 3D} and ~\ref{2D 3D DIY}, becomes much clearer. As three-dimensional beings, what we truly excel at is 2D visualization.

\restoreparindent This poses no problem, since the technique of sections allows us to translate 4D proofs into their 3D sections, each of which is itself a three-dimensional figure that can be studied through 2D sections.

\restoreparindent Therefore, we will devote this subchapter to translating the proof of Nicomachus’ 4D Theorem into a 2D visual proof, which will also serve as preparation for making the leap to 5D.\par
\vspace{5mm}
\noindent
\begin{tikzpicture}[scale=0.5]




\begin{scope}[shift={(0,0)}]

\begin{scope}[shift={(17.5,-4.25)}]
\foreach \x in {0, 0.75, 1.5, 2.25} {
            \draw [fill=blue!55, line width=0.3mm] (\x,0) rectangle (\x+0.75, 0.75);
            
            \draw [fill=red!55, line width=0.3mm] (\x+0.75,2.25) rectangle (\x+1.5, 3);
       
    }
    
    \foreach \x in {0, 0.75, 1.5} { 
            \draw [fill=blue!55, line width=0.3mm] (\x,0.75) rectangle (\x+0.75, 1.5);
            
            \draw [fill=red!55, line width=0.3mm] (\x +1.5,1.5) rectangle (\x+2.25, 2.25);
        }
        
         \foreach \x in {0, 0.75} { 
            \draw [fill=blue!55, line width=0.3mm] (\x,1.5) rectangle (\x+0.75, 2.25);
            \draw [fill=red!55, line width=0.3mm] (\x+2.25,0.75) rectangle (\x+3, 1.5);
        }
    
    \draw [fill=blue!55, line width=0.3mm] (0,2.25) rectangle (0.75, 3);
    
    \draw [fill=red!55, line width=0.3mm] (3,0) rectangle (3.75, 0.75);

    
      \draw [fill=orange!55, line width=0.3mm] (0,3) rectangle (0.75, 3.75);

\end{scope}

\begin{scope}[shift={(17.5,-10.75)}]
\foreach \x in {0, 0.75, 1.5, 2.25} {
            \draw [fill=blue!55, line width=0.3mm] (\x,0) rectangle (\x+0.75, 0.75);

    }
    
    \foreach \x in {0, 0.75, 1.5} { 
            \draw [fill=blue!55, line width=0.3mm] (\x,0.75) rectangle (\x+0.75, 1.5);
            
            \draw [fill=red!55, line width=0.3mm] (\x +1.5,1.5) rectangle (\x+2.25, 2.25);
            \draw [fill=red!55, line width=0.3mm] (\x+1.5,2.25) rectangle (\x+2.25, 3);
        }
        
         \foreach \x in {0, 0.75} { 
            \draw [fill=blue!55, line width=0.3mm] (\x,1.5) rectangle (\x+0.75, 2.25);
            \draw [fill=red!55, line width=0.3mm] (\x+2.25,0.75) rectangle (\x+3, 1.5);
        }

    \draw [fill=red!55, line width=0.3mm] (3,0) rectangle (3.75, 0.75);

     \foreach \x in {0, 0.75} { 
     \foreach \y in {0, -0.75}{
      \draw [fill=orange!55, line width=0.3mm] (\x,\y+3) rectangle (\x+0.75,\y+3.75);
      }
}

\end{scope}

\begin{scope}[shift={(17.5,-17.25)}]
\foreach \x in {0, 0.75, 1.5, 2.25} {
            \draw [fill=blue!55, line width=0.3mm] (\x,0) rectangle (\x+0.75, 0.75);

    }
    
    \foreach \x in {0, 0.75, 1.5} { 
            \draw [fill=blue!55, line width=0.3mm] (\x,0.75) rectangle (\x+0.75, 1.5);

        }
        
         \foreach \x in {0, 0.75} { 
            
            \draw [fill=red!55, line width=0.3mm] (\x+2.25,0.75) rectangle (\x+3, 1.5);
            
              \draw [fill=red!55, line width=0.3mm] (\x +2.25,1.5) rectangle (\x+3, 2.25);
            \draw [fill=red!55, line width=0.3mm] (\x+2.25,2.25) rectangle (\x+3, 3);
        }

    \draw [fill=red!55, line width=0.3mm] (3,0) rectangle (3.75, 0.75);

     \foreach \x in {0, 0.75, 1.5} { 
     \foreach \y in {0, -0.75, -1.5}{
      \draw [fill=orange!55, line width=0.3mm] (\x,\y+3) rectangle (\x+0.75,\y+3.75);
      }
}

\end{scope}

\begin{scope}[shift={(17.5,-23.75)}]
\foreach \x in {0, 0.75, 1.5, 2.25} {
            \draw [fill=blue!55, line width=0.3mm] (\x,0) rectangle (\x+0.75, 0.75);

    }

         \foreach \x in {0.75} { 
            
            \draw [fill=red!55, line width=0.3mm] (\x+2.25,0.75) rectangle (\x+3, 1.5);
            
              \draw [fill=red!55, line width=0.3mm] (\x +2.25,1.5) rectangle (\x+3, 2.25);
            \draw [fill=red!55, line width=0.3mm] (\x+2.25,2.25) rectangle (\x+3, 3);
        }

    \draw [fill=red!55, line width=0.3mm] (3,0) rectangle (3.75, 0.75);

     \foreach \x in {0, 0.75, 1.5, 2.25} { 
     \foreach \y in {0, -0.75, -1.5, -2.25}{
      \draw [fill=orange!55, line width=0.3mm] (\x,\y+3) rectangle (\x+0.75,\y+3.75);
      }
}

\end{scope}

\end{scope}


\begin{scope}[shift={(7,0)}]

\begin{scope}[shift={(17.5,-4.25)}]
\foreach \x in {0, 0.75, 1.5, 2.25} {
            \draw [fill=blue!55, line width=0.3mm] (\x,0) rectangle (\x+0.75, 0.75);

    }
    
    \foreach \x in {0, 0.75, 1.5} { 
            \draw [fill=blue!55, line width=0.3mm] (\x,0.75) rectangle (\x+0.75, 1.5);
            
            \draw [fill=red!55, line width=0.3mm] (\x +1.5,1.5) rectangle (\x+2.25, 2.25);
            
             \draw [fill=red!55, line width=0.3mm] (\x+1.5,2.25) rectangle (\x+2.25, 3);
        }
        
         \foreach \x in {0, 0.75} { 
            \draw [fill=blue!55, line width=0.3mm] (\x,1.5) rectangle (\x+0.75, 2.25);
            \draw [fill=red!55, line width=0.3mm] (\x+2.25,0.75) rectangle (\x+3, 1.5);
        }
    
    
    \draw [fill=red!55, line width=0.3mm] (3,0) rectangle (3.75, 0.75);

    

\end{scope}

\begin{scope}[shift={(17.5,-10.75)}]
\foreach \x in {0, 0.75, 1.5, 2.25} {
            \draw [fill=blue!55, line width=0.3mm] (\x,0) rectangle (\x+0.75, 0.75);

    }
    
    \foreach \x in {0, 0.75, 1.5} { 
            \draw [fill=blue!55, line width=0.3mm] (\x,0.75) rectangle (\x+0.75, 1.5);
            
            \draw [fill=red!55, line width=0.3mm] (\x +1.5,1.5) rectangle (\x+2.25, 2.25);
            \draw [fill=red!55, line width=0.3mm] (\x+1.5,2.25) rectangle (\x+2.25, 3);
        }
        
         \foreach \x in {0, 0.75} { 
            \draw [fill=blue!55, line width=0.3mm] (\x,1.5) rectangle (\x+0.75, 2.25);
            \draw [fill=red!55, line width=0.3mm] (\x+2.25,0.75) rectangle (\x+3, 1.5);
        }

    \draw [fill=red!55, line width=0.3mm] (3,0) rectangle (3.75, 0.75);

     \foreach \x in {0, 0.75} { 
     \foreach \y in {0, -0.75}{
      \draw [fill=orange!55, line width=0.3mm] (\x,\y+3) rectangle (\x+0.75,\y+3.75);
      }
}

\end{scope}

\begin{scope}[shift={(17.5,-17.25)}]
\foreach \x in {0, 0.75, 1.5, 2.25} {
            \draw [fill=blue!55, line width=0.3mm] (\x,0) rectangle (\x+0.75, 0.75);

    }
    
    \foreach \x in {0, 0.75, 1.5} { 
            \draw [fill=blue!55, line width=0.3mm] (\x,0.75) rectangle (\x+0.75, 1.5);

        }
        
         \foreach \x in {0, 0.75} { 
            
            \draw [fill=red!55, line width=0.3mm] (\x+2.25,0.75) rectangle (\x+3, 1.5);
            
              \draw [fill=red!55, line width=0.3mm] (\x +2.25,1.5) rectangle (\x+3, 2.25);
            \draw [fill=red!55, line width=0.3mm] (\x+2.25,2.25) rectangle (\x+3, 3);
        }

    \draw [fill=red!55, line width=0.3mm] (3,0) rectangle (3.75, 0.75);

     \foreach \x in {0, 0.75, 1.5} { 
     \foreach \y in {0, -0.75, -1.5}{
      \draw [fill=orange!55, line width=0.3mm] (\x,\y+3) rectangle (\x+0.75,\y+3.75);
      }
}

\end{scope}

\begin{scope}[shift={(17.5,-23.75)}]
\foreach \x in {0, 0.75, 1.5, 2.25} {
            \draw [fill=blue!55, line width=0.3mm] (\x,0) rectangle (\x+0.75, 0.75);

    }

         \foreach \x in {0.75} { 
            
            \draw [fill=red!55, line width=0.3mm] (\x+2.25,0.75) rectangle (\x+3, 1.5);
            
              \draw [fill=red!55, line width=0.3mm] (\x +2.25,1.5) rectangle (\x+3, 2.25);
            \draw [fill=red!55, line width=0.3mm] (\x+2.25,2.25) rectangle (\x+3, 3);
        }

    \draw [fill=red!55, line width=0.3mm] (3,0) rectangle (3.75, 0.75);

     \foreach \x in {0, 0.75, 1.5, 2.25} { 
     \foreach \y in {0, -0.75, -1.5, -2.25}{
      \draw [fill=orange!55, line width=0.3mm] (\x,\y+3) rectangle (\x+0.75,\y+3.75);
      }
}

\end{scope}

\end{scope}


\begin{scope}[shift={(14,0)}]


\begin{scope}[shift={(17.5,-4.25)}]
\foreach \x in {0, 0.75, 1.5, 2.25} {
            \draw [fill=blue!55, line width=0.3mm] (\x,0) rectangle (\x+0.75, 0.75);

   }
    
    \foreach \x in {0, 0.75, 1.5} { 
            \draw [fill=blue!55, line width=0.3mm] (\x,0.75) rectangle (\x+0.75, 1.5);

        }
        
         \foreach \x in {0, 0.75} { 
         
          \draw [fill=red!55, line width=0.3mm] (\x +2.25,1.5) rectangle (\x+3, 2.25);
            
             \draw [fill=red!55, line width=0.3mm] (\x+2.25,2.25) rectangle (\x+3, 3);
            \draw [fill=red!55, line width=0.3mm] (\x+2.25,0.75) rectangle (\x+3, 1.5);
        }
    
    
    \draw [fill=red!55, line width=0.3mm] (3,0) rectangle (3.75, 0.75);

    

\end{scope}

\begin{scope}[shift={(17.5,-10.75)}]
\foreach \x in {0, 0.75, 1.5, 2.25} {
            \draw [fill=blue!55, line width=0.3mm] (\x,0) rectangle (\x+0.75, 0.75);

    }
    
    \foreach \x in {0, 0.75, 1.5} { 
            \draw [fill=blue!55, line width=0.3mm] (\x,0.75) rectangle (\x+0.75, 1.5);

        }
        
         \foreach \x in {0, 0.75} { 
         
  \draw [fill=red!55, line width=0.3mm] (\x +2.25,1.5) rectangle (\x+3, 2.25);
            \draw [fill=red!55, line width=0.3mm] (\x+2.25,2.25) rectangle (\x+3, 3);         
         
            \draw [fill=red!55, line width=0.3mm] (\x+2.25,0.75) rectangle (\x+3, 1.5);
        }

    \draw [fill=red!55, line width=0.3mm] (3,0) rectangle (3.75, 0.75);

\end{scope}

\begin{scope}[shift={(17.5,-17.25)}]
\foreach \x in {0, 0.75, 1.5, 2.25} {
            \draw [fill=blue!55, line width=0.3mm] (\x,0) rectangle (\x+0.75, 0.75);

    }
    
    \foreach \x in {0, 0.75, 1.5} { 
            \draw [fill=blue!55, line width=0.3mm] (\x,0.75) rectangle (\x+0.75, 1.5);

        }
        
         \foreach \x in {0, 0.75} { 
            
            \draw [fill=red!55, line width=0.3mm] (\x+2.25,0.75) rectangle (\x+3, 1.5);
            
              \draw [fill=red!55, line width=0.3mm] (\x +2.25,1.5) rectangle (\x+3, 2.25);
            \draw [fill=red!55, line width=0.3mm] (\x+2.25,2.25) rectangle (\x+3, 3);
        }

    \draw [fill=red!55, line width=0.3mm] (3,0) rectangle (3.75, 0.75);

     \foreach \x in {0, 0.75, 1.5} { 
     \foreach \y in {0, -0.75, -1.5}{
      \draw [fill=orange!55, line width=0.3mm] (\x,\y+3) rectangle (\x+0.75,\y+3.75);
      }
}

\end{scope}

\begin{scope}[shift={(17.5,-23.75)}]
\foreach \x in {0, 0.75, 1.5, 2.25} {
            \draw [fill=blue!55, line width=0.3mm] (\x,0) rectangle (\x+0.75, 0.75);

    }

         \foreach \x in {0.75} { 
            
            \draw [fill=red!55, line width=0.3mm] (\x+2.25,0.75) rectangle (\x+3, 1.5);
            
              \draw [fill=red!55, line width=0.3mm] (\x +2.25,1.5) rectangle (\x+3, 2.25);
            \draw [fill=red!55, line width=0.3mm] (\x+2.25,2.25) rectangle (\x+3, 3);
        }

    \draw [fill=red!55, line width=0.3mm] (3,0) rectangle (3.75, 0.75);

     \foreach \x in {0, 0.75, 1.5, 2.25} { 
     \foreach \y in {0, -0.75, -1.5, -2.25}{
      \draw [fill=orange!55, line width=0.3mm] (\x,\y+3) rectangle (\x+0.75,\y+3.75);
      }
}

\end{scope}

\end{scope}


\begin{scope}[shift={(21,0)}]


\begin{scope}[shift={(17.5,-4.25)}]
\foreach \x in {0, 0.75, 1.5, 2.25} {
            \draw [fill=blue!55, line width=0.3mm] (\x,0) rectangle (\x+0.75, 0.75);

   }
    
    \foreach \x in {0, 0.75, 1.5} { 

        }

         \draw [fill=red!55, line width=0.3mm] (3,1.5) rectangle (3.75, 2.25);
            
             \draw [fill=red!55, line width=0.3mm] (3,2.25) rectangle (3.75, 3);
            
            \draw [fill=red!55, line width=0.3mm] (3,0.75) rectangle (3.75, 1.5);
 
    \draw [fill=red!55, line width=0.3mm] (3,0) rectangle (3.75, 0.75);

    

\end{scope}

\begin{scope}[shift={(17.5,-10.75)}]
\foreach \x in {0, 0.75, 1.5, 2.25} {
            \draw [fill=blue!55, line width=0.3mm] (\x,0) rectangle (\x+0.75, 0.75);

    }

      \draw [fill=red!55, line width=0.3mm] (3,1.5) rectangle (3.75, 2.25);
            \draw [fill=red!55, line width=0.3mm] (3,2.25) rectangle (3.75, 3);

            \draw [fill=red!55, line width=0.3mm] (3,0.75) rectangle (3.75, 1.5);

    \draw [fill=red!55, line width=0.3mm] (3,0) rectangle (3.75, 0.75);

\end{scope}

\begin{scope}[shift={(17.5,-17.25)}]
\foreach \x in {0, 0.75, 1.5, 2.25} {
            \draw [fill=blue!55, line width=0.3mm] (\x,0) rectangle (\x+0.75, 0.75);

    }

       \draw [fill=red!55, line width=0.3mm] (3,0.75) rectangle (3.75, 1.5);
            
              \draw [fill=red!55, line width=0.3mm] (3,1.5) rectangle (3.75, 2.25);
            \draw [fill=red!55, line width=0.3mm] (3,2.25) rectangle (3.75, 3);

    \draw [fill=red!55, line width=0.3mm] (3,0) rectangle (3.75, 0.75);

\end{scope}

\begin{scope}[shift={(17.5,-23.75)}]
\foreach \x in {0, 0.75, 1.5, 2.25} {
            \draw [fill=blue!55, line width=0.3mm] (\x,0) rectangle (\x+0.75, 0.75);

    }

         \foreach \x in {0.75} { 
            
            \draw [fill=red!55, line width=0.3mm] (\x+2.25,0.75) rectangle (\x+3, 1.5);
            
              \draw [fill=red!55, line width=0.3mm] (\x +2.25,1.5) rectangle (\x+3, 2.25);
            \draw [fill=red!55, line width=0.3mm] (\x+2.25,2.25) rectangle (\x+3, 3);
        }

    \draw [fill=red!55, line width=0.3mm] (3,0) rectangle (3.75, 0.75);

     \foreach \x in {0, 0.75, 1.5, 2.25} { 
     \foreach \y in {0, -0.75, -1.5, -2.25}{
      \draw [fill=orange!55, line width=0.3mm] (\x,\y+3) rectangle (\x+0.75,\y+3.75);
      }
}

\end{scope}

\end{scope}

\end{tikzpicture}
\vskip .8cm
\captionof{figure}{2D sections of Figure $\ref{3 pyramids 4D}$}
\label{2D section Nicomachus 4D}
\vskip .2cm

 \restoreparindent If we take a 2D X-ray of each of the 3D sections of the three 4D pyramids assembled as in Figure~\ref{3 pyramids 4D}, we obtain $n$ columns of 2D puzzles (in Figure~\ref{2D section Nicomachus 4D} there are four columns because in our illustration $n=4$), one column for each 3D section. Note that the first column is exactly the X-ray of Figure~\ref{2D of 3D}, since we have X-rayed the 3D block corresponding to the first 3D section of the assembled 4D pyramids, and this first X-ray always coincides with the visual proof in the lower dimension.

\restoreparindent We observe that in the second column the $1\times 1$ square and the step of length $1$ from the triangles in the first column have been removed; in the third column, the $2\times 2$ squares and the steps of length $2$ are removed as well, and this pattern continues until the last section, where all squares up to $(n-1)\times(n-1)$ and all the steps up to length $n-1$ have been removed.

\restoreparindent It is precisely in the gaps that are created that the fourth 4D pyramid fits (see Figure~\ref{4 pyramids 4D}). In Figure~\ref{2D section Nicomachus 4D Bis}, we see that, in addition to the $n$ columns (four in our illustration) corresponding to the 2D sections of the 4D blocks, there is an extra row. This extra row is precisely the part that protruded from the 3D blocks.\par
\vspace{5mm}
\noindent
\begin{tikzpicture}[scale=0.5]


\begin{scope}[shift={(20.3,7.2)}]
 \foreach \x in {0} { 
  \foreach \y in {0} { 
\begin{scope}[shift={(17.5+\x, -4.5-\y)}] 
   
\draw[<->] (0.65,0) -- (3.75,0) node[midway, below, yshift=0.5mm] {$n$};
\draw[<->] (4.03,0.2) -- (4.03,3.3) node[midway, right, xshift=-0.5mm] {$n$};

\end{scope}

}
}
\end{scope}




\begin{scope}[shift={(0,0)}]

\begin{scope}[shift={(17.5,-4.25)}]
\foreach \x in {0, 0.75, 1.5, 2.25} {
            \draw [fill=blue!55, line width=0.3mm] (\x,0) rectangle (\x+0.75, 0.75);
            
            \draw [fill=red!55, line width=0.3mm] (\x+0.75,2.25) rectangle (\x+1.5, 3);
       
    }
    
    \foreach \x in {0, 0.75, 1.5} { 
            \draw [fill=blue!55, line width=0.3mm] (\x,0.75) rectangle (\x+0.75, 1.5);
            
            \draw [fill=red!55, line width=0.3mm] (\x +1.5,1.5) rectangle (\x+2.25, 2.25);
        }
        
         \foreach \x in {0, 0.75} { 
            \draw [fill=blue!55, line width=0.3mm] (\x,1.5) rectangle (\x+0.75, 2.25);
            \draw [fill=red!55, line width=0.3mm] (\x+2.25,0.75) rectangle (\x+3, 1.5);
        }
    
    \draw [fill=blue!55, line width=0.3mm] (0,2.25) rectangle (0.75, 3);
    
    \draw [fill=red!55, line width=0.3mm] (3,0) rectangle (3.75, 0.75);

    
      \draw [fill=orange!55, line width=0.3mm] (0,3) rectangle (0.75, 3.75);


\begin{scope}[shift={(0,6.5)}]
 \draw [fill=green!55, line width=0.3mm] (0,3) rectangle (0.75, 3.75);
\end{scope}

\end{scope}

\begin{scope}[shift={(17.5,-10.75)}]
\foreach \x in {0, 0.75, 1.5, 2.25} {
            \draw [fill=blue!55, line width=0.3mm] (\x,0) rectangle (\x+0.75, 0.75);

    }
    
    \foreach \x in {0, 0.75, 1.5} { 
            \draw [fill=blue!55, line width=0.3mm] (\x,0.75) rectangle (\x+0.75, 1.5);
            
            \draw [fill=red!55, line width=0.3mm] (\x +1.5,1.5) rectangle (\x+2.25, 2.25);
            \draw [fill=red!55, line width=0.3mm] (\x+1.5,2.25) rectangle (\x+2.25, 3);
        }
        
         \foreach \x in {0, 0.75} { 
            \draw [fill=blue!55, line width=0.3mm] (\x,1.5) rectangle (\x+0.75, 2.25);
            \draw [fill=red!55, line width=0.3mm] (\x+2.25,0.75) rectangle (\x+3, 1.5);
        }

    \draw [fill=red!55, line width=0.3mm] (3,0) rectangle (3.75, 0.75);

     \foreach \x in {0, 0.75} { 
     \foreach \y in {0, -0.75}{
      \draw [fill=orange!55, line width=0.3mm] (\x,\y+3) rectangle (\x+0.75,\y+3.75);
      }
}

\end{scope}

\begin{scope}[shift={(17.5,-17.25)}]
\foreach \x in {0, 0.75, 1.5, 2.25} {
            \draw [fill=blue!55, line width=0.3mm] (\x,0) rectangle (\x+0.75, 0.75);

    }
    
    \foreach \x in {0, 0.75, 1.5} { 
            \draw [fill=blue!55, line width=0.3mm] (\x,0.75) rectangle (\x+0.75, 1.5);

        }
        
         \foreach \x in {0, 0.75} { 
            
            \draw [fill=red!55, line width=0.3mm] (\x+2.25,0.75) rectangle (\x+3, 1.5);
            
              \draw [fill=red!55, line width=0.3mm] (\x +2.25,1.5) rectangle (\x+3, 2.25);
            \draw [fill=red!55, line width=0.3mm] (\x+2.25,2.25) rectangle (\x+3, 3);
        }

    \draw [fill=red!55, line width=0.3mm] (3,0) rectangle (3.75, 0.75);

     \foreach \x in {0, 0.75, 1.5} { 
     \foreach \y in {0, -0.75, -1.5}{
      \draw [fill=orange!55, line width=0.3mm] (\x,\y+3) rectangle (\x+0.75,\y+3.75);
      }
}

\end{scope}

\begin{scope}[shift={(17.5,-23.75)}]
\foreach \x in {0, 0.75, 1.5, 2.25} {
            \draw [fill=blue!55, line width=0.3mm] (\x,0) rectangle (\x+0.75, 0.75);

    }

         \foreach \x in {0.75} { 
            
            \draw [fill=red!55, line width=0.3mm] (\x+2.25,0.75) rectangle (\x+3, 1.5);
            
              \draw [fill=red!55, line width=0.3mm] (\x +2.25,1.5) rectangle (\x+3, 2.25);
            \draw [fill=red!55, line width=0.3mm] (\x+2.25,2.25) rectangle (\x+3, 3);
        }

    \draw [fill=red!55, line width=0.3mm] (3,0) rectangle (3.75, 0.75);

     \foreach \x in {0, 0.75, 1.5, 2.25} { 
     \foreach \y in {0, -0.75, -1.5, -2.25}{
      \draw [fill=orange!55, line width=0.3mm] (\x,\y+3) rectangle (\x+0.75,\y+3.75);
      }
}

\end{scope}

\end{scope}


\begin{scope}[shift={(7,0)}]

\begin{scope}[shift={(17.5,-4.25)}]
\foreach \x in {0, 0.75, 1.5, 2.25} {
            \draw [fill=blue!55, line width=0.3mm] (\x,0) rectangle (\x+0.75, 0.75);

    }
    
    \foreach \x in {0, 0.75, 1.5} { 
            \draw [fill=blue!55, line width=0.3mm] (\x,0.75) rectangle (\x+0.75, 1.5);
            
            \draw [fill=red!55, line width=0.3mm] (\x +1.5,1.5) rectangle (\x+2.25, 2.25);
            
             \draw [fill=red!55, line width=0.3mm] (\x+1.5,2.25) rectangle (\x+2.25, 3);
        }
        
         \foreach \x in {0, 0.75} { 
            \draw [fill=blue!55, line width=0.3mm] (\x,1.5) rectangle (\x+0.75, 2.25);
            \draw [fill=red!55, line width=0.3mm] (\x+2.25,0.75) rectangle (\x+3, 1.5);
        }
    
    
    \draw [fill=red!55, line width=0.3mm] (3,0) rectangle (3.75, 0.75);

    

\end{scope}

\begin{scope}[shift={(17.5,-10.75)}]
\foreach \x in {0, 0.75, 1.5, 2.25} {
            \draw [fill=blue!55, line width=0.3mm] (\x,0) rectangle (\x+0.75, 0.75);

    }
    
    \foreach \x in {0, 0.75, 1.5} { 
            \draw [fill=blue!55, line width=0.3mm] (\x,0.75) rectangle (\x+0.75, 1.5);
            
            \draw [fill=red!55, line width=0.3mm] (\x +1.5,1.5) rectangle (\x+2.25, 2.25);
            \draw [fill=red!55, line width=0.3mm] (\x+1.5,2.25) rectangle (\x+2.25, 3);
        }
        
         \foreach \x in {0, 0.75} { 
            \draw [fill=blue!55, line width=0.3mm] (\x,1.5) rectangle (\x+0.75, 2.25);
            \draw [fill=red!55, line width=0.3mm] (\x+2.25,0.75) rectangle (\x+3, 1.5);
        }

    \draw [fill=red!55, line width=0.3mm] (3,0) rectangle (3.75, 0.75);

     \foreach \x in {0, 0.75} { 
     \foreach \y in {0, -0.75}{
      \draw [fill=orange!55, line width=0.3mm] (\x,\y+3) rectangle (\x+0.75,\y+3.75);
      }
}

\begin{scope}[shift={(0,6.5)}]
 \foreach \x in {0, 0.75} { 
     \foreach \y in {0, -0.75}{
      \draw [fill=green!55, line width=0.3mm] (\x,\y+3) rectangle (\x+0.75,\y+3.75);
      }
}
\end{scope}

\begin{scope}[shift={(0,13)}]
 \foreach \x in {0, 0.75} { 
     \foreach \y in {0, -0.75}{
      \draw [fill=green!55, line width=0.3mm] (\x,\y+3) rectangle (\x+0.75,\y+3.75);
      }
}
\end{scope}

\end{scope}

\begin{scope}[shift={(17.5,-17.25)}]
\foreach \x in {0, 0.75, 1.5, 2.25} {
            \draw [fill=blue!55, line width=0.3mm] (\x,0) rectangle (\x+0.75, 0.75);

    }
    
    \foreach \x in {0, 0.75, 1.5} { 
            \draw [fill=blue!55, line width=0.3mm] (\x,0.75) rectangle (\x+0.75, 1.5);

        }
        
         \foreach \x in {0, 0.75} { 
            
            \draw [fill=red!55, line width=0.3mm] (\x+2.25,0.75) rectangle (\x+3, 1.5);
            
              \draw [fill=red!55, line width=0.3mm] (\x +2.25,1.5) rectangle (\x+3, 2.25);
            \draw [fill=red!55, line width=0.3mm] (\x+2.25,2.25) rectangle (\x+3, 3);
        }

    \draw [fill=red!55, line width=0.3mm] (3,0) rectangle (3.75, 0.75);

     \foreach \x in {0, 0.75, 1.5} { 
     \foreach \y in {0, -0.75, -1.5}{
      \draw [fill=orange!55, line width=0.3mm] (\x,\y+3) rectangle (\x+0.75,\y+3.75);
      }
}

\end{scope}

\begin{scope}[shift={(17.5,-23.75)}]
\foreach \x in {0, 0.75, 1.5, 2.25} {
            \draw [fill=blue!55, line width=0.3mm] (\x,0) rectangle (\x+0.75, 0.75);

    }

         \foreach \x in {0.75} { 
            
            \draw [fill=red!55, line width=0.3mm] (\x+2.25,0.75) rectangle (\x+3, 1.5);
            
              \draw [fill=red!55, line width=0.3mm] (\x +2.25,1.5) rectangle (\x+3, 2.25);
            \draw [fill=red!55, line width=0.3mm] (\x+2.25,2.25) rectangle (\x+3, 3);
        }

    \draw [fill=red!55, line width=0.3mm] (3,0) rectangle (3.75, 0.75);

     \foreach \x in {0, 0.75, 1.5, 2.25} { 
     \foreach \y in {0, -0.75, -1.5, -2.25}{
      \draw [fill=orange!55, line width=0.3mm] (\x,\y+3) rectangle (\x+0.75,\y+3.75);
      }
}

\end{scope}

\end{scope}


\begin{scope}[shift={(14,0)}]


\begin{scope}[shift={(17.5,-4.25)}]
\foreach \x in {0, 0.75, 1.5, 2.25} {
            \draw [fill=blue!55, line width=0.3mm] (\x,0) rectangle (\x+0.75, 0.75);

   }
    
    \foreach \x in {0, 0.75, 1.5} { 
            \draw [fill=blue!55, line width=0.3mm] (\x,0.75) rectangle (\x+0.75, 1.5);

        }
        
         \foreach \x in {0, 0.75} { 
         
          \draw [fill=red!55, line width=0.3mm] (\x +2.25,1.5) rectangle (\x+3, 2.25);
            
             \draw [fill=red!55, line width=0.3mm] (\x+2.25,2.25) rectangle (\x+3, 3);
            \draw [fill=red!55, line width=0.3mm] (\x+2.25,0.75) rectangle (\x+3, 1.5);
        }
    
    
    \draw [fill=red!55, line width=0.3mm] (3,0) rectangle (3.75, 0.75);

    

\end{scope}

\begin{scope}[shift={(17.5,-10.75)}]
\foreach \x in {0, 0.75, 1.5, 2.25} {
            \draw [fill=blue!55, line width=0.3mm] (\x,0) rectangle (\x+0.75, 0.75);

    }
    
    \foreach \x in {0, 0.75, 1.5} { 
            \draw [fill=blue!55, line width=0.3mm] (\x,0.75) rectangle (\x+0.75, 1.5);

        }
        
         \foreach \x in {0, 0.75} { 
         
  \draw [fill=red!55, line width=0.3mm] (\x +2.25,1.5) rectangle (\x+3, 2.25);
            \draw [fill=red!55, line width=0.3mm] (\x+2.25,2.25) rectangle (\x+3, 3);         
         
            \draw [fill=red!55, line width=0.3mm] (\x+2.25,0.75) rectangle (\x+3, 1.5);
        }

    \draw [fill=red!55, line width=0.3mm] (3,0) rectangle (3.75, 0.75);

\end{scope}

\begin{scope}[shift={(17.5,-17.25)}]
\foreach \x in {0, 0.75, 1.5, 2.25} {
            \draw [fill=blue!55, line width=0.3mm] (\x,0) rectangle (\x+0.75, 0.75);

    }
    
    \foreach \x in {0, 0.75, 1.5} { 
            \draw [fill=blue!55, line width=0.3mm] (\x,0.75) rectangle (\x+0.75, 1.5);

        }
        
         \foreach \x in {0, 0.75} { 
            
            \draw [fill=red!55, line width=0.3mm] (\x+2.25,0.75) rectangle (\x+3, 1.5);
            
              \draw [fill=red!55, line width=0.3mm] (\x +2.25,1.5) rectangle (\x+3, 2.25);
            \draw [fill=red!55, line width=0.3mm] (\x+2.25,2.25) rectangle (\x+3, 3);
        }

    \draw [fill=red!55, line width=0.3mm] (3,0) rectangle (3.75, 0.75);

     \foreach \x in {0, 0.75, 1.5} { 
     \foreach \y in {0, -0.75, -1.5}{
      \draw [fill=orange!55, line width=0.3mm] (\x,\y+3) rectangle (\x+0.75,\y+3.75);
      }
}

\begin{scope}[shift={(0,6.5)}]
\foreach \x in {0, 0.75, 1.5} { 
     \foreach \y in {0, -0.75, -1.5}{
      \draw [fill=green!55, line width=0.3mm] (\x,\y+3) rectangle (\x+0.75,\y+3.75);
      }
}
\end{scope}

\begin{scope}[shift={(0,13)}]
\foreach \x in {0, 0.75, 1.5} { 
     \foreach \y in {0, -0.75, -1.5}{
      \draw [fill=green!55, line width=0.3mm] (\x,\y+3) rectangle (\x+0.75,\y+3.75);
      }
}
\end{scope}

\begin{scope}[shift={(0,19.5)}]
\foreach \x in {0, 0.75, 1.5} { 
     \foreach \y in {0, -0.75, -1.5}{
      \draw [fill=green!55, line width=0.3mm] (\x,\y+3) rectangle (\x+0.75,\y+3.75);
      }
}
\end{scope}

\end{scope}

\begin{scope}[shift={(17.5,-23.75)}]
\foreach \x in {0, 0.75, 1.5, 2.25} {
            \draw [fill=blue!55, line width=0.3mm] (\x,0) rectangle (\x+0.75, 0.75);

    }

         \foreach \x in {0.75} { 
            
            \draw [fill=red!55, line width=0.3mm] (\x+2.25,0.75) rectangle (\x+3, 1.5);
            
              \draw [fill=red!55, line width=0.3mm] (\x +2.25,1.5) rectangle (\x+3, 2.25);
            \draw [fill=red!55, line width=0.3mm] (\x+2.25,2.25) rectangle (\x+3, 3);
        }

    \draw [fill=red!55, line width=0.3mm] (3,0) rectangle (3.75, 0.75);

     \foreach \x in {0, 0.75, 1.5, 2.25} { 
     \foreach \y in {0, -0.75, -1.5, -2.25}{
      \draw [fill=orange!55, line width=0.3mm] (\x,\y+3) rectangle (\x+0.75,\y+3.75);
      }
}

\end{scope}

\end{scope}


\begin{scope}[shift={(21,0)}]


\begin{scope}[shift={(17.5,-4.25)}]
\foreach \x in {0, 0.75, 1.5, 2.25} {
            \draw [fill=blue!55, line width=0.3mm] (\x,0) rectangle (\x+0.75, 0.75);

   }
    
    \foreach \x in {0, 0.75, 1.5} { 

        }

         \draw [fill=red!55, line width=0.3mm] (3,1.5) rectangle (3.75, 2.25);
            
             \draw [fill=red!55, line width=0.3mm] (3,2.25) rectangle (3.75, 3);
            
            \draw [fill=red!55, line width=0.3mm] (3,0.75) rectangle (3.75, 1.5);
 
    \draw [fill=red!55, line width=0.3mm] (3,0) rectangle (3.75, 0.75);

    

\end{scope}

\begin{scope}[shift={(17.5,-10.75)}]
\foreach \x in {0, 0.75, 1.5, 2.25} {
            \draw [fill=blue!55, line width=0.3mm] (\x,0) rectangle (\x+0.75, 0.75);

    }

      \draw [fill=red!55, line width=0.3mm] (3,1.5) rectangle (3.75, 2.25);
            \draw [fill=red!55, line width=0.3mm] (3,2.25) rectangle (3.75, 3);

            \draw [fill=red!55, line width=0.3mm] (3,0.75) rectangle (3.75, 1.5);

    \draw [fill=red!55, line width=0.3mm] (3,0) rectangle (3.75, 0.75);

\end{scope}

\begin{scope}[shift={(17.5,-17.25)}]
\foreach \x in {0, 0.75, 1.5, 2.25} {
            \draw [fill=blue!55, line width=0.3mm] (\x,0) rectangle (\x+0.75, 0.75);

    }

       \draw [fill=red!55, line width=0.3mm] (3,0.75) rectangle (3.75, 1.5);
            
              \draw [fill=red!55, line width=0.3mm] (3,1.5) rectangle (3.75, 2.25);
            \draw [fill=red!55, line width=0.3mm] (3,2.25) rectangle (3.75, 3);

    \draw [fill=red!55, line width=0.3mm] (3,0) rectangle (3.75, 0.75);

\end{scope}

\begin{scope}[shift={(17.5,-23.75)}]
\foreach \x in {0, 0.75, 1.5, 2.25} {
            \draw [fill=blue!55, line width=0.3mm] (\x,0) rectangle (\x+0.75, 0.75);

    }

         \foreach \x in {0.75} { 
            
            \draw [fill=red!55, line width=0.3mm] (\x+2.25,0.75) rectangle (\x+3, 1.5);
            
              \draw [fill=red!55, line width=0.3mm] (\x +2.25,1.5) rectangle (\x+3, 2.25);
            \draw [fill=red!55, line width=0.3mm] (\x+2.25,2.25) rectangle (\x+3, 3);
        }

    \draw [fill=red!55, line width=0.3mm] (3,0) rectangle (3.75, 0.75);

     \foreach \x in {0, 0.75, 1.5, 2.25} { 
     \foreach \y in {0, -0.75, -1.5, -2.25}{
      \draw [fill=orange!55, line width=0.3mm] (\x,\y+3) rectangle (\x+0.75,\y+3.75);
      }
}


\begin{scope}[shift={(0,6.5)}]
 \foreach \x in {0, 0.75, 1.5, 2.25} { 
     \foreach \y in {0, -0.75, -1.5, -2.25}{
      \draw [fill=green!55, line width=0.3mm] (\x,\y+3) rectangle (\x+0.75,\y+3.75);
      }
}
\end{scope}

\begin{scope}[shift={(0,13)}]
 \foreach \x in {0, 0.75, 1.5, 2.25} { 
     \foreach \y in {0, -0.75, -1.5, -2.25}{
      \draw [fill=green!55, line width=0.3mm] (\x,\y+3) rectangle (\x+0.75,\y+3.75);
      }
}
\end{scope}

\begin{scope}[shift={(0,19.5)}]
 \foreach \x in {0, 0.75, 1.5, 2.25} { 
     \foreach \y in {0, -0.75, -1.5, -2.25}{
      \draw [fill=green!55, line width=0.3mm] (\x,\y+3) rectangle (\x+0.75,\y+3.75);
      }
}
\end{scope}

\begin{scope}[shift={(0,26)}]
 \foreach \x in {0, 0.75, 1.5, 2.25} { 
     \foreach \y in {0, -0.75, -1.5, -2.25}{
      \draw [fill=green!55, line width=0.3mm] (\x,\y+3) rectangle (\x+0.75,\y+3.75);
      }
}
\end{scope}

\end{scope}

\end{scope}

\end{tikzpicture}
\vskip .8cm
\captionof{figure}{2D sections of Figure $\ref{4 pyramids 4D}$}
\label{2D section Nicomachus 4D Bis}
\vskip .2cm
\restoreparindent We now use this extra first row to fill the gaps in the other pieces until each shape becomes a rectangle. In Figure~\ref{DIY Filling}, this step was based on identity~\eqref{Algebraic formula} and a 3D picture. In this 2D version, the same DIY step is shown in a different visual way, which is even simpler.

The top‐right corner is empty in every 2D puzzle, in all rows after the first one. In total, there are $n$ rows and 
$n$ columns with this missing top‐right square. These 
$n\times n$ gaps correspond to the 
$n\times n$ small squares of the green block in the last section of the first row, and we use those green squares to fill all these gaps.

If we now ignore the last row and the last column, we have 
$n-1$ rows and $n-1$ columns of 2D puzzles. In this reduced array, the empty square moves to the second position from the right in the first row. Again, there are as many gaps as green squares in the penultimate section of the first row, so we can fill them as well.

We repeat this procedure: each time we remove the outer row and column, the gap shifts one position to the left, and we fill it with the corresponding green squares from the first row. In the end, we obtain an $n\times n$ array of rectangular 2D puzzles, each with base and height $n+1$ (see Figure~\ref{2D section Nicomachus 4D Bis Bis}).

There are as many unit squares in this 2D puzzle as there are unit cubes in the four 4D pyramids. Therefore $4(1^3 + 2^3 + 3^3 +\cdots +n^3)=n^2 (n+1)^2$ which is exactly the identity we already knew.\par
\vspace{5mm}
\noindent
\begin{tikzpicture}[scale=0.43]

Esta es la placa de radiografía vertical

\draw [dashed, fill=white!40, thick] (17,0) rectangle (44.6,-25);
\draw [dashed, fill=white!40, thick] (17,0) rectangle (37.6,-19);
\draw [dashed, fill=white!40, thick] (17,0) rectangle (30.6,-12);
\draw [dashed, fill=white!40, thick] (17,0) rectangle (23.6,-5.5);

\begin{scope}[shift={(0,0)}]

\begin{scope}[shift={(17.5,-4.25)}]
\foreach \x in {0, 0.75, 1.5, 2.25} {
            \draw [fill=blue!55, line width=0.3mm] (\x,0) rectangle (\x+0.75, 0.75);
             \draw [fill=green!55, line width=0.3mm] (\x+0.75,3) rectangle (\x+1.5, 3.75);
            
            \draw [fill=red!55, line width=0.3mm] (\x+0.75,2.25) rectangle (\x+1.5, 3);
       
    }
    
    \foreach \x in {0, 0.75, 1.5} { 
            \draw [fill=blue!55, line width=0.3mm] (\x,0.75) rectangle (\x+0.75, 1.5);
            
            \draw [fill=red!55, line width=0.3mm] (\x +1.5,1.5) rectangle (\x+2.25, 2.25);
        }
        
         \foreach \x in {0, 0.75} { 
            \draw [fill=blue!55, line width=0.3mm] (\x,1.5) rectangle (\x+0.75, 2.25);
            \draw [fill=red!55, line width=0.3mm] (\x+2.25,0.75) rectangle (\x+3, 1.5);
        }
    
    \draw [fill=blue!55, line width=0.3mm] (0,2.25) rectangle (0.75, 3);
    
    \draw [fill=red!55, line width=0.3mm] (3,0) rectangle (3.75, 0.75);

    
      \draw [fill=orange!55, line width=0.3mm] (0,3) rectangle (0.75, 3.75);

\end{scope}

\begin{scope}[shift={(17.5,-10.75)}]
\foreach \x in {0, 0.75, 1.5, 2.25} {
            \draw [fill=blue!55, line width=0.3mm] (\x,0) rectangle (\x+0.75, 0.75);

    }
    
    \foreach \x in {0, 0.75, 1.5} { 
            \draw [fill=blue!55, line width=0.3mm] (\x,0.75) rectangle (\x+0.75, 1.5);
            \draw [fill=green!55, line width=0.3mm] (\x+1.5,3) rectangle (\x+2.25, 3.75);
            
            \draw [fill=red!55, line width=0.3mm] (\x +1.5,1.5) rectangle (\x+2.25, 2.25);
            \draw [fill=red!55, line width=0.3mm] (\x+1.5,2.25) rectangle (\x+2.25, 3);
        }
        
         \foreach \x in {0, 0.75} { 
            \draw [fill=blue!55, line width=0.3mm] (\x,1.5) rectangle (\x+0.75, 2.25);
            \draw [fill=red!55, line width=0.3mm] (\x+2.25,0.75) rectangle (\x+3, 1.5);
        }

    \draw [fill=red!55, line width=0.3mm] (3,0) rectangle (3.75, 0.75);

     \foreach \x in {0, 0.75} { 
     \foreach \y in {0, -0.75}{
      \draw [fill=orange!55, line width=0.3mm] (\x,\y+3) rectangle (\x+0.75,\y+3.75);
      }
}

\end{scope}

\begin{scope}[shift={(17.5,-17.25)}]
\foreach \x in {0, 0.75, 1.5, 2.25} {
            \draw [fill=blue!55, line width=0.3mm] (\x,0) rectangle (\x+0.75, 0.75);

    }
    
    \foreach \x in {0, 0.75, 1.5} { 
            \draw [fill=blue!55, line width=0.3mm] (\x,0.75) rectangle (\x+0.75, 1.5);

        }
        
         \foreach \x in {0, 0.75} { 
             \draw [fill=green!55, line width=0.3mm] (\x+2.25,3) rectangle (\x+3, 3.75);
            \draw [fill=red!55, line width=0.3mm] (\x+2.25,0.75) rectangle (\x+3, 1.5);
            
              \draw [fill=red!55, line width=0.3mm] (\x +2.25,1.5) rectangle (\x+3, 2.25);
            \draw [fill=red!55, line width=0.3mm] (\x+2.25,2.25) rectangle (\x+3, 3);
        }

    \draw [fill=red!55, line width=0.3mm] (3,0) rectangle (3.75, 0.75);

     \foreach \x in {0, 0.75, 1.5} { 
     \foreach \y in {0, -0.75, -1.5}{
      \draw [fill=orange!55, line width=0.3mm] (\x,\y+3) rectangle (\x+0.75,\y+3.75);
      }
}

\end{scope}

\begin{scope}[shift={(17.5,-23.75)}]
\foreach \x in {0, 0.75, 1.5, 2.25} {
            \draw [fill=blue!55, line width=0.3mm] (\x,0) rectangle (\x+0.75, 0.75);
            
    }
    
         \foreach \x in {0.75} { 
            
              \draw [fill=green!55, line width=0.3mm] (\x+2.25,3) rectangle (\x+3, 3.75);
            \draw [fill=red!55, line width=0.3mm] (\x+2.25,0.75) rectangle (\x+3, 1.5);
            
              \draw [fill=red!55, line width=0.3mm] (\x +2.25,1.5) rectangle (\x+3, 2.25);
            \draw [fill=red!55, line width=0.3mm] (\x+2.25,2.25) rectangle (\x+3, 3);
        }
    
    \draw [fill=red!55, line width=0.3mm] (3,0) rectangle (3.75, 0.75);
    
     \foreach \x in {0, 0.75, 1.5, 2.25} { 
     \foreach \y in {0, -0.75, -1.5, -2.25}{
      \draw [fill=orange!55, line width=0.3mm] (\x,\y+3) rectangle (\x+0.75,\y+3.75);
      }
}

\end{scope}

\end{scope}


\begin{scope}[shift={(7,0)}]

\begin{scope}[shift={(17.5,-4.25)}]
\foreach \x in {0, 0.75, 1.5, 2.25} {
            \draw [fill=blue!55, line width=0.3mm] (\x,0) rectangle (\x+0.75, 0.75);

    }
    
    \foreach \x in {0, 0.75, 1.5} { 
  \draw [fill=green!55, line width=0.3mm] (\x+1.5,3) rectangle (\x+2.25, 3.75);    
    
            \draw [fill=blue!55, line width=0.3mm] (\x,0.75) rectangle (\x+0.75, 1.5);
            
            \draw [fill=red!55, line width=0.3mm] (\x +1.5,1.5) rectangle (\x+2.25, 2.25);
            
             \draw [fill=red!55, line width=0.3mm] (\x+1.5,2.25) rectangle (\x+2.25, 3);
        }
        
         \foreach \x in {0, 0.75} { 
            \draw [fill=blue!55, line width=0.3mm] (\x,1.5) rectangle (\x+0.75, 2.25);
            \draw [fill=red!55, line width=0.3mm] (\x+2.25,0.75) rectangle (\x+3, 1.5);
        }
    
    
    \draw [fill=red!55, line width=0.3mm] (3,0) rectangle (3.75, 0.75);
    
    

\end{scope}

\begin{scope}[shift={(17.5,-10.75)}]
\foreach \x in {0, 0.75, 1.5, 2.25} {
            \draw [fill=blue!55, line width=0.3mm] (\x,0) rectangle (\x+0.75, 0.75);
            
    }
    
    \foreach \x in {0, 0.75, 1.5} { 
    
    \draw [fill=green!55, line width=0.3mm] (\x+1.5,3) rectangle (\x+2.25, 3.75);
            \draw [fill=blue!55, line width=0.3mm] (\x,0.75) rectangle (\x+0.75, 1.5);
            
            \draw [fill=red!55, line width=0.3mm] (\x +1.5,1.5) rectangle (\x+2.25, 2.25);
            \draw [fill=red!55, line width=0.3mm] (\x+1.5,2.25) rectangle (\x+2.25, 3);
        }
        
         \foreach \x in {0, 0.75} { 
            \draw [fill=blue!55, line width=0.3mm] (\x,1.5) rectangle (\x+0.75, 2.25);
            \draw [fill=red!55, line width=0.3mm] (\x+2.25,0.75) rectangle (\x+3, 1.5);
        }

    \draw [fill=red!55, line width=0.3mm] (3,0) rectangle (3.75, 0.75);

     \foreach \x in {0, 0.75} { 
     \foreach \y in {0, -0.75}{
      \draw [fill=orange!55, line width=0.3mm] (\x,\y+3) rectangle (\x+0.75,\y+3.75);
      }
}

\begin{scope}[shift={(0,6.5)}]
 \foreach \x in {0, 0.75} { 
     \foreach \y in {0, -0.75}{
      \draw [fill=green!55, line width=0.3mm] (\x,\y+3) rectangle (\x+0.75,\y+3.75);
      }
}
\end{scope}

\end{scope}

\begin{scope}[shift={(17.5,-17.25)}]
\foreach \x in {0, 0.75, 1.5, 2.25} {
            \draw [fill=blue!55, line width=0.3mm] (\x,0) rectangle (\x+0.75, 0.75);
            
    }
    
    \foreach \x in {0, 0.75, 1.5} { 
            \draw [fill=blue!55, line width=0.3mm] (\x,0.75) rectangle (\x+0.75, 1.5);
            
        }
        
         \foreach \x in {0, 0.75} { 
            \draw [fill=green!55, line width=0.3mm] (\x+2.25,3) rectangle (\x+3, 3.75);
            
            \draw [fill=red!55, line width=0.3mm] (\x+2.25,0.75) rectangle (\x+3, 1.5);
            
              \draw [fill=red!55, line width=0.3mm] (\x +2.25,1.5) rectangle (\x+3, 2.25);
            \draw [fill=red!55, line width=0.3mm] (\x+2.25,2.25) rectangle (\x+3, 3);
        }

    \draw [fill=red!55, line width=0.3mm] (3,0) rectangle (3.75, 0.75);
    
     \foreach \x in {0, 0.75, 1.5} { 
     \foreach \y in {0, -0.75, -1.5}{
      \draw [fill=orange!55, line width=0.3mm] (\x,\y+3) rectangle (\x+0.75,\y+3.75);
      }
}

\end{scope}

\begin{scope}[shift={(17.5,-23.75)}]
\foreach \x in {0, 0.75, 1.5, 2.25} {
            \draw [fill=blue!55, line width=0.3mm] (\x,0) rectangle (\x+0.75, 0.75);       
       
    }

         \foreach \x in {0.75} { 
            
\draw [fill=green!55, line width=0.3mm] (\x+2.25,3) rectangle (\x+3, 3.75);            
            
            \draw [fill=red!55, line width=0.3mm] (\x+2.25,0.75) rectangle (\x+3, 1.5);
            
              \draw [fill=red!55, line width=0.3mm] (\x +2.25,1.5) rectangle (\x+3, 2.25);
            \draw [fill=red!55, line width=0.3mm] (\x+2.25,2.25) rectangle (\x+3, 3);
        }

    \draw [fill=red!55, line width=0.3mm] (3,0) rectangle (3.75, 0.75);
    
     \foreach \x in {0, 0.75, 1.5, 2.25} { 
     \foreach \y in {0, -0.75, -1.5, -2.25}{
      \draw [fill=orange!55, line width=0.3mm] (\x,\y+3) rectangle (\x+0.75,\y+3.75);
      }
}

\end{scope}

\end{scope}


\begin{scope}[shift={(14,0)}]


\begin{scope}[shift={(17.5,-4.25)}]
\foreach \x in {0, 0.75, 1.5, 2.25} {
            \draw [fill=blue!55, line width=0.3mm] (\x,0) rectangle (\x+0.75, 0.75);
            
   }
    
    \foreach \x in {0, 0.75, 1.5} { 
            \draw [fill=blue!55, line width=0.3mm] (\x,0.75) rectangle (\x+0.75, 1.5);

        }
        
         \foreach \x in {0, 0.75} { 
         
 \draw [fill=green!55, line width=0.3mm] (\x +2.25,3) rectangle (\x+3, 3.75);         
         
          \draw [fill=red!55, line width=0.3mm] (\x +2.25,1.5) rectangle (\x+3, 2.25);
            
             \draw [fill=red!55, line width=0.3mm] (\x+2.25,2.25) rectangle (\x+3, 3);
            \draw [fill=red!55, line width=0.3mm] (\x+2.25,0.75) rectangle (\x+3, 1.5);
        }
    
    
    \draw [fill=red!55, line width=0.3mm] (3,0) rectangle (3.75, 0.75);

    

\end{scope}

\begin{scope}[shift={(17.5,-10.75)}]
\foreach \x in {0, 0.75, 1.5, 2.25} {
            \draw [fill=blue!55, line width=0.3mm] (\x,0) rectangle (\x+0.75, 0.75);
            
    }
    
    \foreach \x in {0, 0.75, 1.5} { 
            \draw [fill=blue!55, line width=0.3mm] (\x,0.75) rectangle (\x+0.75, 1.5);

        }
        
         \foreach \x in {0, 0.75} { 
         \draw [fill=green!55, line width=0.3mm] (\x +2.25,3) rectangle (\x+3, 3.75);
         
  \draw [fill=red!55, line width=0.3mm] (\x +2.25,1.5) rectangle (\x+3, 2.25);
            \draw [fill=red!55, line width=0.3mm] (\x+2.25,2.25) rectangle (\x+3, 3);         
         
            \draw [fill=red!55, line width=0.3mm] (\x+2.25,0.75) rectangle (\x+3, 1.5);
        }
    
    \draw [fill=red!55, line width=0.3mm] (3,0) rectangle (3.75, 0.75);
    
\end{scope}

\begin{scope}[shift={(17.5,-17.25)}]
\foreach \x in {0, 0.75, 1.5, 2.25} {
            \draw [fill=blue!55, line width=0.3mm] (\x,0) rectangle (\x+0.75, 0.75);
            
    }
    
    \foreach \x in {0, 0.75, 1.5} { 
            \draw [fill=blue!55, line width=0.3mm] (\x,0.75) rectangle (\x+0.75, 1.5);

        }
        
         \foreach \x in {0, 0.75} { 
            
\draw [fill=green!55, line width=0.3mm] (\x+2.25,3) rectangle (\x+3, 3.75);            
            
            \draw [fill=red!55, line width=0.3mm] (\x+2.25,0.75) rectangle (\x+3, 1.5);
            
              \draw [fill=red!55, line width=0.3mm] (\x +2.25,1.5) rectangle (\x+3, 2.25);
            \draw [fill=red!55, line width=0.3mm] (\x+2.25,2.25) rectangle (\x+3, 3);
        }
    
    \draw [fill=red!55, line width=0.3mm] (3,0) rectangle (3.75, 0.75);
    
     \foreach \x in {0, 0.75, 1.5} { 
     \foreach \y in {0, -0.75, -1.5}{
      \draw [fill=orange!55, line width=0.3mm] (\x,\y+3) rectangle (\x+0.75,\y+3.75);
      }
}

\begin{scope}[shift={(0,6.5)}]
\foreach \x in {0, 0.75, 1.5} { 
     \foreach \y in {0, -0.75, -1.5}{
      \draw [fill=green!55, line width=0.3mm] (\x,\y+3) rectangle (\x+0.75,\y+3.75);
      }
}
\end{scope}

\begin{scope}[shift={(0,13)}]
\foreach \x in {0, 0.75, 1.5} { 
     \foreach \y in {0, -0.75, -1.5}{
      \draw [fill=green!55, line width=0.3mm] (\x,\y+3) rectangle (\x+0.75,\y+3.75);
      }
}
\end{scope}

\end{scope}

\begin{scope}[shift={(17.5,-23.75)}]
\foreach \x in {0, 0.75, 1.5, 2.25} {
            \draw [fill=blue!55, line width=0.3mm] (\x,0) rectangle (\x+0.75, 0.75);
               
    }

         \foreach \x in {0.75} { 
             \draw [fill=green!55, line width=0.3mm] (\x+2.25,3) rectangle (\x+3, 3.75); 
             
            \draw [fill=red!55, line width=0.3mm] (\x+2.25,0.75) rectangle (\x+3, 1.5);
            
              \draw [fill=red!55, line width=0.3mm] (\x +2.25,1.5) rectangle (\x+3, 2.25);
            \draw [fill=red!55, line width=0.3mm] (\x+2.25,2.25) rectangle (\x+3, 3);
        }

    \draw [fill=red!55, line width=0.3mm] (3,0) rectangle (3.75, 0.75);
    
     \foreach \x in {0, 0.75, 1.5, 2.25} { 
     \foreach \y in {0, -0.75, -1.5, -2.25}{
      \draw [fill=orange!55, line width=0.3mm] (\x,\y+3) rectangle (\x+0.75,\y+3.75);
      }
}

\end{scope}

\end{scope}


\begin{scope}[shift={(21,0)}]


\begin{scope}[shift={(17.5,-4.25)}]
\foreach \x in {0, 0.75, 1.5, 2.25} {
            \draw [fill=blue!55, line width=0.3mm] (\x,0) rectangle (\x+0.75, 0.75);
            
   }
    
    \foreach \x in {0, 0.75, 1.5} { 
            
        }

\draw [fill=green!55, line width=0.3mm] (3,3) rectangle (3.75, 3.75);        
        
         \draw [fill=red!55, line width=0.3mm] (3,1.5) rectangle (3.75, 2.25);
            
             \draw [fill=red!55, line width=0.3mm] (3,2.25) rectangle (3.75, 3);
            
            \draw [fill=red!55, line width=0.3mm] (3,0.75) rectangle (3.75, 1.5);
 
    \draw [fill=red!55, line width=0.3mm] (3,0) rectangle (3.75, 0.75);
    
    

\end{scope}

\begin{scope}[shift={(17.5,-10.75)}]
\foreach \x in {0, 0.75, 1.5, 2.25} {
            \draw [fill=blue!55, line width=0.3mm] (\x,0) rectangle (\x+0.75, 0.75);

    }
    
   \draw [fill=green!55, line width=0.3mm] (3,3) rectangle (3.75, 3.75);
        
      \draw [fill=red!55, line width=0.3mm] (3,1.5) rectangle (3.75, 2.25);
            \draw [fill=red!55, line width=0.3mm] (3,2.25) rectangle (3.75, 3);

            \draw [fill=red!55, line width=0.3mm] (3,0.75) rectangle (3.75, 1.5);

    \draw [fill=red!55, line width=0.3mm] (3,0) rectangle (3.75, 0.75);

\end{scope}

\begin{scope}[shift={(17.5,-17.25)}]
\foreach \x in {0, 0.75, 1.5, 2.25} {
            \draw [fill=blue!55, line width=0.3mm] (\x,0) rectangle (\x+0.75, 0.75);

    }
    
    \draw [fill=green!55, line width=0.3mm] (3,3) rectangle (3.75, 3.75);
  
       \draw [fill=red!55, line width=0.3mm] (3,0.75) rectangle (3.75, 1.5);
            
              \draw [fill=red!55, line width=0.3mm] (3,1.5) rectangle (3.75, 2.25);
            \draw [fill=red!55, line width=0.3mm] (3,2.25) rectangle (3.75, 3);

    \draw [fill=red!55, line width=0.3mm] (3,0) rectangle (3.75, 0.75);

\end{scope}

\begin{scope}[shift={(17.5,-23.75)}]
\foreach \x in {0, 0.75, 1.5, 2.25} {
            \draw [fill=blue!55, line width=0.3mm] (\x,0) rectangle (\x+0.75, 0.75);

    }

         \foreach \x in {0.75} { 
            
\draw [fill=green!55, line width=0.3mm] (\x+2.25,3) rectangle (\x+3, 3.75);            
            
            \draw [fill=red!55, line width=0.3mm] (\x+2.25,0.75) rectangle (\x+3, 1.5);
            
              \draw [fill=red!55, line width=0.3mm] (\x +2.25,1.5) rectangle (\x+3, 2.25);
            \draw [fill=red!55, line width=0.3mm] (\x+2.25,2.25) rectangle (\x+3, 3);
        }

    \draw [fill=red!55, line width=0.3mm] (3,0) rectangle (3.75, 0.75);
    
     \foreach \x in {0, 0.75, 1.5, 2.25} { 
     \foreach \y in {0, -0.75, -1.5, -2.25}{
      \draw [fill=orange!55, line width=0.3mm] (\x,\y+3) rectangle (\x+0.75,\y+3.75);
      }
}


\begin{scope}[shift={(0,6.5)}]
 \foreach \x in {0, 0.75, 1.5, 2.25} { 
     \foreach \y in {0, -0.75, -1.5, -2.25}{
      \draw [fill=green!55, line width=0.3mm] (\x,\y+3) rectangle (\x+0.75,\y+3.75);
      }
}
\end{scope}

\begin{scope}[shift={(0,13)}]
 \foreach \x in {0, 0.75, 1.5, 2.25} { 
     \foreach \y in {0, -0.75, -1.5, -2.25}{
      \draw [fill=green!55, line width=0.3mm] (\x,\y+3) rectangle (\x+0.75,\y+3.75);
      }
}
\end{scope}

\begin{scope}[shift={(0,19.5)}]
 \foreach \x in {0, 0.75, 1.5, 2.25} { 
     \foreach \y in {0, -0.75, -1.5, -2.25}{
      \draw [fill=green!55, line width=0.3mm] (\x,\y+3) rectangle (\x+0.75,\y+3.75);
      }
}
\end{scope}

\end{scope}

\end{scope}


 \foreach \x in {0} { 
  \foreach \y in {0} { 
\begin{scope}[shift={(17.5+\x, -4.5-\y)}] 
   
\draw[<->] (0,0) -- (3.75,0) node[midway, below, yshift=0.5mm] {\small $n+1$};
\draw[<->] (4,0.2) -- (4,3.95) node[midway, right, xshift=-0.5mm] {\small $n+1$};

\end{scope}

}
}

\draw[<->] (17,-25.5) -- (44.5,-25.5) node[midway, below, yshift=-2mm] {\LARGE $n$};

\draw[<->] (45.1,-25) -- (45.1,0) node[midway, right, xshift=1.7mm] {\LARGE $n$};
\end{tikzpicture}
\vskip .8cm
\captionof{figure}{DIY step in the 2D puzzles}
\label{2D section Nicomachus 4D Bis Bis}

\section{The mysterious $(3n^2+3n-1)$ factor} 

\restoreparindent The last section of this work is devoted to a visual explanation of the identity
$$\sum_{k=1}^n k^4=\frac{1}{5}(n^5 + \frac{5}{2}n^4+\frac{5}{3}n^3-\frac{1}{6}n) = \frac{n(n+1)(2n+1)(3n^2+3n-1)}{30}.$$

\restoreparindent We will focus in particular on the factorized form
\vskip .3cm

\begin{tikzpicture}
\draw[-, line width=0mm] (0,0) -- (0, 0);

\begin{scope}[shift={(3.5,-2)}]
    \node (frac) at (0,0) {\huge $\displaystyle \frac{n(n+1)(n+\tikz[baseline=(char.base)]\node[draw, dashed, line width=0.5pt, inner sep=4pt, rectangle](char){$\frac{1}{2}$};)\tikz[baseline=(char.base)]\node[draw, dashed, line width=0.5pt, inner sep=3.2pt, rectangle](char){$(n^2 + n -\frac{1}{3})$};}{5}$};

    
      \draw[->, line width=0.3mm] (0,1) -- (0, 1.3) node[midway, above, yshift=4pt]{How does this term appear?.};
      \draw[->, line width=0.3mm] (4,0.4) -- (4.4, 0.4) node[right,  align=left]{\hskip .2cm Why does an irreducible  \\\hskip .2cm  factor appear?};

\end{scope}
\end{tikzpicture}
\captionof{figure}{Formula for the Sum of Fourth Powers.}\label{Fourth Powers}

\vskip .2cm

\restoreparindent The denominator tells us the dimension in which the visual construction takes place, although we will only see it through its lower–dimensional sections down to 2D. Using what we have learned in the lower–dimensional cases, we will identify the 2D puzzle that has to be rearranged into identical rectangles. The number and the side lengths of these rectangles must match the factors in the numerator of the previous expression. In the next subsections we will justify each of these factors.

\subsection{Step 0: Setting Up the 2D Puzzle Through Sections}

The sum of fourth powers $1^4 +2^4+ 3^4 + \cdots + n^4$ is equal to the number of unit 5D hypercubes in a 5D pyramid. In this work, we compute this sum by taking five such 5D pyramids and assembling and adjusting them so that they form a 5D rectangular block. The total number of unit hypercubes is then just the product of the five side lengths of this block.

Since we cannot visualize beyond three dimensions, we work with lower–dimensional sections. We start by taking four of the five 5D pyramids and assembling them, keeping in mind that the fifth pyramid will later be fitted into the remaining gaps. This configuration has $n$ four–dimensional sections. The first of these sections is exactly the 4D puzzle that appeared in the case of Nicomachus’ theorem. This first 4D section, in turn, has $n$ three–dimensional sections, which are precisely the ones shown in Figure~\ref{4 pyramids 4D}. Finally, each of these $n$ three–dimensional sections can be reduced to the 2D puzzle in Figure~\ref{2D section Nicomachus 4D Bis}, where each column corresponds to one section. 

We reproduce this figure again, but now it no longer represents four assembled 4D pyramids. Instead, it represents the first 4D section of four assembled 5D pyramids.\par
\vspace{5mm}
\noindent
\begin{tikzpicture}[scale=0.5]




\begin{scope}[shift={(0,0)}]

\begin{scope}[shift={(17.5,-4.25)}]
\foreach \x in {0, 0.75, 1.5, 2.25} {
            \draw [fill=blue!55, line width=0.3mm] (\x,0) rectangle (\x+0.75, 0.75);
            
            \draw [fill=red!55, line width=0.3mm] (\x+0.75,2.25) rectangle (\x+1.5, 3);
       
    }
    
    \foreach \x in {0, 0.75, 1.5} { 
            \draw [fill=blue!55, line width=0.3mm] (\x,0.75) rectangle (\x+0.75, 1.5);
            
            \draw [fill=red!55, line width=0.3mm] (\x +1.5,1.5) rectangle (\x+2.25, 2.25);
        }
        
         \foreach \x in {0, 0.75} { 
            \draw [fill=blue!55, line width=0.3mm] (\x,1.5) rectangle (\x+0.75, 2.25);
            \draw [fill=red!55, line width=0.3mm] (\x+2.25,0.75) rectangle (\x+3, 1.5);
        }
    
    \draw [fill=blue!55, line width=0.3mm] (0,2.25) rectangle (0.75, 3);
    
    \draw [fill=red!55, line width=0.3mm] (3,0) rectangle (3.75, 0.75);

    
      \draw [fill=orange!55, line width=0.3mm] (0,3) rectangle (0.75, 3.75);


\begin{scope}[shift={(0,6.5)}]
 \draw [fill=green!55, line width=0.3mm] (0,3) rectangle (0.75, 3.75);
\end{scope}

\end{scope}

\begin{scope}[shift={(17.5,-10.75)}]
\foreach \x in {0, 0.75, 1.5, 2.25} {
            \draw [fill=blue!55, line width=0.3mm] (\x,0) rectangle (\x+0.75, 0.75);

    }
    
    \foreach \x in {0, 0.75, 1.5} { 
            \draw [fill=blue!55, line width=0.3mm] (\x,0.75) rectangle (\x+0.75, 1.5);
            
            \draw [fill=red!55, line width=0.3mm] (\x +1.5,1.5) rectangle (\x+2.25, 2.25);
            \draw [fill=red!55, line width=0.3mm] (\x+1.5,2.25) rectangle (\x+2.25, 3);
        }
        
         \foreach \x in {0, 0.75} { 
            \draw [fill=blue!55, line width=0.3mm] (\x,1.5) rectangle (\x+0.75, 2.25);
            \draw [fill=red!55, line width=0.3mm] (\x+2.25,0.75) rectangle (\x+3, 1.5);
        }

    \draw [fill=red!55, line width=0.3mm] (3,0) rectangle (3.75, 0.75);

     \foreach \x in {0, 0.75} { 
     \foreach \y in {0, -0.75}{
      \draw [fill=orange!55, line width=0.3mm] (\x,\y+3) rectangle (\x+0.75,\y+3.75);
      }
}

\end{scope}

\begin{scope}[shift={(17.5,-17.25)}]
\foreach \x in {0, 0.75, 1.5, 2.25} {
            \draw [fill=blue!55, line width=0.3mm] (\x,0) rectangle (\x+0.75, 0.75);

    }
    
    \foreach \x in {0, 0.75, 1.5} { 
            \draw [fill=blue!55, line width=0.3mm] (\x,0.75) rectangle (\x+0.75, 1.5);

        }
        
         \foreach \x in {0, 0.75} { 
            
            \draw [fill=red!55, line width=0.3mm] (\x+2.25,0.75) rectangle (\x+3, 1.5);
            
              \draw [fill=red!55, line width=0.3mm] (\x +2.25,1.5) rectangle (\x+3, 2.25);
            \draw [fill=red!55, line width=0.3mm] (\x+2.25,2.25) rectangle (\x+3, 3);
        }

    \draw [fill=red!55, line width=0.3mm] (3,0) rectangle (3.75, 0.75);

     \foreach \x in {0, 0.75, 1.5} { 
     \foreach \y in {0, -0.75, -1.5}{
      \draw [fill=orange!55, line width=0.3mm] (\x,\y+3) rectangle (\x+0.75,\y+3.75);
      }
}

\end{scope}

\begin{scope}[shift={(17.5,-23.75)}]
\foreach \x in {0, 0.75, 1.5, 2.25} {
            \draw [fill=blue!55, line width=0.3mm] (\x,0) rectangle (\x+0.75, 0.75);

    }

         \foreach \x in {0.75} { 
            
            \draw [fill=red!55, line width=0.3mm] (\x+2.25,0.75) rectangle (\x+3, 1.5);
            
              \draw [fill=red!55, line width=0.3mm] (\x +2.25,1.5) rectangle (\x+3, 2.25);
            \draw [fill=red!55, line width=0.3mm] (\x+2.25,2.25) rectangle (\x+3, 3);
        }

    \draw [fill=red!55, line width=0.3mm] (3,0) rectangle (3.75, 0.75);

     \foreach \x in {0, 0.75, 1.5, 2.25} { 
     \foreach \y in {0, -0.75, -1.5, -2.25}{
      \draw [fill=orange!55, line width=0.3mm] (\x,\y+3) rectangle (\x+0.75,\y+3.75);
      }
}

\end{scope}

\end{scope}


\begin{scope}[shift={(7,0)}]

\begin{scope}[shift={(17.5,-4.25)}]
\foreach \x in {0, 0.75, 1.5, 2.25} {
            \draw [fill=blue!55, line width=0.3mm] (\x,0) rectangle (\x+0.75, 0.75);

    }
    
    \foreach \x in {0, 0.75, 1.5} { 
            \draw [fill=blue!55, line width=0.3mm] (\x,0.75) rectangle (\x+0.75, 1.5);
            
            \draw [fill=red!55, line width=0.3mm] (\x +1.5,1.5) rectangle (\x+2.25, 2.25);
            
             \draw [fill=red!55, line width=0.3mm] (\x+1.5,2.25) rectangle (\x+2.25, 3);
        }
        
         \foreach \x in {0, 0.75} { 
            \draw [fill=blue!55, line width=0.3mm] (\x,1.5) rectangle (\x+0.75, 2.25);
            \draw [fill=red!55, line width=0.3mm] (\x+2.25,0.75) rectangle (\x+3, 1.5);
        }
    
    
    \draw [fill=red!55, line width=0.3mm] (3,0) rectangle (3.75, 0.75);

    

\end{scope}

\begin{scope}[shift={(17.5,-10.75)}]
\foreach \x in {0, 0.75, 1.5, 2.25} {
            \draw [fill=blue!55, line width=0.3mm] (\x,0) rectangle (\x+0.75, 0.75);

    }
    
    \foreach \x in {0, 0.75, 1.5} { 
            \draw [fill=blue!55, line width=0.3mm] (\x,0.75) rectangle (\x+0.75, 1.5);
            
            \draw [fill=red!55, line width=0.3mm] (\x +1.5,1.5) rectangle (\x+2.25, 2.25);
            \draw [fill=red!55, line width=0.3mm] (\x+1.5,2.25) rectangle (\x+2.25, 3);
        }
        
         \foreach \x in {0, 0.75} { 
            \draw [fill=blue!55, line width=0.3mm] (\x,1.5) rectangle (\x+0.75, 2.25);
            \draw [fill=red!55, line width=0.3mm] (\x+2.25,0.75) rectangle (\x+3, 1.5);
        }

    \draw [fill=red!55, line width=0.3mm] (3,0) rectangle (3.75, 0.75);

     \foreach \x in {0, 0.75} { 
     \foreach \y in {0, -0.75}{
      \draw [fill=orange!55, line width=0.3mm] (\x,\y+3) rectangle (\x+0.75,\y+3.75);
      }
}

\begin{scope}[shift={(0,6.5)}]
 \foreach \x in {0, 0.75} { 
     \foreach \y in {0, -0.75}{
      \draw [fill=green!55, line width=0.3mm] (\x,\y+3) rectangle (\x+0.75,\y+3.75);
      }
}
\end{scope}

\begin{scope}[shift={(0,13)}]
 \foreach \x in {0, 0.75} { 
     \foreach \y in {0, -0.75}{
      \draw [fill=green!55, line width=0.3mm] (\x,\y+3) rectangle (\x+0.75,\y+3.75);
      }
}
\end{scope}

\end{scope}

\begin{scope}[shift={(17.5,-17.25)}]
\foreach \x in {0, 0.75, 1.5, 2.25} {
            \draw [fill=blue!55, line width=0.3mm] (\x,0) rectangle (\x+0.75, 0.75);

    }
    
    \foreach \x in {0, 0.75, 1.5} { 
            \draw [fill=blue!55, line width=0.3mm] (\x,0.75) rectangle (\x+0.75, 1.5);

        }
        
         \foreach \x in {0, 0.75} { 
            
            \draw [fill=red!55, line width=0.3mm] (\x+2.25,0.75) rectangle (\x+3, 1.5);
            
              \draw [fill=red!55, line width=0.3mm] (\x +2.25,1.5) rectangle (\x+3, 2.25);
            \draw [fill=red!55, line width=0.3mm] (\x+2.25,2.25) rectangle (\x+3, 3);
        }

    \draw [fill=red!55, line width=0.3mm] (3,0) rectangle (3.75, 0.75);

     \foreach \x in {0, 0.75, 1.5} { 
     \foreach \y in {0, -0.75, -1.5}{
      \draw [fill=orange!55, line width=0.3mm] (\x,\y+3) rectangle (\x+0.75,\y+3.75);
      }
}

\end{scope}

\begin{scope}[shift={(17.5,-23.75)}]
\foreach \x in {0, 0.75, 1.5, 2.25} {
            \draw [fill=blue!55, line width=0.3mm] (\x,0) rectangle (\x+0.75, 0.75);
       
    }

         \foreach \x in {0.75} { 
            
            \draw [fill=red!55, line width=0.3mm] (\x+2.25,0.75) rectangle (\x+3, 1.5);
            
              \draw [fill=red!55, line width=0.3mm] (\x +2.25,1.5) rectangle (\x+3, 2.25);
            \draw [fill=red!55, line width=0.3mm] (\x+2.25,2.25) rectangle (\x+3, 3);
        }

    \draw [fill=red!55, line width=0.3mm] (3,0) rectangle (3.75, 0.75);
    
     \foreach \x in {0, 0.75, 1.5, 2.25} { 
     \foreach \y in {0, -0.75, -1.5, -2.25}{
      \draw [fill=orange!55, line width=0.3mm] (\x,\y+3) rectangle (\x+0.75,\y+3.75);
      }
}

\end{scope}

\end{scope}


\begin{scope}[shift={(14,0)}]


\begin{scope}[shift={(17.5,-4.25)}]
\foreach \x in {0, 0.75, 1.5, 2.25} {
            \draw [fill=blue!55, line width=0.3mm] (\x,0) rectangle (\x+0.75, 0.75);

   }
    
    \foreach \x in {0, 0.75, 1.5} { 
            \draw [fill=blue!55, line width=0.3mm] (\x,0.75) rectangle (\x+0.75, 1.5);

        }
        
         \foreach \x in {0, 0.75} { 
         
          \draw [fill=red!55, line width=0.3mm] (\x +2.25,1.5) rectangle (\x+3, 2.25);
            
             \draw [fill=red!55, line width=0.3mm] (\x+2.25,2.25) rectangle (\x+3, 3);
            \draw [fill=red!55, line width=0.3mm] (\x+2.25,0.75) rectangle (\x+3, 1.5);
        }
    
    
    \draw [fill=red!55, line width=0.3mm] (3,0) rectangle (3.75, 0.75);
    
    

\end{scope}

\begin{scope}[shift={(17.5,-10.75)}]
\foreach \x in {0, 0.75, 1.5, 2.25} {
            \draw [fill=blue!55, line width=0.3mm] (\x,0) rectangle (\x+0.75, 0.75);

    }
    
    \foreach \x in {0, 0.75, 1.5} { 
            \draw [fill=blue!55, line width=0.3mm] (\x,0.75) rectangle (\x+0.75, 1.5);

        }
        
         \foreach \x in {0, 0.75} { 
         
  \draw [fill=red!55, line width=0.3mm] (\x +2.25,1.5) rectangle (\x+3, 2.25);
            \draw [fill=red!55, line width=0.3mm] (\x+2.25,2.25) rectangle (\x+3, 3);         
         
            \draw [fill=red!55, line width=0.3mm] (\x+2.25,0.75) rectangle (\x+3, 1.5);
        }

    \draw [fill=red!55, line width=0.3mm] (3,0) rectangle (3.75, 0.75);

\end{scope}

\begin{scope}[shift={(17.5,-17.25)}]
\foreach \x in {0, 0.75, 1.5, 2.25} {
            \draw [fill=blue!55, line width=0.3mm] (\x,0) rectangle (\x+0.75, 0.75);
            
    }
    
    \foreach \x in {0, 0.75, 1.5} { 
            \draw [fill=blue!55, line width=0.3mm] (\x,0.75) rectangle (\x+0.75, 1.5);
            
        }
        
         \foreach \x in {0, 0.75} { 
            
            \draw [fill=red!55, line width=0.3mm] (\x+2.25,0.75) rectangle (\x+3, 1.5);
            
              \draw [fill=red!55, line width=0.3mm] (\x +2.25,1.5) rectangle (\x+3, 2.25);
            \draw [fill=red!55, line width=0.3mm] (\x+2.25,2.25) rectangle (\x+3, 3);
        }

    \draw [fill=red!55, line width=0.3mm] (3,0) rectangle (3.75, 0.75);
    
     \foreach \x in {0, 0.75, 1.5} { 
     \foreach \y in {0, -0.75, -1.5}{
      \draw [fill=orange!55, line width=0.3mm] (\x,\y+3) rectangle (\x+0.75,\y+3.75);
      }
}

\begin{scope}[shift={(0,6.5)}]
\foreach \x in {0, 0.75, 1.5} { 
     \foreach \y in {0, -0.75, -1.5}{
      \draw [fill=green!55, line width=0.3mm] (\x,\y+3) rectangle (\x+0.75,\y+3.75);
      }
}
\end{scope}

\begin{scope}[shift={(0,13)}]
\foreach \x in {0, 0.75, 1.5} { 
     \foreach \y in {0, -0.75, -1.5}{
      \draw [fill=green!55, line width=0.3mm] (\x,\y+3) rectangle (\x+0.75,\y+3.75);
      }
}
\end{scope}

\begin{scope}[shift={(0,19.5)}]
\foreach \x in {0, 0.75, 1.5} { 
     \foreach \y in {0, -0.75, -1.5}{
      \draw [fill=green!55, line width=0.3mm] (\x,\y+3) rectangle (\x+0.75,\y+3.75);
      }
}
\end{scope}

\end{scope}

\begin{scope}[shift={(17.5,-23.75)}]
\foreach \x in {0, 0.75, 1.5, 2.25} {
            \draw [fill=blue!55, line width=0.3mm] (\x,0) rectangle (\x+0.75, 0.75);

    }
    
         \foreach \x in {0.75} { 
            
            \draw [fill=red!55, line width=0.3mm] (\x+2.25,0.75) rectangle (\x+3, 1.5);
            
              \draw [fill=red!55, line width=0.3mm] (\x +2.25,1.5) rectangle (\x+3, 2.25);
            \draw [fill=red!55, line width=0.3mm] (\x+2.25,2.25) rectangle (\x+3, 3);
        }

    \draw [fill=red!55, line width=0.3mm] (3,0) rectangle (3.75, 0.75);
    
     \foreach \x in {0, 0.75, 1.5, 2.25} { 
     \foreach \y in {0, -0.75, -1.5, -2.25}{
      \draw [fill=orange!55, line width=0.3mm] (\x,\y+3) rectangle (\x+0.75,\y+3.75);
      }
}

\end{scope}

\end{scope}


\begin{scope}[shift={(21,0)}]


\begin{scope}[shift={(17.5,-4.25)}]
\foreach \x in {0, 0.75, 1.5, 2.25} {
            \draw [fill=blue!55, line width=0.3mm] (\x,0) rectangle (\x+0.75, 0.75);

   }
    
    \foreach \x in {0, 0.75, 1.5} { 

        }

         \draw [fill=red!55, line width=0.3mm] (3,1.5) rectangle (3.75, 2.25);
            
             \draw [fill=red!55, line width=0.3mm] (3,2.25) rectangle (3.75, 3);
            
            \draw [fill=red!55, line width=0.3mm] (3,0.75) rectangle (3.75, 1.5);
 
    \draw [fill=red!55, line width=0.3mm] (3,0) rectangle (3.75, 0.75);
    
    
\end{scope}

\begin{scope}[shift={(17.5,-10.75)}]
\foreach \x in {0, 0.75, 1.5, 2.25} {
            \draw [fill=blue!55, line width=0.3mm] (\x,0) rectangle (\x+0.75, 0.75);
        
    }

      \draw [fill=red!55, line width=0.3mm] (3,1.5) rectangle (3.75, 2.25);
            \draw [fill=red!55, line width=0.3mm] (3,2.25) rectangle (3.75, 3);

            \draw [fill=red!55, line width=0.3mm] (3,0.75) rectangle (3.75, 1.5);
    
    \draw [fill=red!55, line width=0.3mm] (3,0) rectangle (3.75, 0.75);
    
\end{scope}

\begin{scope}[shift={(17.5,-17.25)}]
\foreach \x in {0, 0.75, 1.5, 2.25} {
            \draw [fill=blue!55, line width=0.3mm] (\x,0) rectangle (\x+0.75, 0.75);
            
    }

       \draw [fill=red!55, line width=0.3mm] (3,0.75) rectangle (3.75, 1.5);
            
              \draw [fill=red!55, line width=0.3mm] (3,1.5) rectangle (3.75, 2.25);
            \draw [fill=red!55, line width=0.3mm] (3,2.25) rectangle (3.75, 3);

    \draw [fill=red!55, line width=0.3mm] (3,0) rectangle (3.75, 0.75);

\end{scope}

\begin{scope}[shift={(17.5,-23.75)}]
\foreach \x in {0, 0.75, 1.5, 2.25} {
            \draw [fill=blue!55, line width=0.3mm] (\x,0) rectangle (\x+0.75, 0.75);
            
    }
    
         \foreach \x in {0.75} { 
            
            \draw [fill=red!55, line width=0.3mm] (\x+2.25,0.75) rectangle (\x+3, 1.5);
            
              \draw [fill=red!55, line width=0.3mm] (\x +2.25,1.5) rectangle (\x+3, 2.25);
            \draw [fill=red!55, line width=0.3mm] (\x+2.25,2.25) rectangle (\x+3, 3);
        }

    \draw [fill=red!55, line width=0.3mm] (3,0) rectangle (3.75, 0.75);
    
     \foreach \x in {0, 0.75, 1.5, 2.25} { 
     \foreach \y in {0, -0.75, -1.5, -2.25}{
      \draw [fill=orange!55, line width=0.3mm] (\x,\y+3) rectangle (\x+0.75,\y+3.75);
      }
}


\begin{scope}[shift={(0,6.5)}]
 \foreach \x in {0, 0.75, 1.5, 2.25} { 
     \foreach \y in {0, -0.75, -1.5, -2.25}{
      \draw [fill=green!55, line width=0.3mm] (\x,\y+3) rectangle (\x+0.75,\y+3.75);
      }
}
\end{scope}

\begin{scope}[shift={(0,13)}]
 \foreach \x in {0, 0.75, 1.5, 2.25} { 
     \foreach \y in {0, -0.75, -1.5, -2.25}{
      \draw [fill=green!55, line width=0.3mm] (\x,\y+3) rectangle (\x+0.75,\y+3.75);
      }
}
\end{scope}

\begin{scope}[shift={(0,19.5)}]
 \foreach \x in {0, 0.75, 1.5, 2.25} { 
     \foreach \y in {0, -0.75, -1.5, -2.25}{
      \draw [fill=green!55, line width=0.3mm] (\x,\y+3) rectangle (\x+0.75,\y+3.75);
      }
}
\end{scope}

\begin{scope}[shift={(0,26)}]
 \foreach \x in {0, 0.75, 1.5, 2.25} { 
     \foreach \y in {0, -0.75, -1.5, -2.25}{
      \draw [fill=green!55, line width=0.3mm] (\x,\y+3) rectangle (\x+0.75,\y+3.75);
      }
}
\end{scope}

\end{scope}

\end{scope}

\end{tikzpicture}
\vskip .8cm

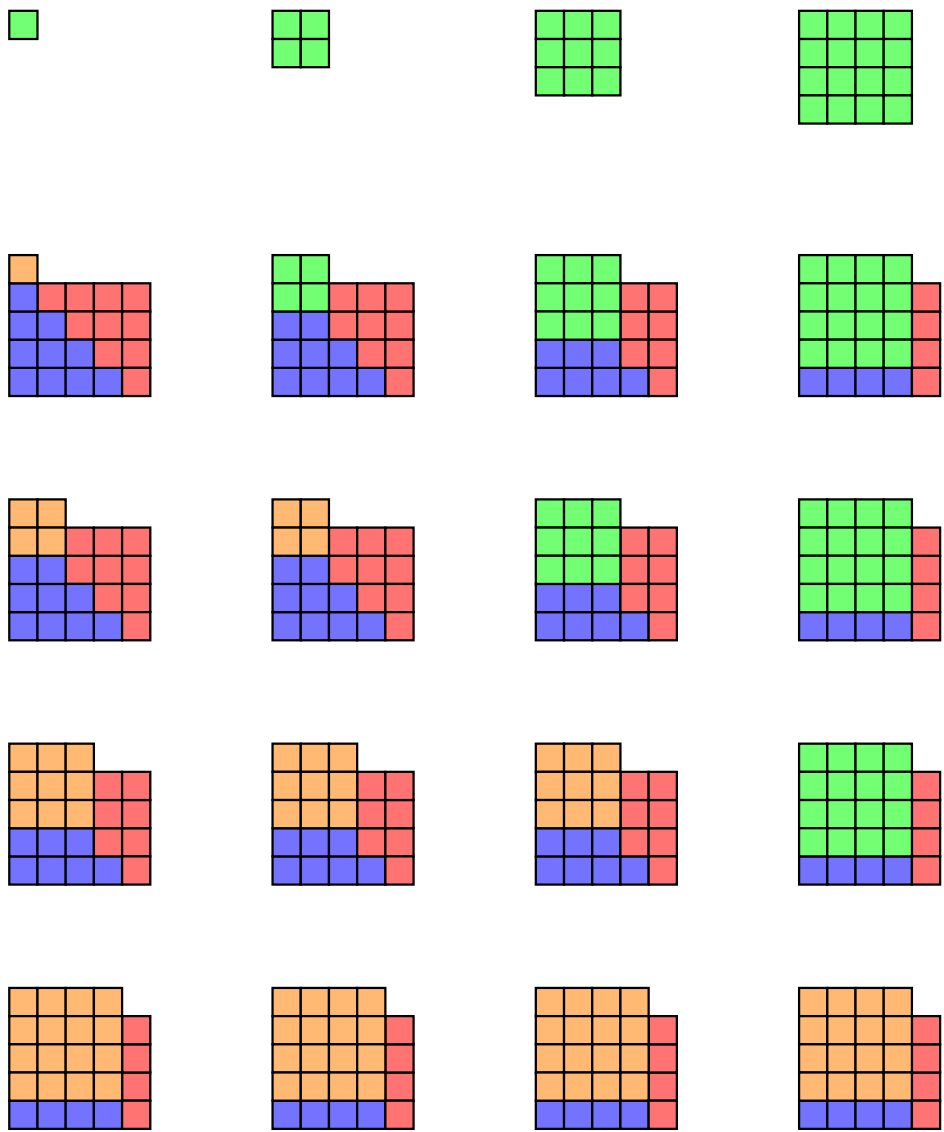
\captionof{figure}{First 4D section of four assembled 5D pyramids}
\label{First 4D Section}

\vskip .3cm

The following sections are obtained from the first one by successively removing layers. From the first section, we remove all squares of side length $1$ and the step of length $1$ in each triangle; the remaining figure is the second section. To obtain the next section, we then also remove all squares of side length 
$2$ and the steps of length $2$ in the triangles, and so on. Continuing in this way for side lengths 
$1, 2, \dots , n-1,$ the last section contains only squares of side length $n$ and steps of length $n$.\par
\vspace{5mm}
\noindent
\begin{tikzpicture}[scale=0.5]


\begin{scope}[shift={(0,0)}]

\begin{scope}[shift={(17.5,-4.25)}]
\foreach \x in {0, 0.75, 1.5, 2.25} {
            \draw [fill=blue!55, line width=0.3mm] (\x,0) rectangle (\x+0.75, 0.75);

    }
    
    \foreach \x in {0, 0.75, 1.5} { 
    
     \draw [fill=red!55, line width=0.3mm] (\x+1.5,2.25) rectangle (\x+2.25, 3);
            \draw [fill=blue!55, line width=0.3mm] (\x,0.75) rectangle (\x+0.75, 1.5);
            
            \draw [fill=red!55, line width=0.3mm] (\x +1.5,1.5) rectangle (\x+2.25, 2.25);
        }
        
         \foreach \x in {0, 0.75} { 
            \draw [fill=blue!55, line width=0.3mm] (\x,1.5) rectangle (\x+0.75, 2.25);
            \draw [fill=red!55, line width=0.3mm] (\x+2.25,0.75) rectangle (\x+3, 1.5);
        }

    \draw [fill=red!55, line width=0.3mm] (3,0) rectangle (3.75, 0.75);

    

\end{scope}

\begin{scope}[shift={(17.5,-10.75)}]
\foreach \x in {0, 0.75, 1.5, 2.25} {
            \draw [fill=blue!55, line width=0.3mm] (\x,0) rectangle (\x+0.75, 0.75);
            
    }
    
    \foreach \x in {0, 0.75, 1.5} { 
            \draw [fill=blue!55, line width=0.3mm] (\x,0.75) rectangle (\x+0.75, 1.5);
            
            \draw [fill=red!55, line width=0.3mm] (\x +1.5,1.5) rectangle (\x+2.25, 2.25);
            \draw [fill=red!55, line width=0.3mm] (\x+1.5,2.25) rectangle (\x+2.25, 3);
        }
        
         \foreach \x in {0, 0.75} { 
            \draw [fill=blue!55, line width=0.3mm] (\x,1.5) rectangle (\x+0.75, 2.25);
            \draw [fill=red!55, line width=0.3mm] (\x+2.25,0.75) rectangle (\x+3, 1.5);
        }
    
    \draw [fill=red!55, line width=0.3mm] (3,0) rectangle (3.75, 0.75);
    
     \foreach \x in {0, 0.75} { 
     \foreach \y in {0, -0.75}{
      \draw [fill=orange!55, line width=0.3mm] (\x,\y+3) rectangle (\x+0.75,\y+3.75);
      }
}

\end{scope}

\begin{scope}[shift={(17.5,-17.25)}]
\foreach \x in {0, 0.75, 1.5, 2.25} {
            \draw [fill=blue!55, line width=0.3mm] (\x,0) rectangle (\x+0.75, 0.75);
            
    }
    
    \foreach \x in {0, 0.75, 1.5} { 
            \draw [fill=blue!55, line width=0.3mm] (\x,0.75) rectangle (\x+0.75, 1.5);
            
        }
        
         \foreach \x in {0, 0.75} { 
            
            \draw [fill=red!55, line width=0.3mm] (\x+2.25,0.75) rectangle (\x+3, 1.5);
            
              \draw [fill=red!55, line width=0.3mm] (\x +2.25,1.5) rectangle (\x+3, 2.25);
            \draw [fill=red!55, line width=0.3mm] (\x+2.25,2.25) rectangle (\x+3, 3);
        }

    \draw [fill=red!55, line width=0.3mm] (3,0) rectangle (3.75, 0.75);
    
     \foreach \x in {0, 0.75, 1.5} { 
     \foreach \y in {0, -0.75, -1.5}{
      \draw [fill=orange!55, line width=0.3mm] (\x,\y+3) rectangle (\x+0.75,\y+3.75);
      }
}

\end{scope}

\begin{scope}[shift={(17.5,-23.75)}]
\foreach \x in {0, 0.75, 1.5, 2.25} {
            \draw [fill=blue!55, line width=0.3mm] (\x,0) rectangle (\x+0.75, 0.75);
            
    }
    
         \foreach \x in {0.75} { 
            
            \draw [fill=red!55, line width=0.3mm] (\x+2.25,0.75) rectangle (\x+3, 1.5);
            
              \draw [fill=red!55, line width=0.3mm] (\x +2.25,1.5) rectangle (\x+3, 2.25);
            \draw [fill=red!55, line width=0.3mm] (\x+2.25,2.25) rectangle (\x+3, 3);
        }
    
    \draw [fill=red!55, line width=0.3mm] (3,0) rectangle (3.75, 0.75);
    
     \foreach \x in {0, 0.75, 1.5, 2.25} { 
     \foreach \y in {0, -0.75, -1.5, -2.25}{
      \draw [fill=orange!55, line width=0.3mm] (\x,\y+3) rectangle (\x+0.75,\y+3.75);
      }
}

\end{scope}

\end{scope}


\begin{scope}[shift={(7,0)}]

\begin{scope}[shift={(17.5,-4.25)}]
\foreach \x in {0, 0.75, 1.5, 2.25} {
            \draw [fill=blue!55, line width=0.3mm] (\x,0) rectangle (\x+0.75, 0.75);
            
    }
    
    \foreach \x in {0, 0.75, 1.5} { 
            \draw [fill=blue!55, line width=0.3mm] (\x,0.75) rectangle (\x+0.75, 1.5);
            
            \draw [fill=red!55, line width=0.3mm] (\x +1.5,1.5) rectangle (\x+2.25, 2.25);
            
             \draw [fill=red!55, line width=0.3mm] (\x+1.5,2.25) rectangle (\x+2.25, 3);
        }
        
         \foreach \x in {0, 0.75} { 
            \draw [fill=blue!55, line width=0.3mm] (\x,1.5) rectangle (\x+0.75, 2.25);
            \draw [fill=red!55, line width=0.3mm] (\x+2.25,0.75) rectangle (\x+3, 1.5);
        }
    
    
    \draw [fill=red!55, line width=0.3mm] (3,0) rectangle (3.75, 0.75);
    
    

\end{scope}

\begin{scope}[shift={(17.5,-10.75)}]
\foreach \x in {0, 0.75, 1.5, 2.25} {
            \draw [fill=blue!55, line width=0.3mm] (\x,0) rectangle (\x+0.75, 0.75);
            
    }
    
    \foreach \x in {0, 0.75, 1.5} { 
            \draw [fill=blue!55, line width=0.3mm] (\x,0.75) rectangle (\x+0.75, 1.5);
            
            \draw [fill=red!55, line width=0.3mm] (\x +1.5,1.5) rectangle (\x+2.25, 2.25);
            \draw [fill=red!55, line width=0.3mm] (\x+1.5,2.25) rectangle (\x+2.25, 3);
        }
        
         \foreach \x in {0, 0.75} { 
            \draw [fill=blue!55, line width=0.3mm] (\x,1.5) rectangle (\x+0.75, 2.25);
            \draw [fill=red!55, line width=0.3mm] (\x+2.25,0.75) rectangle (\x+3, 1.5);
        }

    \draw [fill=red!55, line width=0.3mm] (3,0) rectangle (3.75, 0.75);
    
     \foreach \x in {0, 0.75} { 
     \foreach \y in {0, -0.75}{
      \draw [fill=orange!55, line width=0.3mm] (\x,\y+3) rectangle (\x+0.75,\y+3.75);
      }
}

\begin{scope}[shift={(0,6.5)}]
 \foreach \x in {0, 0.75} { 
     \foreach \y in {0, -0.75}{
      \draw [fill=green!55, line width=0.3mm] (\x,\y+3) rectangle (\x+0.75,\y+3.75);
      }
}
\end{scope}

\begin{scope}[shift={(0,13)}]
 \foreach \x in {0, 0.75} { 
     \foreach \y in {0, -0.75}{
      \draw [fill=green!55, line width=0.3mm] (\x,\y+3) rectangle (\x+0.75,\y+3.75);
      }
}
\end{scope}

\end{scope}

\begin{scope}[shift={(17.5,-17.25)}]
\foreach \x in {0, 0.75, 1.5, 2.25} {
            \draw [fill=blue!55, line width=0.3mm] (\x,0) rectangle (\x+0.75, 0.75);
            
    }
    
    \foreach \x in {0, 0.75, 1.5} { 
            \draw [fill=blue!55, line width=0.3mm] (\x,0.75) rectangle (\x+0.75, 1.5);
    
        }
        
         \foreach \x in {0, 0.75} { 
            
            \draw [fill=red!55, line width=0.3mm] (\x+2.25,0.75) rectangle (\x+3, 1.5);
            
              \draw [fill=red!55, line width=0.3mm] (\x +2.25,1.5) rectangle (\x+3, 2.25);
            \draw [fill=red!55, line width=0.3mm] (\x+2.25,2.25) rectangle (\x+3, 3);
        }

    \draw [fill=red!55, line width=0.3mm] (3,0) rectangle (3.75, 0.75);
    
     \foreach \x in {0, 0.75, 1.5} { 
     \foreach \y in {0, -0.75, -1.5}{
      \draw [fill=orange!55, line width=0.3mm] (\x,\y+3) rectangle (\x+0.75,\y+3.75);
      }
}

\end{scope}

\begin{scope}[shift={(17.5,-23.75)}]
\foreach \x in {0, 0.75, 1.5, 2.25} {
            \draw [fill=blue!55, line width=0.3mm] (\x,0) rectangle (\x+0.75, 0.75);
            
    }

         \foreach \x in {0.75} { 
            
            \draw [fill=red!55, line width=0.3mm] (\x+2.25,0.75) rectangle (\x+3, 1.5);
            
              \draw [fill=red!55, line width=0.3mm] (\x +2.25,1.5) rectangle (\x+3, 2.25);
            \draw [fill=red!55, line width=0.3mm] (\x+2.25,2.25) rectangle (\x+3, 3);
        }
    
    \draw [fill=red!55, line width=0.3mm] (3,0) rectangle (3.75, 0.75);
    
     \foreach \x in {0, 0.75, 1.5, 2.25} { 
     \foreach \y in {0, -0.75, -1.5, -2.25}{
      \draw [fill=orange!55, line width=0.3mm] (\x,\y+3) rectangle (\x+0.75,\y+3.75);
      }
}

\end{scope}

\end{scope}


\begin{scope}[shift={(14,0)}]


\begin{scope}[shift={(17.5,-4.25)}]
\foreach \x in {0, 0.75, 1.5, 2.25} {
            \draw [fill=blue!55, line width=0.3mm] (\x,0) rectangle (\x+0.75, 0.75);
            
   }
    
    \foreach \x in {0, 0.75, 1.5} { 
            \draw [fill=blue!55, line width=0.3mm] (\x,0.75) rectangle (\x+0.75, 1.5);
            
        }
        
         \foreach \x in {0, 0.75} { 
         
          \draw [fill=red!55, line width=0.3mm] (\x +2.25,1.5) rectangle (\x+3, 2.25);
            
             \draw [fill=red!55, line width=0.3mm] (\x+2.25,2.25) rectangle (\x+3, 3);
            \draw [fill=red!55, line width=0.3mm] (\x+2.25,0.75) rectangle (\x+3, 1.5);
        }
    
    
    \draw [fill=red!55, line width=0.3mm] (3,0) rectangle (3.75, 0.75);
    
    

\end{scope}

\begin{scope}[shift={(17.5,-10.75)}]
\foreach \x in {0, 0.75, 1.5, 2.25} {
            \draw [fill=blue!55, line width=0.3mm] (\x,0) rectangle (\x+0.75, 0.75);
            
    }
    
    \foreach \x in {0, 0.75, 1.5} { 
            \draw [fill=blue!55, line width=0.3mm] (\x,0.75) rectangle (\x+0.75, 1.5);
            
        }
        
         \foreach \x in {0, 0.75} { 
         
  \draw [fill=red!55, line width=0.3mm] (\x +2.25,1.5) rectangle (\x+3, 2.25);
            \draw [fill=red!55, line width=0.3mm] (\x+2.25,2.25) rectangle (\x+3, 3);         
         
            \draw [fill=red!55, line width=0.3mm] (\x+2.25,0.75) rectangle (\x+3, 1.5);
        }
    
    \draw [fill=red!55, line width=0.3mm] (3,0) rectangle (3.75, 0.75);
    
\end{scope}

\begin{scope}[shift={(17.5,-17.25)}]
\foreach \x in {0, 0.75, 1.5, 2.25} {
            \draw [fill=blue!55, line width=0.3mm] (\x,0) rectangle (\x+0.75, 0.75);
            
    }
    
    \foreach \x in {0, 0.75, 1.5} { 
            \draw [fill=blue!55, line width=0.3mm] (\x,0.75) rectangle (\x+0.75, 1.5);
            
        }
        
         \foreach \x in {0, 0.75} { 
            
            \draw [fill=red!55, line width=0.3mm] (\x+2.25,0.75) rectangle (\x+3, 1.5);
            
              \draw [fill=red!55, line width=0.3mm] (\x +2.25,1.5) rectangle (\x+3, 2.25);
            \draw [fill=red!55, line width=0.3mm] (\x+2.25,2.25) rectangle (\x+3, 3);
        }
    
    \draw [fill=red!55, line width=0.3mm] (3,0) rectangle (3.75, 0.75);
    
     \foreach \x in {0, 0.75, 1.5} { 
     \foreach \y in {0, -0.75, -1.5}{
      \draw [fill=orange!55, line width=0.3mm] (\x,\y+3) rectangle (\x+0.75,\y+3.75);
      }
}

\begin{scope}[shift={(0,6.5)}]
\foreach \x in {0, 0.75, 1.5} { 
     \foreach \y in {0, -0.75, -1.5}{
      \draw [fill=green!55, line width=0.3mm] (\x,\y+3) rectangle (\x+0.75,\y+3.75);
      }
}
\end{scope}

\begin{scope}[shift={(0,13)}]
\foreach \x in {0, 0.75, 1.5} { 
     \foreach \y in {0, -0.75, -1.5}{
      \draw [fill=green!55, line width=0.3mm] (\x,\y+3) rectangle (\x+0.75,\y+3.75);
      }
}
\end{scope}

\begin{scope}[shift={(0,19.5)}]
\foreach \x in {0, 0.75, 1.5} { 
     \foreach \y in {0, -0.75, -1.5}{
      \draw [fill=green!55, line width=0.3mm] (\x,\y+3) rectangle (\x+0.75,\y+3.75);
      }
}
\end{scope}

\end{scope}

\begin{scope}[shift={(17.5,-23.75)}]
\foreach \x in {0, 0.75, 1.5, 2.25} {
            \draw [fill=blue!55, line width=0.3mm] (\x,0) rectangle (\x+0.75, 0.75);
       
    }
    
         \foreach \x in {0.75} { 
            
            \draw [fill=red!55, line width=0.3mm] (\x+2.25,0.75) rectangle (\x+3, 1.5);
            
              \draw [fill=red!55, line width=0.3mm] (\x +2.25,1.5) rectangle (\x+3, 2.25);
            \draw [fill=red!55, line width=0.3mm] (\x+2.25,2.25) rectangle (\x+3, 3);
        }
    
    \draw [fill=red!55, line width=0.3mm] (3,0) rectangle (3.75, 0.75);
    
     \foreach \x in {0, 0.75, 1.5, 2.25} { 
     \foreach \y in {0, -0.75, -1.5, -2.25}{
      \draw [fill=orange!55, line width=0.3mm] (\x,\y+3) rectangle (\x+0.75,\y+3.75);
      }
}

\end{scope}

\end{scope}


\begin{scope}[shift={(21,0)}]


\begin{scope}[shift={(17.5,-4.25)}]
\foreach \x in {0, 0.75, 1.5, 2.25} {
            \draw [fill=blue!55, line width=0.3mm] (\x,0) rectangle (\x+0.75, 0.75);
            
   }
    
    \foreach \x in {0, 0.75, 1.5} { 
            
        }
        
         \draw [fill=red!55, line width=0.3mm] (3,1.5) rectangle (3.75, 2.25);
            
             \draw [fill=red!55, line width=0.3mm] (3,2.25) rectangle (3.75, 3);
            
            \draw [fill=red!55, line width=0.3mm] (3,0.75) rectangle (3.75, 1.5);
 
    \draw [fill=red!55, line width=0.3mm] (3,0) rectangle (3.75, 0.75);
    
    
     
\end{scope}

\begin{scope}[shift={(17.5,-10.75)}]
\foreach \x in {0, 0.75, 1.5, 2.25} {
            \draw [fill=blue!55, line width=0.3mm] (\x,0) rectangle (\x+0.75, 0.75);

    }
    
      \draw [fill=red!55, line width=0.3mm] (3,1.5) rectangle (3.75, 2.25);
            \draw [fill=red!55, line width=0.3mm] (3,2.25) rectangle (3.75, 3);

            \draw [fill=red!55, line width=0.3mm] (3,0.75) rectangle (3.75, 1.5);

    \draw [fill=red!55, line width=0.3mm] (3,0) rectangle (3.75, 0.75);
    
\end{scope}

\begin{scope}[shift={(17.5,-17.25)}]
\foreach \x in {0, 0.75, 1.5, 2.25} {
            \draw [fill=blue!55, line width=0.3mm] (\x,0) rectangle (\x+0.75, 0.75);
            
    }

       \draw [fill=red!55, line width=0.3mm] (3,0.75) rectangle (3.75, 1.5);
            
              \draw [fill=red!55, line width=0.3mm] (3,1.5) rectangle (3.75, 2.25);
            \draw [fill=red!55, line width=0.3mm] (3,2.25) rectangle (3.75, 3);

    \draw [fill=red!55, line width=0.3mm] (3,0) rectangle (3.75, 0.75);

\end{scope}

\begin{scope}[shift={(17.5,-23.75)}]
\foreach \x in {0, 0.75, 1.5, 2.25} {
            \draw [fill=blue!55, line width=0.3mm] (\x,0) rectangle (\x+0.75, 0.75);
            
    }
    
         \foreach \x in {0.75} { 
            
            \draw [fill=red!55, line width=0.3mm] (\x+2.25,0.75) rectangle (\x+3, 1.5);
            
              \draw [fill=red!55, line width=0.3mm] (\x +2.25,1.5) rectangle (\x+3, 2.25);
            \draw [fill=red!55, line width=0.3mm] (\x+2.25,2.25) rectangle (\x+3, 3);
        }

    \draw [fill=red!55, line width=0.3mm] (3,0) rectangle (3.75, 0.75);
    
     \foreach \x in {0, 0.75, 1.5, 2.25} { 
     \foreach \y in {0, -0.75, -1.5, -2.25}{
      \draw [fill=orange!55, line width=0.3mm] (\x,\y+3) rectangle (\x+0.75,\y+3.75);
      }
}


\begin{scope}[shift={(0,6.5)}]
 \foreach \x in {0, 0.75, 1.5, 2.25} { 
     \foreach \y in {0, -0.75, -1.5, -2.25}{
      \draw [fill=green!55, line width=0.3mm] (\x,\y+3) rectangle (\x+0.75,\y+3.75);
      }
}
\end{scope}

\begin{scope}[shift={(0,13)}]
 \foreach \x in {0, 0.75, 1.5, 2.25} { 
     \foreach \y in {0, -0.75, -1.5, -2.25}{
      \draw [fill=green!55, line width=0.3mm] (\x,\y+3) rectangle (\x+0.75,\y+3.75);
      }
}
\end{scope}

\begin{scope}[shift={(0,19.5)}]
 \foreach \x in {0, 0.75, 1.5, 2.25} { 
     \foreach \y in {0, -0.75, -1.5, -2.25}{
      \draw [fill=green!55, line width=0.3mm] (\x,\y+3) rectangle (\x+0.75,\y+3.75);
      }
}
\end{scope}

\begin{scope}[shift={(0,26)}]
 \foreach \x in {0, 0.75, 1.5, 2.25} { 
     \foreach \y in {0, -0.75, -1.5, -2.25}{
      \draw [fill=green!55, line width=0.3mm] (\x,\y+3) rectangle (\x+0.75,\y+3.75);
      }
}
\end{scope}

\end{scope}

\end{scope}

\end{tikzpicture}
\vskip .8cm

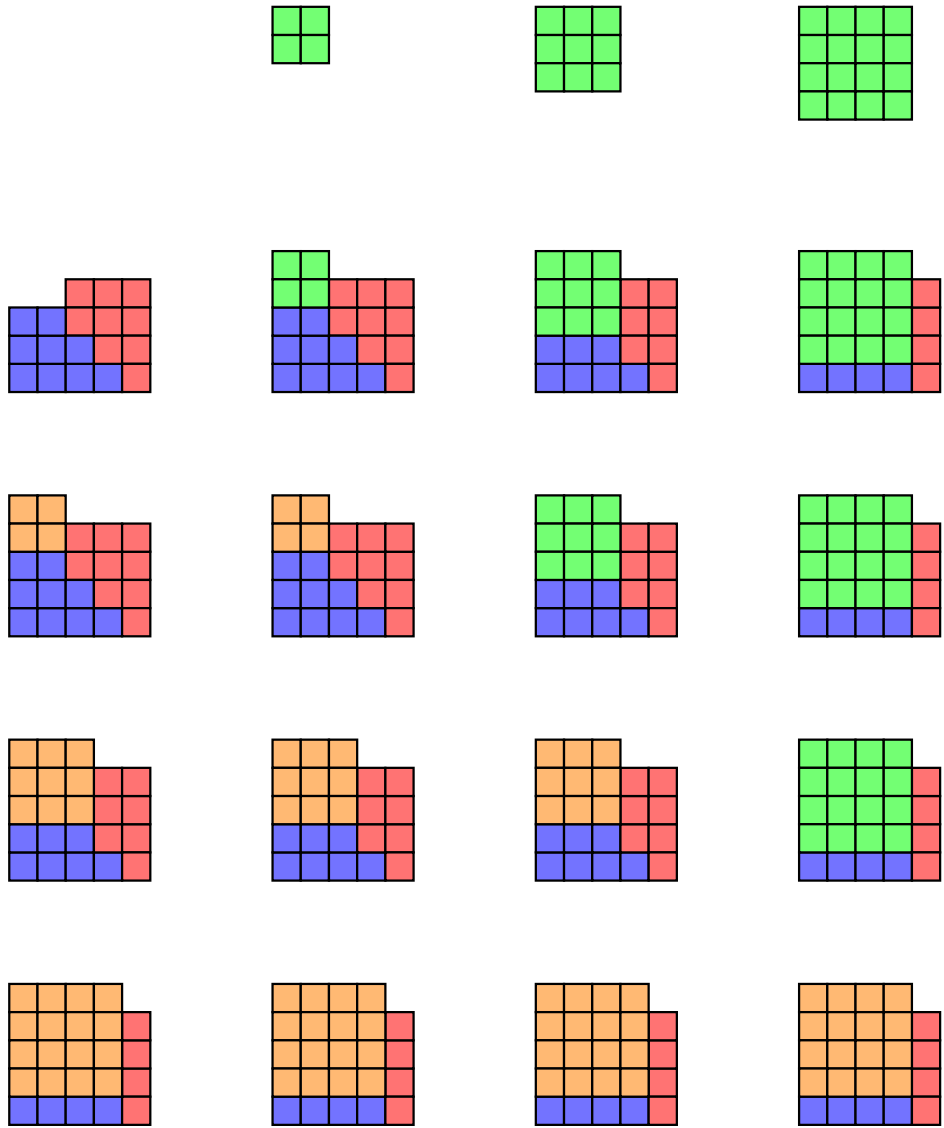
\captionof{figure}{Second 4D section of four assembled 5D pyramids}
\label{Second 4D Section}
\par
\vspace{5mm}
\noindent
\begin{tikzpicture}[scale=0.5]



\begin{scope}[shift={(0,0)}]

\begin{scope}[shift={(17.5,-4.25)}]
\foreach \x in {0, 0.75, 1.5, 2.25} {
            \draw [fill=blue!55, line width=0.3mm] (\x,0) rectangle (\x+0.75, 0.75);

    }
    
    \foreach \x in {0, 0.75, 1.5} {

            \draw [fill=blue!55, line width=0.3mm] (\x,0.75) rectangle (\x+0.75, 1.5);
           
        }
        
         \foreach \x in {0, 0.75} { 
         
          \draw [fill=red!55, line width=0.3mm] (\x+2.25,2.25) rectangle (\x+3, 3);
            \draw [fill=red!55, line width=0.3mm] (\x +2.25,1.5) rectangle (\x+3, 2.25);
            
            \draw [fill=red!55, line width=0.3mm] (\x+2.25,0.75) rectangle (\x+3, 1.5);
        }

    \draw [fill=red!55, line width=0.3mm] (3,0) rectangle (3.75, 0.75);
    
    

\end{scope}

\begin{scope}[shift={(17.5,-10.75)}]
\foreach \x in {0, 0.75, 1.5, 2.25} {
            \draw [fill=blue!55, line width=0.3mm] (\x,0) rectangle (\x+0.75, 0.75);
            
    }
    
    \foreach \x in {0, 0.75, 1.5} { 
            \draw [fill=blue!55, line width=0.3mm] (\x,0.75) rectangle (\x+0.75, 1.5);
            
        }
        
         \foreach \x in {0, 0.75} { 
         
         \draw [fill=red!55, line width=0.3mm] (\x +2.25,1.5) rectangle (\x+3, 2.25);
            \draw [fill=red!55, line width=0.3mm] (\x+2.25,2.25) rectangle (\x+3, 3);
            
            \draw [fill=red!55, line width=0.3mm] (\x+2.25,0.75) rectangle (\x+3, 1.5);
        }
    
    \draw [fill=red!55, line width=0.3mm] (3,0) rectangle (3.75, 0.75);

\end{scope}

\begin{scope}[shift={(17.5,-17.25)}]
\foreach \x in {0, 0.75, 1.5, 2.25} {
            \draw [fill=blue!55, line width=0.3mm] (\x,0) rectangle (\x+0.75, 0.75);
            
    }
    
    \foreach \x in {0, 0.75, 1.5} { 
            \draw [fill=blue!55, line width=0.3mm] (\x,0.75) rectangle (\x+0.75, 1.5);

        }
        
         \foreach \x in {0, 0.75} { 
            
            \draw [fill=red!55, line width=0.3mm] (\x+2.25,0.75) rectangle (\x+3, 1.5);
            
              \draw [fill=red!55, line width=0.3mm] (\x +2.25,1.5) rectangle (\x+3, 2.25);
            \draw [fill=red!55, line width=0.3mm] (\x+2.25,2.25) rectangle (\x+3, 3);
        }

    \draw [fill=red!55, line width=0.3mm] (3,0) rectangle (3.75, 0.75);

     \foreach \x in {0, 0.75, 1.5} { 
     \foreach \y in {0, -0.75, -1.5}{
      \draw [fill=orange!55, line width=0.3mm] (\x,\y+3) rectangle (\x+0.75,\y+3.75);
      }
}

\end{scope}

\begin{scope}[shift={(17.5,-23.75)}]
\foreach \x in {0, 0.75, 1.5, 2.25} {
            \draw [fill=blue!55, line width=0.3mm] (\x,0) rectangle (\x+0.75, 0.75);
            
    }
    
         \foreach \x in {0.75} { 
            
            \draw [fill=red!55, line width=0.3mm] (\x+2.25,0.75) rectangle (\x+3, 1.5);
            
              \draw [fill=red!55, line width=0.3mm] (\x +2.25,1.5) rectangle (\x+3, 2.25);
            \draw [fill=red!55, line width=0.3mm] (\x+2.25,2.25) rectangle (\x+3, 3);
        }

    \draw [fill=red!55, line width=0.3mm] (3,0) rectangle (3.75, 0.75);
    
     \foreach \x in {0, 0.75, 1.5, 2.25} { 
     \foreach \y in {0, -0.75, -1.5, -2.25}{
      \draw [fill=orange!55, line width=0.3mm] (\x,\y+3) rectangle (\x+0.75,\y+3.75);
      }
}

\end{scope}

\end{scope}


\begin{scope}[shift={(7,0)}]

\begin{scope}[shift={(17.5,-4.25)}]
\foreach \x in {0, 0.75, 1.5, 2.25} {
            \draw [fill=blue!55, line width=0.3mm] (\x,0) rectangle (\x+0.75, 0.75);

    }
    
    \foreach \x in {0, 0.75, 1.5} { 
            \draw [fill=blue!55, line width=0.3mm] (\x,0.75) rectangle (\x+0.75, 1.5);

        }
        
         \foreach \x in {0, 0.75} { 
            \draw [fill=red!55, line width=0.3mm] (\x +2.25,1.5) rectangle (\x+3, 2.25);
            
             \draw [fill=red!55, line width=0.3mm] (\x+2.25,2.25) rectangle (\x+3, 3);
            \draw [fill=red!55, line width=0.3mm] (\x+2.25,0.75) rectangle (\x+3, 1.5);
        }
    
    
    \draw [fill=red!55, line width=0.3mm] (3,0) rectangle (3.75, 0.75);
    
    

\end{scope}

\begin{scope}[shift={(17.5,-10.75)}]
\foreach \x in {0, 0.75, 1.5, 2.25} {
            \draw [fill=blue!55, line width=0.3mm] (\x,0) rectangle (\x+0.75, 0.75);
            
    }
    
    \foreach \x in {0, 0.75, 1.5} { 
            \draw [fill=blue!55, line width=0.3mm] (\x,0.75) rectangle (\x+0.75, 1.5);
            
        }
        
         \foreach \x in {0, 0.75} { 
             \draw [fill=red!55, line width=0.3mm] (\x +2.25,1.5) rectangle (\x+3, 2.25);
            \draw [fill=red!55, line width=0.3mm] (\x+2.25,2.25) rectangle (\x+3, 3);
            \draw [fill=red!55, line width=0.3mm] (\x+2.25,0.75) rectangle (\x+3, 1.5);
        }

    \draw [fill=red!55, line width=0.3mm] (3,0) rectangle (3.75, 0.75);
    

\end{scope}

\begin{scope}[shift={(17.5,-17.25)}]
\foreach \x in {0, 0.75, 1.5, 2.25} {
            \draw [fill=blue!55, line width=0.3mm] (\x,0) rectangle (\x+0.75, 0.75);
       
    }
    
    \foreach \x in {0, 0.75, 1.5} { 
            \draw [fill=blue!55, line width=0.3mm] (\x,0.75) rectangle (\x+0.75, 1.5);

        }
        
         \foreach \x in {0, 0.75} { 
            
            \draw [fill=red!55, line width=0.3mm] (\x+2.25,0.75) rectangle (\x+3, 1.5);
            
              \draw [fill=red!55, line width=0.3mm] (\x +2.25,1.5) rectangle (\x+3, 2.25);
            \draw [fill=red!55, line width=0.3mm] (\x+2.25,2.25) rectangle (\x+3, 3);
        }

    \draw [fill=red!55, line width=0.3mm] (3,0) rectangle (3.75, 0.75);
    
     \foreach \x in {0, 0.75, 1.5} { 
     \foreach \y in {0, -0.75, -1.5}{
      \draw [fill=orange!55, line width=0.3mm] (\x,\y+3) rectangle (\x+0.75,\y+3.75);
      }
}

\end{scope}

\begin{scope}[shift={(17.5,-23.75)}]
\foreach \x in {0, 0.75, 1.5, 2.25} {
            \draw [fill=blue!55, line width=0.3mm] (\x,0) rectangle (\x+0.75, 0.75);
            
    }
    
         \foreach \x in {0.75} { 
            
            \draw [fill=red!55, line width=0.3mm] (\x+2.25,0.75) rectangle (\x+3, 1.5);
            
              \draw [fill=red!55, line width=0.3mm] (\x +2.25,1.5) rectangle (\x+3, 2.25);
            \draw [fill=red!55, line width=0.3mm] (\x+2.25,2.25) rectangle (\x+3, 3);
        }
    
    \draw [fill=red!55, line width=0.3mm] (3,0) rectangle (3.75, 0.75);
    
     \foreach \x in {0, 0.75, 1.5, 2.25} { 
     \foreach \y in {0, -0.75, -1.5, -2.25}{
      \draw [fill=orange!55, line width=0.3mm] (\x,\y+3) rectangle (\x+0.75,\y+3.75);
      }
}

\end{scope}

\end{scope}


\begin{scope}[shift={(14,0)}]


\begin{scope}[shift={(17.5,-4.25)}]
\foreach \x in {0, 0.75, 1.5, 2.25} {
            \draw [fill=blue!55, line width=0.3mm] (\x,0) rectangle (\x+0.75, 0.75);
            
   }
    
    \foreach \x in {0, 0.75, 1.5} { 
            \draw [fill=blue!55, line width=0.3mm] (\x,0.75) rectangle (\x+0.75, 1.5);
            
        }
        
         \foreach \x in {0, 0.75} { 
         
          \draw [fill=red!55, line width=0.3mm] (\x +2.25,1.5) rectangle (\x+3, 2.25);
            
             \draw [fill=red!55, line width=0.3mm] (\x+2.25,2.25) rectangle (\x+3, 3);
            \draw [fill=red!55, line width=0.3mm] (\x+2.25,0.75) rectangle (\x+3, 1.5);
        }
    
    
    \draw [fill=red!55, line width=0.3mm] (3,0) rectangle (3.75, 0.75);
    
    

\end{scope}

\begin{scope}[shift={(17.5,-10.75)}]
\foreach \x in {0, 0.75, 1.5, 2.25} {
            \draw [fill=blue!55, line width=0.3mm] (\x,0) rectangle (\x+0.75, 0.75);
            
    }
    
    \foreach \x in {0, 0.75, 1.5} { 
            \draw [fill=blue!55, line width=0.3mm] (\x,0.75) rectangle (\x+0.75, 1.5);
            
        }
        
         \foreach \x in {0, 0.75} { 
         
  \draw [fill=red!55, line width=0.3mm] (\x +2.25,1.5) rectangle (\x+3, 2.25);
            \draw [fill=red!55, line width=0.3mm] (\x+2.25,2.25) rectangle (\x+3, 3);         
         
            \draw [fill=red!55, line width=0.3mm] (\x+2.25,0.75) rectangle (\x+3, 1.5);
        }
    
    \draw [fill=red!55, line width=0.3mm] (3,0) rectangle (3.75, 0.75);
    
\end{scope}

\begin{scope}[shift={(17.5,-17.25)}]
\foreach \x in {0, 0.75, 1.5, 2.25} {
            \draw [fill=blue!55, line width=0.3mm] (\x,0) rectangle (\x+0.75, 0.75);
            
    }
    
    \foreach \x in {0, 0.75, 1.5} { 
            \draw [fill=blue!55, line width=0.3mm] (\x,0.75) rectangle (\x+0.75, 1.5);
            
        }
        
         \foreach \x in {0, 0.75} { 
            
            \draw [fill=red!55, line width=0.3mm] (\x+2.25,0.75) rectangle (\x+3, 1.5);
            
              \draw [fill=red!55, line width=0.3mm] (\x +2.25,1.5) rectangle (\x+3, 2.25);
            \draw [fill=red!55, line width=0.3mm] (\x+2.25,2.25) rectangle (\x+3, 3);
        }

    \draw [fill=red!55, line width=0.3mm] (3,0) rectangle (3.75, 0.75);
    
     \foreach \x in {0, 0.75, 1.5} { 
     \foreach \y in {0, -0.75, -1.5}{
      \draw [fill=orange!55, line width=0.3mm] (\x,\y+3) rectangle (\x+0.75,\y+3.75);
      }
}

\begin{scope}[shift={(0,6.5)}]
\foreach \x in {0, 0.75, 1.5} { 
     \foreach \y in {0, -0.75, -1.5}{
      \draw [fill=green!55, line width=0.3mm] (\x,\y+3) rectangle (\x+0.75,\y+3.75);
      }
}
\end{scope}

\begin{scope}[shift={(0,13)}]
\foreach \x in {0, 0.75, 1.5} { 
     \foreach \y in {0, -0.75, -1.5}{
      \draw [fill=green!55, line width=0.3mm] (\x,\y+3) rectangle (\x+0.75,\y+3.75);
      }
}
\end{scope}

\begin{scope}[shift={(0,19.5)}]
\foreach \x in {0, 0.75, 1.5} { 
     \foreach \y in {0, -0.75, -1.5}{
      \draw [fill=green!55, line width=0.3mm] (\x,\y+3) rectangle (\x+0.75,\y+3.75);
      }
}
\end{scope}

\end{scope}

\begin{scope}[shift={(17.5,-23.75)}]
\foreach \x in {0, 0.75, 1.5, 2.25} {
            \draw [fill=blue!55, line width=0.3mm] (\x,0) rectangle (\x+0.75, 0.75);
            
    }
    
         \foreach \x in {0.75} { 
            
            \draw [fill=red!55, line width=0.3mm] (\x+2.25,0.75) rectangle (\x+3, 1.5);
            
              \draw [fill=red!55, line width=0.3mm] (\x +2.25,1.5) rectangle (\x+3, 2.25);
            \draw [fill=red!55, line width=0.3mm] (\x+2.25,2.25) rectangle (\x+3, 3);
        }

    \draw [fill=red!55, line width=0.3mm] (3,0) rectangle (3.75, 0.75);
    
     \foreach \x in {0, 0.75, 1.5, 2.25} { 
     \foreach \y in {0, -0.75, -1.5, -2.25}{
      \draw [fill=orange!55, line width=0.3mm] (\x,\y+3) rectangle (\x+0.75,\y+3.75);
      }
}

\end{scope}
\end{scope}


\begin{scope}[shift={(21,0)}]


\begin{scope}[shift={(17.5,-4.25)}]
\foreach \x in {0, 0.75, 1.5, 2.25} {
            \draw [fill=blue!55, line width=0.3mm] (\x,0) rectangle (\x+0.75, 0.75);
            
   }
    
    \foreach \x in {0, 0.75, 1.5} { 

        }

         \draw [fill=red!55, line width=0.3mm] (3,1.5) rectangle (3.75, 2.25);
            
             \draw [fill=red!55, line width=0.3mm] (3,2.25) rectangle (3.75, 3);
            
            \draw [fill=red!55, line width=0.3mm] (3,0.75) rectangle (3.75, 1.5);
 
    \draw [fill=red!55, line width=0.3mm] (3,0) rectangle (3.75, 0.75);
    
    

\end{scope}

\begin{scope}[shift={(17.5,-10.75)}]
\foreach \x in {0, 0.75, 1.5, 2.25} {
            \draw [fill=blue!55, line width=0.3mm] (\x,0) rectangle (\x+0.75, 0.75);
       
    }
    
      \draw [fill=red!55, line width=0.3mm] (3,1.5) rectangle (3.75, 2.25);
            \draw [fill=red!55, line width=0.3mm] (3,2.25) rectangle (3.75, 3);

            \draw [fill=red!55, line width=0.3mm] (3,0.75) rectangle (3.75, 1.5);

    \draw [fill=red!55, line width=0.3mm] (3,0) rectangle (3.75, 0.75);
    
\end{scope}

\begin{scope}[shift={(17.5,-17.25)}]
\foreach \x in {0, 0.75, 1.5, 2.25} {
            \draw [fill=blue!55, line width=0.3mm] (\x,0) rectangle (\x+0.75, 0.75);
            
    }
    
       \draw [fill=red!55, line width=0.3mm] (3,0.75) rectangle (3.75, 1.5);
            
              \draw [fill=red!55, line width=0.3mm] (3,1.5) rectangle (3.75, 2.25);
            \draw [fill=red!55, line width=0.3mm] (3,2.25) rectangle (3.75, 3);

    \draw [fill=red!55, line width=0.3mm] (3,0) rectangle (3.75, 0.75);
    
\end{scope}

\begin{scope}[shift={(17.5,-23.75)}]
\foreach \x in {0, 0.75, 1.5, 2.25} {
            \draw [fill=blue!55, line width=0.3mm] (\x,0) rectangle (\x+0.75, 0.75);

    }
    
         \foreach \x in {0.75} { 
            
            \draw [fill=red!55, line width=0.3mm] (\x+2.25,0.75) rectangle (\x+3, 1.5);
            
              \draw [fill=red!55, line width=0.3mm] (\x +2.25,1.5) rectangle (\x+3, 2.25);
            \draw [fill=red!55, line width=0.3mm] (\x+2.25,2.25) rectangle (\x+3, 3);
        }

    \draw [fill=red!55, line width=0.3mm] (3,0) rectangle (3.75, 0.75);
    
     \foreach \x in {0, 0.75, 1.5, 2.25} { 
     \foreach \y in {0, -0.75, -1.5, -2.25}{
      \draw [fill=orange!55, line width=0.3mm] (\x,\y+3) rectangle (\x+0.75,\y+3.75);
      }
}


\begin{scope}[shift={(0,6.5)}]
 \foreach \x in {0, 0.75, 1.5, 2.25} { 
     \foreach \y in {0, -0.75, -1.5, -2.25}{
      \draw [fill=green!55, line width=0.3mm] (\x,\y+3) rectangle (\x+0.75,\y+3.75);
      }
}
\end{scope}

\begin{scope}[shift={(0,13)}]
 \foreach \x in {0, 0.75, 1.5, 2.25} { 
     \foreach \y in {0, -0.75, -1.5, -2.25}{
      \draw [fill=green!55, line width=0.3mm] (\x,\y+3) rectangle (\x+0.75,\y+3.75);
      }
}
\end{scope}

\begin{scope}[shift={(0,19.5)}]
 \foreach \x in {0, 0.75, 1.5, 2.25} { 
     \foreach \y in {0, -0.75, -1.5, -2.25}{
      \draw [fill=green!55, line width=0.3mm] (\x,\y+3) rectangle (\x+0.75,\y+3.75);
      }
}
\end{scope}

\begin{scope}[shift={(0,26)}]
 \foreach \x in {0, 0.75, 1.5, 2.25} { 
     \foreach \y in {0, -0.75, -1.5, -2.25}{
      \draw [fill=green!55, line width=0.3mm] (\x,\y+3) rectangle (\x+0.75,\y+3.75);
      }
}
\end{scope}

\end{scope}

\end{scope}

\end{tikzpicture}
\vskip .8cm

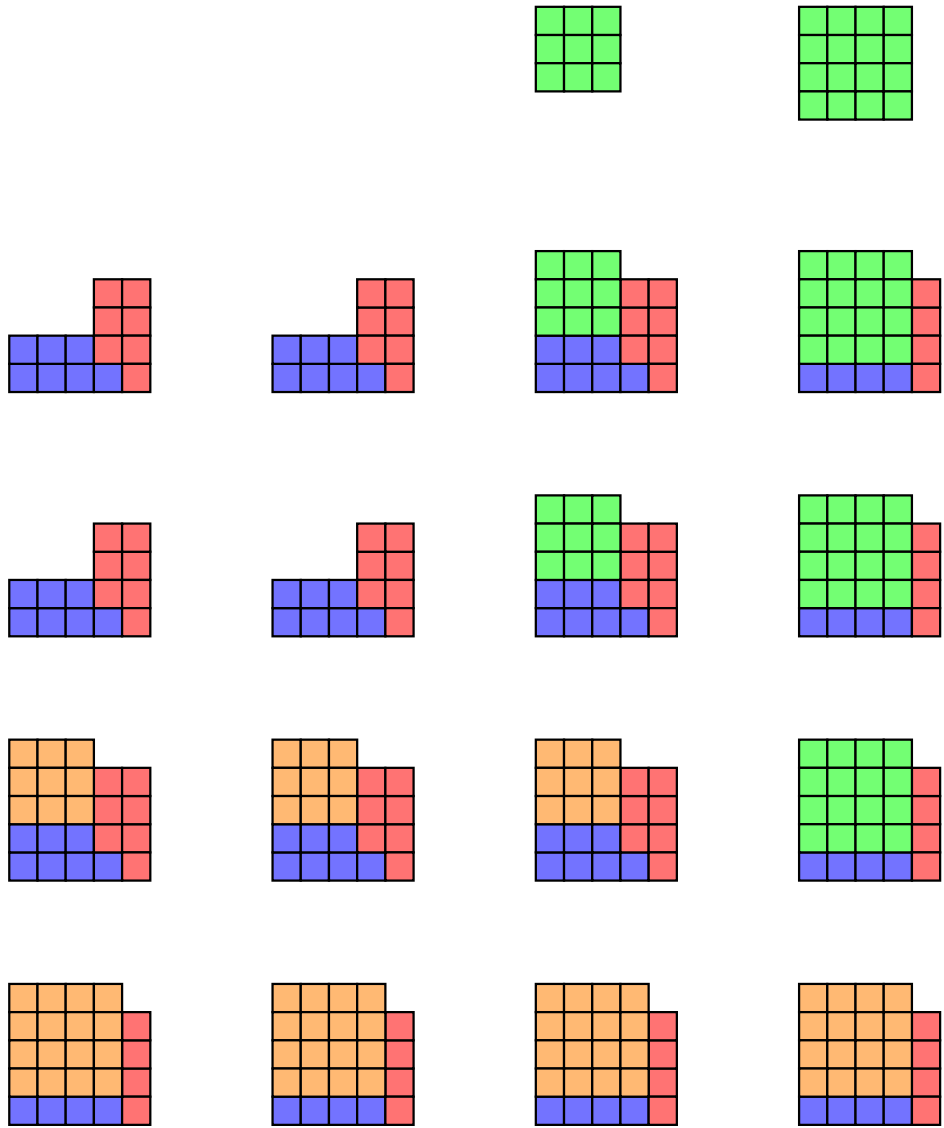
\captionof{figure}{Third 4D section of four assembled 5D pyramids}
\label{Third 4D Section}
\par
\vspace{5mm}
\noindent
\begin{tikzpicture}[scale=0.5]



\begin{scope}[shift={(0,0)}]

\begin{scope}[shift={(17.5,-4.25)}]
\foreach \x in {0, 0.75, 1.5, 2.25} {
            \draw [fill=blue!55, line width=0.3mm] (\x,0) rectangle (\x+0.75, 0.75);

    }
    
    \foreach \x in {0, 0.75, 1.5} {

        }
        
         \foreach \x in {0.75} { 
         
          \draw [fill=red!55, line width=0.3mm] (\x+2.25,2.25) rectangle (\x+3, 3);
            \draw [fill=red!55, line width=0.3mm] (\x +2.25,1.5) rectangle (\x+3, 2.25);
            
            \draw [fill=red!55, line width=0.3mm] (\x+2.25,0.75) rectangle (\x+3, 1.5);
        }

    \draw [fill=red!55, line width=0.3mm] (3,0) rectangle (3.75, 0.75);

    

\end{scope}

\begin{scope}[shift={(17.5,-10.75)}]
\foreach \x in {0, 0.75, 1.5, 2.25} {
            \draw [fill=blue!55, line width=0.3mm] (\x,0) rectangle (\x+0.75, 0.75);

    }
    
    \foreach \x in {0, 0.75, 1.5} {

        }
        
         \foreach \x in {0.75} { 
         
         \draw [fill=red!55, line width=0.3mm] (\x +2.25,1.5) rectangle (\x+3, 2.25);
            \draw [fill=red!55, line width=0.3mm] (\x+2.25,2.25) rectangle (\x+3, 3);
            
            \draw [fill=red!55, line width=0.3mm] (\x+2.25,0.75) rectangle (\x+3, 1.5);
        }

    \draw [fill=red!55, line width=0.3mm] (3,0) rectangle (3.75, 0.75);

\end{scope}

\begin{scope}[shift={(17.5,-17.25)}]
\foreach \x in {0, 0.75, 1.5, 2.25} {
            \draw [fill=blue!55, line width=0.3mm] (\x,0) rectangle (\x+0.75, 0.75);

    }
    
    \foreach \x in {0, 0.75, 1.5} {

        }
        
         \foreach \x in {0.75} { 
            
            \draw [fill=red!55, line width=0.3mm] (\x+2.25,0.75) rectangle (\x+3, 1.5);
            
              \draw [fill=red!55, line width=0.3mm] (\x +2.25,1.5) rectangle (\x+3, 2.25);
            \draw [fill=red!55, line width=0.3mm] (\x+2.25,2.25) rectangle (\x+3, 3);
        }

    \draw [fill=red!55, line width=0.3mm] (3,0) rectangle (3.75, 0.75);


\end{scope}

\begin{scope}[shift={(17.5,-23.75)}]
\foreach \x in {0, 0.75, 1.5, 2.25} {
            \draw [fill=blue!55, line width=0.3mm] (\x,0) rectangle (\x+0.75, 0.75);

    }

         \foreach \x in {0.75} { 
            
            \draw [fill=red!55, line width=0.3mm] (\x+2.25,0.75) rectangle (\x+3, 1.5);
            
              \draw [fill=red!55, line width=0.3mm] (\x +2.25,1.5) rectangle (\x+3, 2.25);
            \draw [fill=red!55, line width=0.3mm] (\x+2.25,2.25) rectangle (\x+3, 3);
        }

    \draw [fill=red!55, line width=0.3mm] (3,0) rectangle (3.75, 0.75);

     \foreach \x in {0, 0.75, 1.5, 2.25} { 
     \foreach \y in {0, -0.75, -1.5, -2.25}{
      \draw [fill=orange!55, line width=0.3mm] (\x,\y+3) rectangle (\x+0.75,\y+3.75);
      }
}

\end{scope}

\end{scope}


\begin{scope}[shift={(7,0)}]

\begin{scope}[shift={(17.5,-4.25)}]
\foreach \x in {0, 0.75, 1.5, 2.25} {
            \draw [fill=blue!55, line width=0.3mm] (\x,0) rectangle (\x+0.75, 0.75);

    }
    
    \foreach \x in {0, 0.75, 1.5} {

        }
        
         \foreach \x in {0.75} { 
            \draw [fill=red!55, line width=0.3mm] (\x +2.25,1.5) rectangle (\x+3, 2.25);
            
             \draw [fill=red!55, line width=0.3mm] (\x+2.25,2.25) rectangle (\x+3, 3);
            \draw [fill=red!55, line width=0.3mm] (\x+2.25,0.75) rectangle (\x+3, 1.5);
        }
    
    
    \draw [fill=red!55, line width=0.3mm] (3,0) rectangle (3.75, 0.75);

    

\end{scope}

\begin{scope}[shift={(17.5,-10.75)}]
\foreach \x in {0, 0.75, 1.5, 2.25} {
            \draw [fill=blue!55, line width=0.3mm] (\x,0) rectangle (\x+0.75, 0.75);

    }
    
    \foreach \x in {0, 0.75, 1.5} {

        }
        
         \foreach \x in {0.75} { 
             \draw [fill=red!55, line width=0.3mm] (\x +2.25,1.5) rectangle (\x+3, 2.25);
            \draw [fill=red!55, line width=0.3mm] (\x+2.25,2.25) rectangle (\x+3, 3);
            \draw [fill=red!55, line width=0.3mm] (\x+2.25,0.75) rectangle (\x+3, 1.5);
        }

    \draw [fill=red!55, line width=0.3mm] (3,0) rectangle (3.75, 0.75);


\end{scope}

\begin{scope}[shift={(17.5,-17.25)}]
\foreach \x in {0, 0.75, 1.5, 2.25} {
            \draw [fill=blue!55, line width=0.3mm] (\x,0) rectangle (\x+0.75, 0.75);

    }
    
    \foreach \x in {0, 0.75, 1.5} {

        }
        
         \foreach \x in {0.75} { 
            
            \draw [fill=red!55, line width=0.3mm] (\x+2.25,0.75) rectangle (\x+3, 1.5);
            
              \draw [fill=red!55, line width=0.3mm] (\x +2.25,1.5) rectangle (\x+3, 2.25);
            \draw [fill=red!55, line width=0.3mm] (\x+2.25,2.25) rectangle (\x+3, 3);
        }

    \draw [fill=red!55, line width=0.3mm] (3,0) rectangle (3.75, 0.75);


\end{scope}

\begin{scope}[shift={(17.5,-23.75)}]
\foreach \x in {0, 0.75, 1.5, 2.25} {
            \draw [fill=blue!55, line width=0.3mm] (\x,0) rectangle (\x+0.75, 0.75);

    }

         \foreach \x in {0.75} { 
            
            \draw [fill=red!55, line width=0.3mm] (\x+2.25,0.75) rectangle (\x+3, 1.5);
            
              \draw [fill=red!55, line width=0.3mm] (\x +2.25,1.5) rectangle (\x+3, 2.25);
            \draw [fill=red!55, line width=0.3mm] (\x+2.25,2.25) rectangle (\x+3, 3);
        }

    \draw [fill=red!55, line width=0.3mm] (3,0) rectangle (3.75, 0.75);

     \foreach \x in {0, 0.75, 1.5, 2.25} { 
     \foreach \y in {0, -0.75, -1.5, -2.25}{
      \draw [fill=orange!55, line width=0.3mm] (\x,\y+3) rectangle (\x+0.75,\y+3.75);
      }
}

\end{scope}

\end{scope}


\begin{scope}[shift={(14,0)}]


\begin{scope}[shift={(17.5,-4.25)}]
\foreach \x in {0, 0.75, 1.5, 2.25} {
            \draw [fill=blue!55, line width=0.3mm] (\x,0) rectangle (\x+0.75, 0.75);

   }
    
    \foreach \x in {0, 0.75, 1.5} {

        }
        
         \foreach \x in {0.75} { 
         
          \draw [fill=red!55, line width=0.3mm] (\x +2.25,1.5) rectangle (\x+3, 2.25);
            
             \draw [fill=red!55, line width=0.3mm] (\x+2.25,2.25) rectangle (\x+3, 3);
            \draw [fill=red!55, line width=0.3mm] (\x+2.25,0.75) rectangle (\x+3, 1.5);
        }
    
    
    \draw [fill=red!55, line width=0.3mm] (3,0) rectangle (3.75, 0.75);

    

\end{scope}

\begin{scope}[shift={(17.5,-10.75)}]
\foreach \x in {0, 0.75, 1.5, 2.25} {
            \draw [fill=blue!55, line width=0.3mm] (\x,0) rectangle (\x+0.75, 0.75);

    }
    
    \foreach \x in {0, 0.75, 1.5} {

        }
        
         \foreach \x in {0.75} { 
         
  \draw [fill=red!55, line width=0.3mm] (\x +2.25,1.5) rectangle (\x+3, 2.25);
            \draw [fill=red!55, line width=0.3mm] (\x+2.25,2.25) rectangle (\x+3, 3);         
         
            \draw [fill=red!55, line width=0.3mm] (\x+2.25,0.75) rectangle (\x+3, 1.5);
        }

    \draw [fill=red!55, line width=0.3mm] (3,0) rectangle (3.75, 0.75);

\end{scope}

\begin{scope}[shift={(17.5,-17.25)}]
\foreach \x in {0, 0.75, 1.5, 2.25} {
            \draw [fill=blue!55, line width=0.3mm] (\x,0) rectangle (\x+0.75, 0.75);

    }
    
    \foreach \x in {0, 0.75, 1.5} {

        }
        
         \foreach \x in {0.75} { 
            
            \draw [fill=red!55, line width=0.3mm] (\x+2.25,0.75) rectangle (\x+3, 1.5);
            
              \draw [fill=red!55, line width=0.3mm] (\x +2.25,1.5) rectangle (\x+3, 2.25);
            \draw [fill=red!55, line width=0.3mm] (\x+2.25,2.25) rectangle (\x+3, 3);
        }

    \draw [fill=red!55, line width=0.3mm] (3,0) rectangle (3.75, 0.75);
    
    

\end{scope}

\begin{scope}[shift={(17.5,-23.75)}]
\foreach \x in {0, 0.75, 1.5, 2.25} {
            \draw [fill=blue!55, line width=0.3mm] (\x,0) rectangle (\x+0.75, 0.75);

    }

         \foreach \x in {0.75} { 
            
            \draw [fill=red!55, line width=0.3mm] (\x+2.25,0.75) rectangle (\x+3, 1.5);
            
              \draw [fill=red!55, line width=0.3mm] (\x +2.25,1.5) rectangle (\x+3, 2.25);
            \draw [fill=red!55, line width=0.3mm] (\x+2.25,2.25) rectangle (\x+3, 3);
        }

    \draw [fill=red!55, line width=0.3mm] (3,0) rectangle (3.75, 0.75);

     \foreach \x in {0, 0.75, 1.5, 2.25} { 
     \foreach \y in {0, -0.75, -1.5, -2.25}{
      \draw [fill=orange!55, line width=0.3mm] (\x,\y+3) rectangle (\x+0.75,\y+3.75);
      }
}

\end{scope}

\end{scope}


\begin{scope}[shift={(21,0)}]


\begin{scope}[shift={(17.5,-4.25)}]
\foreach \x in {0, 0.75, 1.5, 2.25} {
            \draw [fill=blue!55, line width=0.3mm] (\x,0) rectangle (\x+0.75, 0.75);

   }
    
    \foreach \x in {0, 0.75, 1.5} { 

        }

         \draw [fill=red!55, line width=0.3mm] (3,1.5) rectangle (3.75, 2.25);
            
             \draw [fill=red!55, line width=0.3mm] (3,2.25) rectangle (3.75, 3);
            
            \draw [fill=red!55, line width=0.3mm] (3,0.75) rectangle (3.75, 1.5);
 
    \draw [fill=red!55, line width=0.3mm] (3,0) rectangle (3.75, 0.75);

    

\end{scope}

\begin{scope}[shift={(17.5,-10.75)}]
\foreach \x in {0, 0.75, 1.5, 2.25} {
            \draw [fill=blue!55, line width=0.3mm] (\x,0) rectangle (\x+0.75, 0.75);

    }

      \draw [fill=red!55, line width=0.3mm] (3,1.5) rectangle (3.75, 2.25);
            \draw [fill=red!55, line width=0.3mm] (3,2.25) rectangle (3.75, 3);

            \draw [fill=red!55, line width=0.3mm] (3,0.75) rectangle (3.75, 1.5);

    \draw [fill=red!55, line width=0.3mm] (3,0) rectangle (3.75, 0.75);

\end{scope}

\begin{scope}[shift={(17.5,-17.25)}]
\foreach \x in {0, 0.75, 1.5, 2.25} {
            \draw [fill=blue!55, line width=0.3mm] (\x,0) rectangle (\x+0.75, 0.75);

    }

       \draw [fill=red!55, line width=0.3mm] (3,0.75) rectangle (3.75, 1.5);
            
              \draw [fill=red!55, line width=0.3mm] (3,1.5) rectangle (3.75, 2.25);
            \draw [fill=red!55, line width=0.3mm] (3,2.25) rectangle (3.75, 3);

    \draw [fill=red!55, line width=0.3mm] (3,0) rectangle (3.75, 0.75);

\end{scope}

\begin{scope}[shift={(17.5,-23.75)}]
\foreach \x in {0, 0.75, 1.5, 2.25} {
            \draw [fill=blue!55, line width=0.3mm] (\x,0) rectangle (\x+0.75, 0.75);

    }
    
         \foreach \x in {0.75} { 
            
            \draw [fill=red!55, line width=0.3mm] (\x+2.25,0.75) rectangle (\x+3, 1.5);
            
              \draw [fill=red!55, line width=0.3mm] (\x +2.25,1.5) rectangle (\x+3, 2.25);
            \draw [fill=red!55, line width=0.3mm] (\x+2.25,2.25) rectangle (\x+3, 3);
        }

    \draw [fill=red!55, line width=0.3mm] (3,0) rectangle (3.75, 0.75);
    
     \foreach \x in {0, 0.75, 1.5, 2.25} { 
     \foreach \y in {0, -0.75, -1.5, -2.25}{
      \draw [fill=orange!55, line width=0.3mm] (\x,\y+3) rectangle (\x+0.75,\y+3.75);
      }
}


\begin{scope}[shift={(0,6.5)}]
 \foreach \x in {0, 0.75, 1.5, 2.25} { 
     \foreach \y in {0, -0.75, -1.5, -2.25}{
      \draw [fill=green!55, line width=0.3mm] (\x,\y+3) rectangle (\x+0.75,\y+3.75);
      }
}
\end{scope}

\begin{scope}[shift={(0,13)}]
 \foreach \x in {0, 0.75, 1.5, 2.25} { 
     \foreach \y in {0, -0.75, -1.5, -2.25}{
      \draw [fill=green!55, line width=0.3mm] (\x,\y+3) rectangle (\x+0.75,\y+3.75);
      }
}
\end{scope}

\begin{scope}[shift={(0,19.5)}]
 \foreach \x in {0, 0.75, 1.5, 2.25} { 
     \foreach \y in {0, -0.75, -1.5, -2.25}{
      \draw [fill=green!55, line width=0.3mm] (\x,\y+3) rectangle (\x+0.75,\y+3.75);
      }
}
\end{scope}

\begin{scope}[shift={(0,26)}]
 \foreach \x in {0, 0.75, 1.5, 2.25} { 
     \foreach \y in {0, -0.75, -1.5, -2.25}{
      \draw [fill=green!55, line width=0.3mm] (\x,\y+3) rectangle (\x+0.75,\y+3.75);
      }
}
\end{scope}

\end{scope}

\end{scope}

\end{tikzpicture}
\vskip .8cm

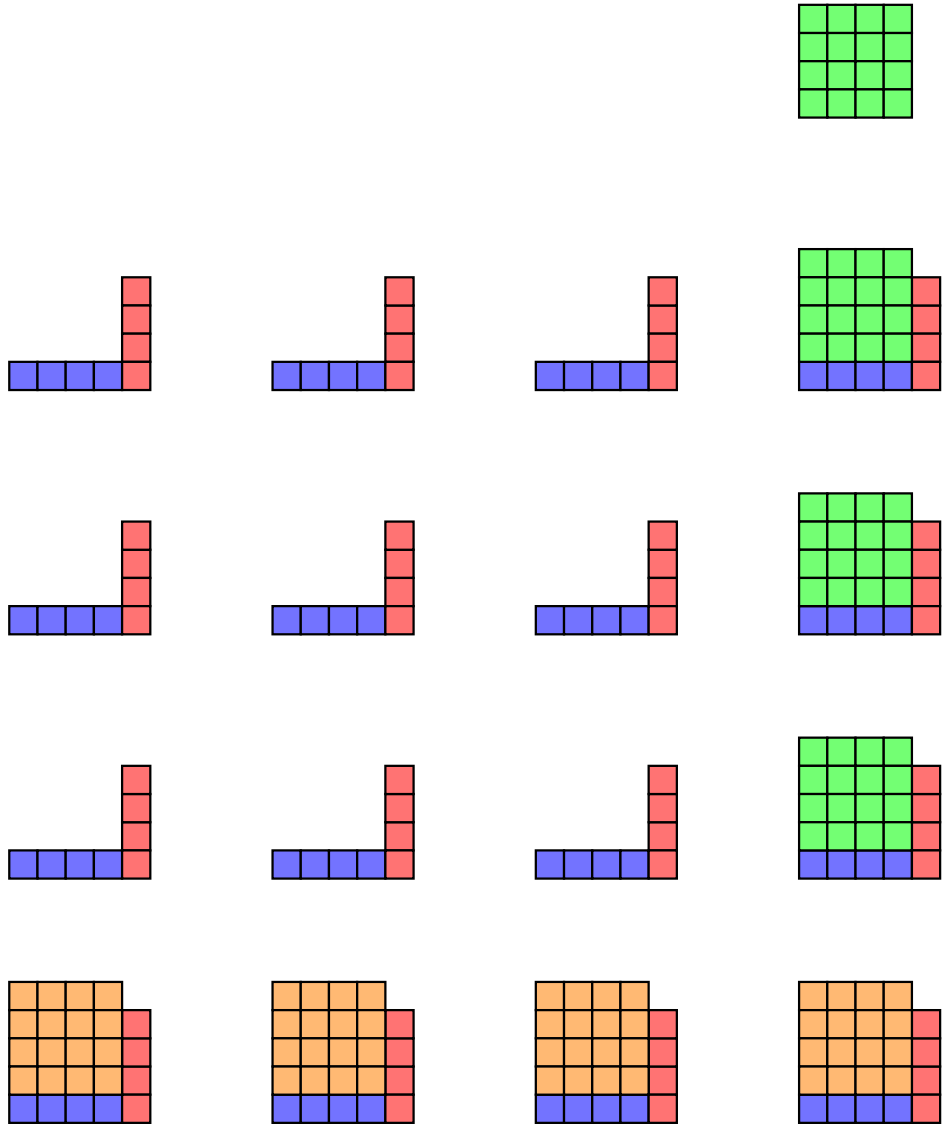
\captionof{figure}{$n$-th 4D section of four assembled 5D pyramids}
\label{n 4D Section}

\begin{remark}[Algebra vs. Visual Arguments II]
At this point it is natural to ask for an algebraic justification that the sections are indeed obtained in this way. In lower dimensions we have already checked that this procedure works geometrically, but as we move to higher dimensions our intuition becomes weaker, and we could ask for an algebraic argument on which to rely.

\restoreparindent For instance, in the case of two 3D pyramids we obtain

\begin {equation*}
\begin{aligned}
2(1^2 +2^2+ 3^2+ \cdots + n^2) &=& 2(1+2+3+\cdots +n )\\
&&+2(2+3+\cdots +n )\\
&&+2(3+\cdots +n )\\
&&\hfill \vdots \ \\
&&+2(\hfill n).
\end{aligned}
\end{equation*}

\restoreparindent Here the left-hand side counts the unit cubes in the two fitted 3D pyramids (the red and blue pyramids in Figure~\ref{2D of 3D}), and each line on the right-hand side corresponds to a 2D section. In each section, both triangular numbers lose one step, until in the last section we are left with two steps of length $n$.

\restoreparindent If instead of two 3D pyramids we assemble three of them, we obtain
\begin {equation*}
\begin{aligned}
3(1^2 +2^2+ 3^2+ \cdots + n^2) &=& 1^2+2^2 + 3^2 + \cdots + n^2\ \\
&&+2(1^2+2^2+3^2+\cdots +n^2 )\\
&=&[1^2 +2(1+2+3+\cdots +n )]\\
&&[2^2+2(2+3+\cdots +n )]\\
&&[3^2+2(3+\cdots +n )]\\
&&\hfill \vdots \ \\
&&[n^2+2(\hfill n)],
\end{aligned}
\end{equation*}
where each line on the right-hand side corresponds to one of the 2D sections shown on the right-hand side of Figure~\ref{2D of 3D}.

\restoreparindent The key Lemma that we will use in higher dimensions states that if we truncate a sum of powers with exponent $p+1$, the corresponding triangular sums with exponent $p$ are also truncated at the same point. More precisely, it says the following.

\begin{lemma}\label{Truncated Lemma}
For any $p\geq 0$, $n\geq 1$ and $1\leq m\leq n$ we have:
\begin {equation}\label{Truncated Algebraic formula}
\begin{aligned}
m^{p+1} +(m+1)^{p+1}+ \cdots + n^{p+1} &=& m^p+(m+1)^p+\cdots +n^p \\
&& \tikz \draw[dashed] (-3,0) -- (1,0); \\
&& m^p+(m+1)^p+\cdots +n^p \\
&&\ \ \ +(m+1)^p+\cdots +n^p \\
&&\hfill \vdots \ \\
&&\hfill +n^p
\end{aligned}
\end{equation}

\end{lemma}
\begin{proof}
    The proof consists in writing $m^{p+1}+(m+1)^{p+1}+\cdots +n^{p+1}$ as the difference 
    $$ (1^{p+1}+2^{p+1}+\cdots +n^{p+1}) - (1^{p+1}+2^{p+1}+\cdots +(m-1)^{p+1})$$
and then applying Lemma~\ref{Algebraic Lemma} to express both sums in terms of triangular sums. The equal terms cancel out, and the statement follows.
\end{proof}

In Chapter~\ref{Nicomachus 4D in 3D}, in order to prove the 4D version of Nicomachus’ theorem through 3D sections, we took three 4-dimensional pyramids and obtained
\begin {equation*}
\begin{aligned}
3(1^3 +2^3+ 3^3+ \cdots + n^3) &=& 3(1^2+2^2+3^2+\cdots +n^2 )\\
&&+3(2^2+3^2+\cdots +n^2 )\\
&&+3(3^2+\cdots +n^2 )\\
&&\hfill \vdots \ \\
&&+3(\hfill n^2),
\end{aligned}
\end{equation*}
where each line on the right-hand side corresponds to one of the 
$n$ three-dimensional sections of the three 4D pyramids, as shown in Figure~\ref{3 pyramids 4D}. The first section coincides with the three assembled 3D pyramids, and in each subsequent section these 3D pyramids lose one floor.

We can now apply Lemma~\ref{Truncated Lemma} to the truncated sums of squares and rewrite this identity as

$$3(1^3 +2^3+ 3^3+ \cdots + n^3)=$$
\begin {equation*}
\begin{aligned}
&=& \begin{pmatrix}
[1^2 +2(1+2+3+\cdots + n)] \\
+[2^2 +2(2+3+\cdots + n)] \\
+[3^2 +2(3+\cdots + n)] \\
\vdots  \\
+[n^2 + 2(n)]
\end{pmatrix} + \begin{pmatrix}
[2(2+3+\cdots + n)]\\
+[2^2 +2(2+3+\cdots + n)] \\
+[3^2 +2(3+\cdots + n)] \\
\vdots  \\
+[n^2 + 2(n)]
\end{pmatrix}\\
&&+\begin{pmatrix}
[2(3+\cdots + n)] \\
+[2(3+\cdots + n)] \\
+[3^2 +2(3+\cdots + n)]\\
\vdots  \\
+[n^2 + 2(n)]
\end{pmatrix}+\cdots +\begin{pmatrix}
[2(n)]\\
+[2(n)]\\
+[2(n)]\\
\vdots\\
+[n^2 + 2(n)]
\end{pmatrix}\\
\end{aligned}
\end{equation*}

where each pair of parentheses represents one column in Figure~\ref{2D section Nicomachus 4D}. We observe that, in the second section, the terms $1^2$ and $1$  are missing; in the next section, the terms $2^2$ and $2$ also disappear; and in the last section only the terms $n^2$ and $n$ remain.

Figure~\ref{2D section Nicomachus 4D Bis}, corresponding to the 2D sections of four assembled 4D pyramids, is represented by the following algebraic manipulation

\begin {equation*}
\begin{aligned}
4(1^3 +2^3+ 3^3+ \cdots + n^3) &=& (1^3 + 2^3 + 3^3 + \cdots + n^3)\\
&&+3(1^3 + 2^3 + 3^3 + \cdots + n^3)\\
&=& 1^2+2^2+3^2+\cdots +n^2 \\
&&\ \ +2^2+3^2+\cdots +n^2 \\
&&\ \ \ \ +3^2+\cdots +n^2 \\
&&\hfill \vdots \ \\
&&+\hfill n^2
\end{aligned}
\end{equation*}
\begin {equation*}
\begin{aligned}
&+& \begin{pmatrix}
[1^2 +2(1+2+3+\cdots + n)] \\
+[2^2 +2(2+3+\cdots + n)] \\
+[3^2 +2(3+\cdots + n)] \\
\vdots  \\
+[n^2 + 2(n)]
\end{pmatrix} + \begin{pmatrix}
[2(2+3+\cdots + n)]\\
+[2^2 +2(2+3+\cdots + n)] \\
+[3^2 +2(3+\cdots + n)] \\
\vdots  \\
+[n^2 + 2(n)]
\end{pmatrix}\\
&&+\begin{pmatrix}
[2(3+\cdots + n)] \\
+[2(3+\cdots + n)] \\
+[3^2 +2(3+\cdots + n)]\\
\vdots  \\
+[n^2 + 2(n)]
\end{pmatrix}+\cdots +\begin{pmatrix}
[2(n)]\\
+[2(n)]\\
+[2(n)]\\
\vdots\\
+[n^2 + 2(n)]
\end{pmatrix}\\
&=& \begin{pmatrix}
\textcolor{darkgreen}{1^2} \\
+[\textcolor{orange}{1^2} +2(1+2+3+\cdots + n)] \\
+[\textcolor{orange}{2^2} +2(2+3+\cdots + n)] \\
+[\textcolor{orange}{3^2} +2(3+\cdots + n)] \\
\vdots  \\
+[\textcolor{orange}{n^2} + 2(n)]
\end{pmatrix} + \begin{pmatrix}
\textcolor{darkgreen}{2^2} \\
+[\textcolor{darkgreen}{2^2} + 2(2+3+\cdots + n)]\\
+[\textcolor{orange}{2^2} +2(2+3+\cdots + n)] \\
+[\textcolor{orange}{3^2} +2(3+\cdots + n)] \\
\vdots  \\
+[\textcolor{orange}{n^2} + 2(n)]
\end{pmatrix}\\
&&+\begin{pmatrix}
\textcolor{darkgreen}{3^2} \\
+[\textcolor{darkgreen}{3^2}+2(3+\cdots + n)] \\
+[\textcolor{darkgreen}{3^2}+2(3+\cdots + n)] \\
+[\textcolor{orange}{3^2} +2(3+\cdots + n)]\\
\vdots  \\
+[\textcolor{orange}{n^2} + 2(n)]
\end{pmatrix}+\cdots +\begin{pmatrix}
\textcolor{darkgreen}{n^2}\\
[\textcolor{darkgreen}{n^2}+2(n)]\\
+[\textcolor{darkgreen}{n^2}+2(n)]\\
+[\textcolor{darkgreen}{n^2}+2(n)]\\
\vdots\\
+[\textcolor{orange}{n^2} + 2(n)]
\end{pmatrix}\\
\end{aligned}
\end{equation*}

In the last equation we have color-coded the unit squares so that they match the corresponding ones in Figure~\ref{2D section Nicomachus 4D Bis}. This is meant to help the reader interpret this algebraic manipulation in parallel with the manipulation of figurate numbers. However, giving a direct algebraic description of the step from Figure~\ref{2D section Nicomachus 4D Bis} to Figure~\ref{2D section Nicomachus 4D Bis Bis} is rather hard to formulate.

We can agree that the algebra needed to encode the visual arguments for assembling pyramids and their sections is rather cumbersome and tedious, as it strongly depends on how the different terms are grouped in parentheses. This kind of algebraic manipulation is non-standard and not especially pleasant. By contrast, the pictures that give rise to these calculations are simple and transparent. This is a situation in which visual arguments seem more appropriate, and in some sense more powerful, than the algebra itself. In any case, the purpose of this remark is to justify the method of using a diagram in one dimension less to construct the sections in the next dimension by successively removing squares and steps. The 5-dimensional case is analogous and can be checked in the same way, but we prefer not to write it out in full for aesthetic reasons.
\end{remark}

\vskip .5cm
\subsection{Step 1: Adding the Fifth 5D Pyramid}
To represent the fifth 5D pyramid that we want to assemble, we will use the identity
 $$1^4+2^4+3^4+\cdots +n^4 = 1^2 \cdot 1^2+2^2 \cdot 2^2 +3^2 \cdot 3^2 +\cdots + n^2 \cdot n^2$$
as shown in Figure~\ref{fifth 5D pyramid}.\par
\vspace{5mm}
\noindent
\begin{tikzpicture}[scale=0.47]

\draw [fill=pink!40, line width=0.3mm] (0,3) rectangle (0.75,3.75);

\foreach \p in {0, 2} {
\foreach \q in {0, 2} {
\begin{scope}[shift={(\p, \q-4)}]
 \foreach \x in {0, 0.75} { 
     \foreach \y in {0, -0.75}{
      \draw [fill=pink!40, line width=0.3mm] (\x,\y+3) rectangle (\x+0.75,\y+3.75);
      }
}
\end{scope}

}
}

\foreach \p in {0, 3.5, 7} {
\foreach \q in {0, 3.5, 7} {
\begin{scope}[shift={(\p, \q-14)}]
 \foreach \x in {0, 0.75, 1.5} { 
     \foreach \y in {0, -0.75, -1.5}{
      \draw [fill=pink!40, line width=0.3mm] (\x,\y+3) rectangle (\x+0.75,\y+3.75);
      }
}
\end{scope}

}
}

\foreach \p in {0, 5, 10, 15} {
\foreach \q in {0, 5, 10, 15} {
\begin{scope}[shift={(\p, \q-33)}]
 \foreach \x in {0, 0.75, 1.5, 2.25} { 
     \foreach \y in {0, -0.75, -1.5, -2.25}{
      \draw [fill=pink!40, line width=0.3mm] (\x,\y+3) rectangle (\x+0.75,\y+3.75);
      }
}
\end{scope}

}
}

\node at (18,0) {\huge $1^4 +2^4+3^4+\cdots +n^4$};


\begin{scope}[shift={(13, -18)}] 
   
\draw[<->] (2,0.4) -- (5,0.4) node[midway, below, yshift=0.5mm] {$n$};
\draw[<->] (5.4,1-0.25) -- (5.4,4-0.25) node[midway, right, xshift=-0.5mm] {$n$};

\end{scope}

\begin{scope}[shift={(-15, -8)}] 

\draw[<->] (15,-25.3) -- (33,-25.3) node[midway, below, yshift=-2mm] {\Huge $n$};

\draw[<->] (34.2,-24.3) -- (34.2,-6.3) node[midway, right, xshift=1.7mm] {\Huge $n$};


\draw [dashed, ultra thick] (14.5,-5.75) rectangle (28.5,-19.75);
\draw [dashed, ultra thick] (14.5, 5.25) rectangle (21.25,-1.5);
\draw [dashed, ultra thick] (14.75,10) rectangle (16.75, 8) ;

\end{scope}
\end{tikzpicture}

\vskip .8cm
\captionof{figure}{2D sections of the fifth 5D pyramid}
\label{fifth 5D pyramid}

\restoreparindent Next, we apply the same visual filling argument as in Figure~\ref{2D section Nicomachus 4D Bis Bis} to the diagrams in Figures~\ref{First 4D Section}–\ref{n 4D Section}. After adding the green unit squares in each of the 
$n$ sections, some square gaps appear, which are then filled with the white unit squares indicated with a dashed outline in Figure~\ref{fifth 5D pyramid}. In this way we obtain Figures~\ref{first with 5D}–\ref{n with 5D}, where all the puzzles have become perfect rectangular blocks.\par
\vspace{10mm}
\noindent
\begin{tikzpicture}[scale=0.45]


\begin{scope}[shift={(0,0)}]

\begin{scope}[shift={(17.5,-4.25)}]
\foreach \x in {0, 0.75, 1.5, 2.25} {
            \draw [fill=blue!55, line width=0.3mm] (\x,0) rectangle (\x+0.75, 0.75);
             \draw [fill=green!55, line width=0.3mm] (\x+0.75,3) rectangle (\x+1.5, 3.75);
            
            \draw [fill=red!55, line width=0.3mm] (\x+0.75,2.25) rectangle (\x+1.5, 3);
       
    }
    
    \foreach \x in {0, 0.75, 1.5} { 
            \draw [fill=blue!55, line width=0.3mm] (\x,0.75) rectangle (\x+0.75, 1.5);
            
            \draw [fill=red!55, line width=0.3mm] (\x +1.5,1.5) rectangle (\x+2.25, 2.25);
        }
        
         \foreach \x in {0, 0.75} { 
            \draw [fill=blue!55, line width=0.3mm] (\x,1.5) rectangle (\x+0.75, 2.25);
            \draw [fill=red!55, line width=0.3mm] (\x+2.25,0.75) rectangle (\x+3, 1.5);
        }
    
    \draw [fill=blue!55, line width=0.3mm] (0,2.25) rectangle (0.75, 3);
    
    \draw [fill=red!55, line width=0.3mm] (3,0) rectangle (3.75, 0.75);

    
      \draw [fill=orange!55, line width=0.3mm] (0,3) rectangle (0.75, 3.75);

\end{scope}

\begin{scope}[shift={(17.5,-10.75)}]
\foreach \x in {0, 0.75, 1.5, 2.25} {
            \draw [fill=blue!55, line width=0.3mm] (\x,0) rectangle (\x+0.75, 0.75);

    }
    
    \foreach \x in {0, 0.75, 1.5} { 
            \draw [fill=blue!55, line width=0.3mm] (\x,0.75) rectangle (\x+0.75, 1.5);
            \draw [fill=green!55, line width=0.3mm] (\x+1.5,3) rectangle (\x+2.25, 3.75);
            
            \draw [fill=red!55, line width=0.3mm] (\x +1.5,1.5) rectangle (\x+2.25, 2.25);
            \draw [fill=red!55, line width=0.3mm] (\x+1.5,2.25) rectangle (\x+2.25, 3);
        }
        
         \foreach \x in {0, 0.75} { 
            \draw [fill=blue!55, line width=0.3mm] (\x,1.5) rectangle (\x+0.75, 2.25);
            \draw [fill=red!55, line width=0.3mm] (\x+2.25,0.75) rectangle (\x+3, 1.5);
        }

    \draw [fill=red!55, line width=0.3mm] (3,0) rectangle (3.75, 0.75);

     \foreach \x in {0, 0.75} { 
     \foreach \y in {0, -0.75}{
      \draw [fill=orange!55, line width=0.3mm] (\x,\y+3) rectangle (\x+0.75,\y+3.75);
      }
}

\end{scope}

\begin{scope}[shift={(17.5,-17.25)}]
\foreach \x in {0, 0.75, 1.5, 2.25} {
            \draw [fill=blue!55, line width=0.3mm] (\x,0) rectangle (\x+0.75, 0.75);

    }
    
    \foreach \x in {0, 0.75, 1.5} { 
            \draw [fill=blue!55, line width=0.3mm] (\x,0.75) rectangle (\x+0.75, 1.5);

        }
        
         \foreach \x in {0, 0.75} { 
             \draw [fill=green!55, line width=0.3mm] (\x+2.25,3) rectangle (\x+3, 3.75);
            \draw [fill=red!55, line width=0.3mm] (\x+2.25,0.75) rectangle (\x+3, 1.5);
            
              \draw [fill=red!55, line width=0.3mm] (\x +2.25,1.5) rectangle (\x+3, 2.25);
            \draw [fill=red!55, line width=0.3mm] (\x+2.25,2.25) rectangle (\x+3, 3);
        }

    \draw [fill=red!55, line width=0.3mm] (3,0) rectangle (3.75, 0.75);

     \foreach \x in {0, 0.75, 1.5} { 
     \foreach \y in {0, -0.75, -1.5}{
      \draw [fill=orange!55, line width=0.3mm] (\x,\y+3) rectangle (\x+0.75,\y+3.75);
      }
}

\end{scope}

\begin{scope}[shift={(17.5,-23.75)}]
\foreach \x in {0, 0.75, 1.5, 2.25} {
            \draw [fill=blue!55, line width=0.3mm] (\x,0) rectangle (\x+0.75, 0.75);

    }

         \foreach \x in {0.75} { 
            
              \draw [fill=green!55, line width=0.3mm] (\x+2.25,3) rectangle (\x+3, 3.75);
            \draw [fill=red!55, line width=0.3mm] (\x+2.25,0.75) rectangle (\x+3, 1.5);
            
              \draw [fill=red!55, line width=0.3mm] (\x +2.25,1.5) rectangle (\x+3, 2.25);
            \draw [fill=red!55, line width=0.3mm] (\x+2.25,2.25) rectangle (\x+3, 3);
        }

    \draw [fill=red!55, line width=0.3mm] (3,0) rectangle (3.75, 0.75);

     \foreach \x in {0, 0.75, 1.5, 2.25} { 
     \foreach \y in {0, -0.75, -1.5, -2.25}{
      \draw [fill=orange!55, line width=0.3mm] (\x,\y+3) rectangle (\x+0.75,\y+3.75);
      }
}

\end{scope}

\end{scope}


\begin{scope}[shift={(7,0)}]

\begin{scope}[shift={(17.5,-4.25)}]
\foreach \x in {0, 0.75, 1.5, 2.25} {
            \draw [fill=blue!55, line width=0.3mm] (\x,0) rectangle (\x+0.75, 0.75);

    }
    
    \foreach \x in {0, 0.75, 1.5} { 
  \draw [fill=green!55, line width=0.3mm] (\x+1.5,3) rectangle (\x+2.25, 3.75);    
    
            \draw [fill=blue!55, line width=0.3mm] (\x,0.75) rectangle (\x+0.75, 1.5);
            
            \draw [fill=red!55, line width=0.3mm] (\x +1.5,1.5) rectangle (\x+2.25, 2.25);
            
             \draw [fill=red!55, line width=0.3mm] (\x+1.5,2.25) rectangle (\x+2.25, 3);
        }
        
         \foreach \x in {0, 0.75} { 
            \draw [fill=blue!55, line width=0.3mm] (\x,1.5) rectangle (\x+0.75, 2.25);
            \draw [fill=red!55, line width=0.3mm] (\x+2.25,0.75) rectangle (\x+3, 1.5);
        }
    
    
    \draw [fill=red!55, line width=0.3mm] (3,0) rectangle (3.75, 0.75);

    

\end{scope}

\begin{scope}[shift={(17.5,-10.75)}]
\foreach \x in {0, 0.75, 1.5, 2.25} {
            \draw [fill=blue!55, line width=0.3mm] (\x,0) rectangle (\x+0.75, 0.75);

    }
    
    \foreach \x in {0, 0.75, 1.5} { 
    
    \draw [fill=green!55, line width=0.3mm] (\x+1.5,3) rectangle (\x+2.25, 3.75);
            \draw [fill=blue!55, line width=0.3mm] (\x,0.75) rectangle (\x+0.75, 1.5);
            
            \draw [fill=red!55, line width=0.3mm] (\x +1.5,1.5) rectangle (\x+2.25, 2.25);
            \draw [fill=red!55, line width=0.3mm] (\x+1.5,2.25) rectangle (\x+2.25, 3);
        }
        
         \foreach \x in {0, 0.75} { 
            \draw [fill=blue!55, line width=0.3mm] (\x,1.5) rectangle (\x+0.75, 2.25);
            \draw [fill=red!55, line width=0.3mm] (\x+2.25,0.75) rectangle (\x+3, 1.5);
        }

    \draw [fill=red!55, line width=0.3mm] (3,0) rectangle (3.75, 0.75);

     \foreach \x in {0, 0.75} { 
     \foreach \y in {0, -0.75}{
      \draw [fill=orange!55, line width=0.3mm] (\x,\y+3) rectangle (\x+0.75,\y+3.75);
      }
}

\begin{scope}[shift={(0,6.5)}]
 \foreach \x in {0, 0.75} { 
     \foreach \y in {0, -0.75}{
      \draw [fill=green!55, line width=0.3mm] (\x,\y+3) rectangle (\x+0.75,\y+3.75);
      }
}
\end{scope}

\end{scope}

\begin{scope}[shift={(17.5,-17.25)}]
\foreach \x in {0, 0.75, 1.5, 2.25} {
            \draw [fill=blue!55, line width=0.3mm] (\x,0) rectangle (\x+0.75, 0.75);

    }
    
    \foreach \x in {0, 0.75, 1.5} { 
            \draw [fill=blue!55, line width=0.3mm] (\x,0.75) rectangle (\x+0.75, 1.5);

        }
        
         \foreach \x in {0, 0.75} { 
            \draw [fill=green!55, line width=0.3mm] (\x+2.25,3) rectangle (\x+3, 3.75);
            
            \draw [fill=red!55, line width=0.3mm] (\x+2.25,0.75) rectangle (\x+3, 1.5);
            
              \draw [fill=red!55, line width=0.3mm] (\x +2.25,1.5) rectangle (\x+3, 2.25);
            \draw [fill=red!55, line width=0.3mm] (\x+2.25,2.25) rectangle (\x+3, 3);
        }

    \draw [fill=red!55, line width=0.3mm] (3,0) rectangle (3.75, 0.75);

     \foreach \x in {0, 0.75, 1.5} { 
     \foreach \y in {0, -0.75, -1.5}{
      \draw [fill=orange!55, line width=0.3mm] (\x,\y+3) rectangle (\x+0.75,\y+3.75);
      }
}

\end{scope}

\begin{scope}[shift={(17.5,-23.75)}]
\foreach \x in {0, 0.75, 1.5, 2.25} {
            \draw [fill=blue!55, line width=0.3mm] (\x,0) rectangle (\x+0.75, 0.75);

    }

         \foreach \x in {0.75} { 
            
\draw [fill=green!55, line width=0.3mm] (\x+2.25,3) rectangle (\x+3, 3.75);            
            
            \draw [fill=red!55, line width=0.3mm] (\x+2.25,0.75) rectangle (\x+3, 1.5);
            
              \draw [fill=red!55, line width=0.3mm] (\x +2.25,1.5) rectangle (\x+3, 2.25);
            \draw [fill=red!55, line width=0.3mm] (\x+2.25,2.25) rectangle (\x+3, 3);
        }

    \draw [fill=red!55, line width=0.3mm] (3,0) rectangle (3.75, 0.75);

     \foreach \x in {0, 0.75, 1.5, 2.25} { 
     \foreach \y in {0, -0.75, -1.5, -2.25}{
      \draw [fill=orange!55, line width=0.3mm] (\x,\y+3) rectangle (\x+0.75,\y+3.75);
      }
}

\end{scope}

\end{scope}


\begin{scope}[shift={(14,0)}]


\begin{scope}[shift={(17.5,-4.25)}]
\foreach \x in {0, 0.75, 1.5, 2.25} {
            \draw [fill=blue!55, line width=0.3mm] (\x,0) rectangle (\x+0.75, 0.75);

   }
    
    \foreach \x in {0, 0.75, 1.5} { 
            \draw [fill=blue!55, line width=0.3mm] (\x,0.75) rectangle (\x+0.75, 1.5);

        }
        
         \foreach \x in {0, 0.75} { 
         
 \draw [fill=green!55, line width=0.3mm] (\x +2.25,3) rectangle (\x+3, 3.75);         
         
          \draw [fill=red!55, line width=0.3mm] (\x +2.25,1.5) rectangle (\x+3, 2.25);
            
             \draw [fill=red!55, line width=0.3mm] (\x+2.25,2.25) rectangle (\x+3, 3);
            \draw [fill=red!55, line width=0.3mm] (\x+2.25,0.75) rectangle (\x+3, 1.5);
        }
    
    
    \draw [fill=red!55, line width=0.3mm] (3,0) rectangle (3.75, 0.75);

    

\end{scope}

\begin{scope}[shift={(17.5,-10.75)}]
\foreach \x in {0, 0.75, 1.5, 2.25} {
            \draw [fill=blue!55, line width=0.3mm] (\x,0) rectangle (\x+0.75, 0.75);

    }
    
    \foreach \x in {0, 0.75, 1.5} { 
            \draw [fill=blue!55, line width=0.3mm] (\x,0.75) rectangle (\x+0.75, 1.5);

        }
        
         \foreach \x in {0, 0.75} { 
         \draw [fill=green!55, line width=0.3mm] (\x +2.25,3) rectangle (\x+3, 3.75);
         
  \draw [fill=red!55, line width=0.3mm] (\x +2.25,1.5) rectangle (\x+3, 2.25);
            \draw [fill=red!55, line width=0.3mm] (\x+2.25,2.25) rectangle (\x+3, 3);         
         
            \draw [fill=red!55, line width=0.3mm] (\x+2.25,0.75) rectangle (\x+3, 1.5);
        }

    \draw [fill=red!55, line width=0.3mm] (3,0) rectangle (3.75, 0.75);

\end{scope}

\begin{scope}[shift={(17.5,-17.25)}]
\foreach \x in {0, 0.75, 1.5, 2.25} {
            \draw [fill=blue!55, line width=0.3mm] (\x,0) rectangle (\x+0.75, 0.75);

    }
    
    \foreach \x in {0, 0.75, 1.5} { 
            \draw [fill=blue!55, line width=0.3mm] (\x,0.75) rectangle (\x+0.75, 1.5);

        }
        
         \foreach \x in {0, 0.75} { 
            
\draw [fill=green!55, line width=0.3mm] (\x+2.25,3) rectangle (\x+3, 3.75);            
            
            \draw [fill=red!55, line width=0.3mm] (\x+2.25,0.75) rectangle (\x+3, 1.5);
            
              \draw [fill=red!55, line width=0.3mm] (\x +2.25,1.5) rectangle (\x+3, 2.25);
            \draw [fill=red!55, line width=0.3mm] (\x+2.25,2.25) rectangle (\x+3, 3);
        }

    \draw [fill=red!55, line width=0.3mm] (3,0) rectangle (3.75, 0.75);

     \foreach \x in {0, 0.75, 1.5} { 
     \foreach \y in {0, -0.75, -1.5}{
      \draw [fill=orange!55, line width=0.3mm] (\x,\y+3) rectangle (\x+0.75,\y+3.75);
      }
}

\begin{scope}[shift={(0,6.5)}]
\foreach \x in {0, 0.75, 1.5} { 
     \foreach \y in {0, -0.75, -1.5}{
      \draw [fill=green!55, line width=0.3mm] (\x,\y+3) rectangle (\x+0.75,\y+3.75);
      }
}
\end{scope}

\begin{scope}[shift={(0,13)}]
\foreach \x in {0, 0.75, 1.5} { 
     \foreach \y in {0, -0.75, -1.5}{
      \draw [fill=green!55, line width=0.3mm] (\x,\y+3) rectangle (\x+0.75,\y+3.75);
      }
}
\end{scope}

\end{scope}

\begin{scope}[shift={(17.5,-23.75)}]
\foreach \x in {0, 0.75, 1.5, 2.25} {
            \draw [fill=blue!55, line width=0.3mm] (\x,0) rectangle (\x+0.75, 0.75);

    }

         \foreach \x in {0.75} { 
             \draw [fill=green!55, line width=0.3mm] (\x+2.25,3) rectangle (\x+3, 3.75); 
             
            \draw [fill=red!55, line width=0.3mm] (\x+2.25,0.75) rectangle (\x+3, 1.5);
            
              \draw [fill=red!55, line width=0.3mm] (\x +2.25,1.5) rectangle (\x+3, 2.25);
            \draw [fill=red!55, line width=0.3mm] (\x+2.25,2.25) rectangle (\x+3, 3);
        }

    \draw [fill=red!55, line width=0.3mm] (3,0) rectangle (3.75, 0.75);

     \foreach \x in {0, 0.75, 1.5, 2.25} { 
     \foreach \y in {0, -0.75, -1.5, -2.25}{
      \draw [fill=orange!55, line width=0.3mm] (\x,\y+3) rectangle (\x+0.75,\y+3.75);
      }
}

\end{scope}

\end{scope}


\begin{scope}[shift={(21,0)}]


\begin{scope}[shift={(17.5,-4.25)}]
\foreach \x in {0, 0.75, 1.5, 2.25} {
            \draw [fill=blue!55, line width=0.3mm] (\x,0) rectangle (\x+0.75, 0.75);

   }
    
    \foreach \x in {0, 0.75, 1.5} { 

        }

\draw [fill=green!55, line width=0.3mm] (3,3) rectangle (3.75, 3.75);        
        
         \draw [fill=red!55, line width=0.3mm] (3,1.5) rectangle (3.75, 2.25);
            
             \draw [fill=red!55, line width=0.3mm] (3,2.25) rectangle (3.75, 3);
            
            \draw [fill=red!55, line width=0.3mm] (3,0.75) rectangle (3.75, 1.5);
 
    \draw [fill=red!55, line width=0.3mm] (3,0) rectangle (3.75, 0.75);

    

\end{scope}

\begin{scope}[shift={(17.5,-10.75)}]
\foreach \x in {0, 0.75, 1.5, 2.25} {
            \draw [fill=blue!55, line width=0.3mm] (\x,0) rectangle (\x+0.75, 0.75);

    }
    
   \draw [fill=green!55, line width=0.3mm] (3,3) rectangle (3.75, 3.75);
        
      \draw [fill=red!55, line width=0.3mm] (3,1.5) rectangle (3.75, 2.25);
            \draw [fill=red!55, line width=0.3mm] (3,2.25) rectangle (3.75, 3);

            \draw [fill=red!55, line width=0.3mm] (3,0.75) rectangle (3.75, 1.5);

    \draw [fill=red!55, line width=0.3mm] (3,0) rectangle (3.75, 0.75);

\end{scope}

\begin{scope}[shift={(17.5,-17.25)}]
\foreach \x in {0, 0.75, 1.5, 2.25} {
            \draw [fill=blue!55, line width=0.3mm] (\x,0) rectangle (\x+0.75, 0.75);

    }
    
    \draw [fill=green!55, line width=0.3mm] (3,3) rectangle (3.75, 3.75);
  
       \draw [fill=red!55, line width=0.3mm] (3,0.75) rectangle (3.75, 1.5);
            
              \draw [fill=red!55, line width=0.3mm] (3,1.5) rectangle (3.75, 2.25);
            \draw [fill=red!55, line width=0.3mm] (3,2.25) rectangle (3.75, 3);

    \draw [fill=red!55, line width=0.3mm] (3,0) rectangle (3.75, 0.75);

\end{scope}

\begin{scope}[shift={(17.5,-23.75)}]
\foreach \x in {0, 0.75, 1.5, 2.25} {
            \draw [fill=blue!55, line width=0.3mm] (\x,0) rectangle (\x+0.75, 0.75);

    }

         \foreach \x in {0.75} { 
            
\draw [fill=green!55, line width=0.3mm] (\x+2.25,3) rectangle (\x+3, 3.75);            
            
            \draw [fill=red!55, line width=0.3mm] (\x+2.25,0.75) rectangle (\x+3, 1.5);
            
              \draw [fill=red!55, line width=0.3mm] (\x +2.25,1.5) rectangle (\x+3, 2.25);
            \draw [fill=red!55, line width=0.3mm] (\x+2.25,2.25) rectangle (\x+3, 3);
        }

    \draw [fill=red!55, line width=0.3mm] (3,0) rectangle (3.75, 0.75);

     \foreach \x in {0, 0.75, 1.5, 2.25} { 
     \foreach \y in {0, -0.75, -1.5, -2.25}{
      \draw [fill=orange!55, line width=0.3mm] (\x,\y+3) rectangle (\x+0.75,\y+3.75);
      }
}


\begin{scope}[shift={(0,6.5)}]
 \foreach \x in {0, 0.75, 1.5, 2.25} { 
     \foreach \y in {0, -0.75, -1.5, -2.25}{
      \draw [fill=green!55, line width=0.3mm] (\x,\y+3) rectangle (\x+0.75,\y+3.75);
      }
}
\end{scope}

\begin{scope}[shift={(0,13)}]
 \foreach \x in {0, 0.75, 1.5, 2.25} { 
     \foreach \y in {0, -0.75, -1.5, -2.25}{
      \draw [fill=green!55, line width=0.3mm] (\x,\y+3) rectangle (\x+0.75,\y+3.75);
      }
}
\end{scope}

\begin{scope}[shift={(0,19.5)}]
 \foreach \x in {0, 0.75, 1.5, 2.25} { 
     \foreach \y in {0, -0.75, -1.5, -2.25}{
      \draw [fill=green!55, line width=0.3mm] (\x,\y+3) rectangle (\x+0.75,\y+3.75);
      }
}
\end{scope}

\end{scope}

\end{scope}


 \foreach \x in {0} { 
  \foreach \y in {0} { 
\begin{scope}[shift={(17.5+\x, -4.5-\y)}] 
   
\draw[<->] (0,0) -- (3.75,0) node[midway, below, yshift=0.5mm] {$n+1$};
\draw[<->] (4,0.2) -- (4,3.95) node[midway, right, xshift=-0.5mm] {$n+1$};

\end{scope}

}
}

\draw[<->] (17,-25) -- (42.5,-25) node[midway, below, yshift=-2mm] {\huge $n$};

\draw[<->] (43.7,-24) -- (43.7,-0.5) node[midway, right, xshift=1.7mm] {\huge $n$};

\end{tikzpicture}
\vskip .8cm

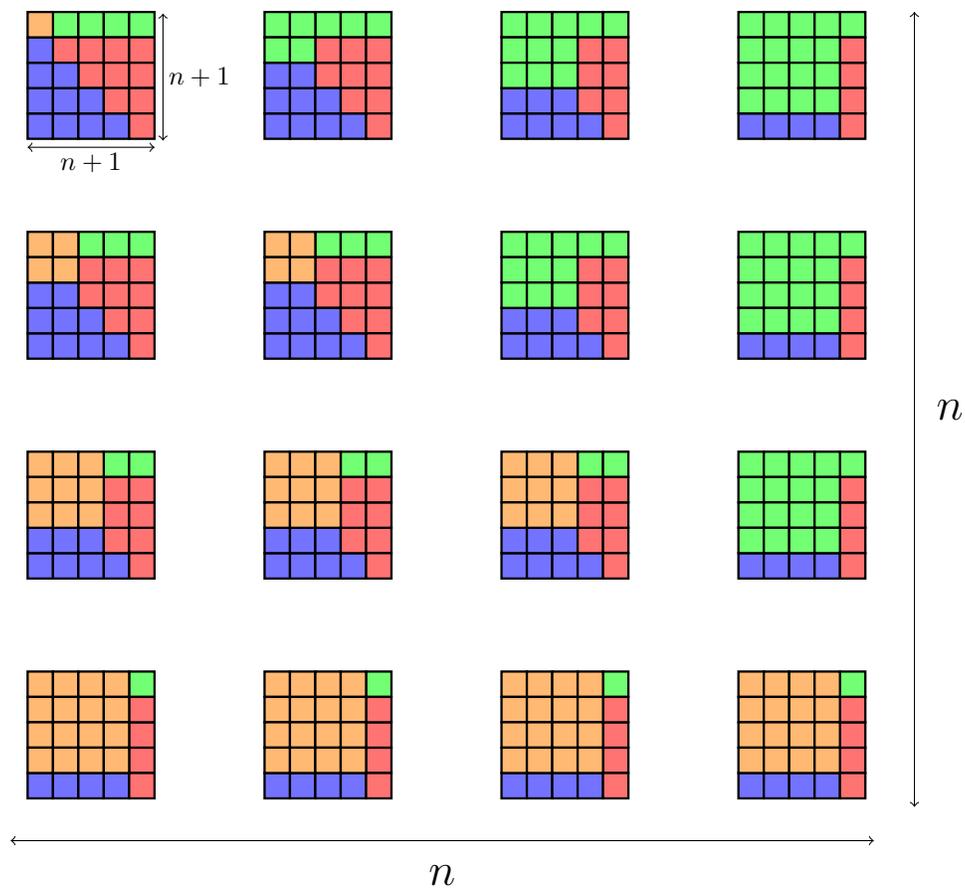
\captionof{figure}{First section with five 5D pyramids}
\label{first with 5D}

\begin{tikzpicture}[scale=0.45]


\begin{scope}[shift={(0,0)}]

\begin{scope}[shift={(17.5,-4.25)}]
\foreach \x in {0, 0.75, 1.5, 2.25} {
            \draw [fill=blue!55, line width=0.3mm] (\x,0) rectangle (\x+0.75, 0.75);

    }
    
    \foreach \x in {0, 0.75, 1.5} { 
     \draw [fill=red!55, line width=0.3mm] (\x+1.5,2.25) rectangle (\x+2.25, 3);
    \draw [fill=green!55, line width=0.3mm] (\x+1.5,3) rectangle (\x+2.25, 3.75);
            \draw [fill=blue!55, line width=0.3mm] (\x,0.75) rectangle (\x+0.75, 1.5);
            
            \draw [fill=red!55, line width=0.3mm] (\x +1.5,1.5) rectangle (\x+2.25, 2.25);
        }
        
         \foreach \x in {0, 0.75} { 
            \draw [fill=blue!55, line width=0.3mm] (\x,1.5) rectangle (\x+0.75, 2.25);
            \draw [fill=red!55, line width=0.3mm] (\x+2.25,0.75) rectangle (\x+3, 1.5);
        }

    \draw [fill=red!55, line width=0.3mm] (3,0) rectangle (3.75, 0.75);

\begin{scope}[shift={(0,0)}]
\foreach \x in {0, 0.75} { 
     \foreach \y in {0, -0.75}{
      \draw [fill=pink!40, line width=0.3mm] (\x,\y+3) rectangle (\x+0.75,\y+3.75);
      }
}
    \end{scope}

\end{scope}

\begin{scope}[shift={(17.5,-10.75)}]
\foreach \x in {0, 0.75, 1.5, 2.25} {
            \draw [fill=blue!55, line width=0.3mm] (\x,0) rectangle (\x+0.75, 0.75);

    }
    
    \foreach \x in {0, 0.75, 1.5} { 
            \draw [fill=blue!55, line width=0.3mm] (\x,0.75) rectangle (\x+0.75, 1.5);
            \draw [fill=green!55, line width=0.3mm] (\x+1.5,3) rectangle (\x+2.25, 3.75);
            
            \draw [fill=red!55, line width=0.3mm] (\x +1.5,1.5) rectangle (\x+2.25, 2.25);
            \draw [fill=red!55, line width=0.3mm] (\x+1.5,2.25) rectangle (\x+2.25, 3);
        }
        
         \foreach \x in {0, 0.75} { 
            \draw [fill=blue!55, line width=0.3mm] (\x,1.5) rectangle (\x+0.75, 2.25);
            \draw [fill=red!55, line width=0.3mm] (\x+2.25,0.75) rectangle (\x+3, 1.5);
        }

    \draw [fill=red!55, line width=0.3mm] (3,0) rectangle (3.75, 0.75);

     \foreach \x in {0, 0.75} { 
     \foreach \y in {0, -0.75}{
      \draw [fill=orange!55, line width=0.3mm] (\x,\y+3) rectangle (\x+0.75,\y+3.75);
      }
}

\end{scope}

\begin{scope}[shift={(17.5,-17.25)}]
\foreach \x in {0, 0.75, 1.5, 2.25} {
            \draw [fill=blue!55, line width=0.3mm] (\x,0) rectangle (\x+0.75, 0.75);

    }
    
    \foreach \x in {0, 0.75, 1.5} { 
            \draw [fill=blue!55, line width=0.3mm] (\x,0.75) rectangle (\x+0.75, 1.5);

        }
        
         \foreach \x in {0, 0.75} { 
             \draw [fill=green!55, line width=0.3mm] (\x+2.25,3) rectangle (\x+3, 3.75);
            \draw [fill=red!55, line width=0.3mm] (\x+2.25,0.75) rectangle (\x+3, 1.5);
            
              \draw [fill=red!55, line width=0.3mm] (\x +2.25,1.5) rectangle (\x+3, 2.25);
            \draw [fill=red!55, line width=0.3mm] (\x+2.25,2.25) rectangle (\x+3, 3);
        }

    \draw [fill=red!55, line width=0.3mm] (3,0) rectangle (3.75, 0.75);

     \foreach \x in {0, 0.75, 1.5} { 
     \foreach \y in {0, -0.75, -1.5}{
      \draw [fill=orange!55, line width=0.3mm] (\x,\y+3) rectangle (\x+0.75,\y+3.75);
      }
}

\end{scope}

\begin{scope}[shift={(17.5,-23.75)}]
\foreach \x in {0, 0.75, 1.5, 2.25} {
            \draw [fill=blue!55, line width=0.3mm] (\x,0) rectangle (\x+0.75, 0.75);

    }

         \foreach \x in {0.75} { 
            
              \draw [fill=green!55, line width=0.3mm] (\x+2.25,3) rectangle (\x+3, 3.75);
            \draw [fill=red!55, line width=0.3mm] (\x+2.25,0.75) rectangle (\x+3, 1.5);
            
              \draw [fill=red!55, line width=0.3mm] (\x +2.25,1.5) rectangle (\x+3, 2.25);
            \draw [fill=red!55, line width=0.3mm] (\x+2.25,2.25) rectangle (\x+3, 3);
        }

    \draw [fill=red!55, line width=0.3mm] (3,0) rectangle (3.75, 0.75);

     \foreach \x in {0, 0.75, 1.5, 2.25} { 
     \foreach \y in {0, -0.75, -1.5, -2.25}{
      \draw [fill=orange!55, line width=0.3mm] (\x,\y+3) rectangle (\x+0.75,\y+3.75);
      }
}

\end{scope}

\end{scope}


\begin{scope}[shift={(7,0)}]

\begin{scope}[shift={(17.5,-4.25)}]
\foreach \x in {0, 0.75, 1.5, 2.25} {
            \draw [fill=blue!55, line width=0.3mm] (\x,0) rectangle (\x+0.75, 0.75);

    }
    
    \foreach \x in {0, 0.75, 1.5} { 
  \draw [fill=green!55, line width=0.3mm] (\x+1.5,3) rectangle (\x+2.25, 3.75);    
    
            \draw [fill=blue!55, line width=0.3mm] (\x,0.75) rectangle (\x+0.75, 1.5);
            
            \draw [fill=red!55, line width=0.3mm] (\x +1.5,1.5) rectangle (\x+2.25, 2.25);
            
             \draw [fill=red!55, line width=0.3mm] (\x+1.5,2.25) rectangle (\x+2.25, 3);
        }
        
         \foreach \x in {0, 0.75} { 
            \draw [fill=blue!55, line width=0.3mm] (\x,1.5) rectangle (\x+0.75, 2.25);
            \draw [fill=red!55, line width=0.3mm] (\x+2.25,0.75) rectangle (\x+3, 1.5);
        }
    
    
    \draw [fill=red!55, line width=0.3mm] (3,0) rectangle (3.75, 0.75);

    

\end{scope}

\begin{scope}[shift={(17.5,-10.75)}]
\foreach \x in {0, 0.75, 1.5, 2.25} {
            \draw [fill=blue!55, line width=0.3mm] (\x,0) rectangle (\x+0.75, 0.75);

    }
    
    \foreach \x in {0, 0.75, 1.5} { 
    
    \draw [fill=green!55, line width=0.3mm] (\x+1.5,3) rectangle (\x+2.25, 3.75);
            \draw [fill=blue!55, line width=0.3mm] (\x,0.75) rectangle (\x+0.75, 1.5);
            
            \draw [fill=red!55, line width=0.3mm] (\x +1.5,1.5) rectangle (\x+2.25, 2.25);
            \draw [fill=red!55, line width=0.3mm] (\x+1.5,2.25) rectangle (\x+2.25, 3);
        }
        
         \foreach \x in {0, 0.75} { 
            \draw [fill=blue!55, line width=0.3mm] (\x,1.5) rectangle (\x+0.75, 2.25);
            \draw [fill=red!55, line width=0.3mm] (\x+2.25,0.75) rectangle (\x+3, 1.5);
        }

    \draw [fill=red!55, line width=0.3mm] (3,0) rectangle (3.75, 0.75);

     \foreach \x in {0, 0.75} { 
     \foreach \y in {0, -0.75}{
      \draw [fill=orange!55, line width=0.3mm] (\x,\y+3) rectangle (\x+0.75,\y+3.75);
      }
}

\begin{scope}[shift={(0,6.5)}]
 \foreach \x in {0, 0.75} { 
     \foreach \y in {0, -0.75}{
      \draw [fill=green!55, line width=0.3mm] (\x,\y+3) rectangle (\x+0.75,\y+3.75);
      }
}
\end{scope}

\end{scope}

\begin{scope}[shift={(17.5,-17.25)}]
\foreach \x in {0, 0.75, 1.5, 2.25} {
            \draw [fill=blue!55, line width=0.3mm] (\x,0) rectangle (\x+0.75, 0.75);

    }
    
    \foreach \x in {0, 0.75, 1.5} { 
            \draw [fill=blue!55, line width=0.3mm] (\x,0.75) rectangle (\x+0.75, 1.5);

        }
        
         \foreach \x in {0, 0.75} { 
            \draw [fill=green!55, line width=0.3mm] (\x+2.25,3) rectangle (\x+3, 3.75);
            
            \draw [fill=red!55, line width=0.3mm] (\x+2.25,0.75) rectangle (\x+3, 1.5);
            
              \draw [fill=red!55, line width=0.3mm] (\x +2.25,1.5) rectangle (\x+3, 2.25);
            \draw [fill=red!55, line width=0.3mm] (\x+2.25,2.25) rectangle (\x+3, 3);
        }

    \draw [fill=red!55, line width=0.3mm] (3,0) rectangle (3.75, 0.75);

     \foreach \x in {0, 0.75, 1.5} { 
     \foreach \y in {0, -0.75, -1.5}{
      \draw [fill=orange!55, line width=0.3mm] (\x,\y+3) rectangle (\x+0.75,\y+3.75);
      }
}

\end{scope}

\begin{scope}[shift={(17.5,-23.75)}]
\foreach \x in {0, 0.75, 1.5, 2.25} {
            \draw [fill=blue!55, line width=0.3mm] (\x,0) rectangle (\x+0.75, 0.75);

    }

         \foreach \x in {0.75} { 
            
\draw [fill=green!55, line width=0.3mm] (\x+2.25,3) rectangle (\x+3, 3.75);            
            
            \draw [fill=red!55, line width=0.3mm] (\x+2.25,0.75) rectangle (\x+3, 1.5);
            
              \draw [fill=red!55, line width=0.3mm] (\x +2.25,1.5) rectangle (\x+3, 2.25);
            \draw [fill=red!55, line width=0.3mm] (\x+2.25,2.25) rectangle (\x+3, 3);
        }

    \draw [fill=red!55, line width=0.3mm] (3,0) rectangle (3.75, 0.75);

     \foreach \x in {0, 0.75, 1.5, 2.25} { 
     \foreach \y in {0, -0.75, -1.5, -2.25}{
      \draw [fill=orange!55, line width=0.3mm] (\x,\y+3) rectangle (\x+0.75,\y+3.75);
      }
}

\end{scope}

\end{scope}


\begin{scope}[shift={(14,0)}]


\begin{scope}[shift={(17.5,-4.25)}]
\foreach \x in {0, 0.75, 1.5, 2.25} {
            \draw [fill=blue!55, line width=0.3mm] (\x,0) rectangle (\x+0.75, 0.75);

   }
    
    \foreach \x in {0, 0.75, 1.5} { 
            \draw [fill=blue!55, line width=0.3mm] (\x,0.75) rectangle (\x+0.75, 1.5);

        }
        
         \foreach \x in {0, 0.75} { 
         
 \draw [fill=green!55, line width=0.3mm] (\x +2.25,3) rectangle (\x+3, 3.75);         
         
          \draw [fill=red!55, line width=0.3mm] (\x +2.25,1.5) rectangle (\x+3, 2.25);
            
             \draw [fill=red!55, line width=0.3mm] (\x+2.25,2.25) rectangle (\x+3, 3);
            \draw [fill=red!55, line width=0.3mm] (\x+2.25,0.75) rectangle (\x+3, 1.5);
        }
    
    
    \draw [fill=red!55, line width=0.3mm] (3,0) rectangle (3.75, 0.75);

    

\end{scope}

\begin{scope}[shift={(17.5,-10.75)}]
\foreach \x in {0, 0.75, 1.5, 2.25} {
            \draw [fill=blue!55, line width=0.3mm] (\x,0) rectangle (\x+0.75, 0.75);

    }
    
    \foreach \x in {0, 0.75, 1.5} { 
            \draw [fill=blue!55, line width=0.3mm] (\x,0.75) rectangle (\x+0.75, 1.5);

        }
        
         \foreach \x in {0, 0.75} { 
         \draw [fill=green!55, line width=0.3mm] (\x +2.25,3) rectangle (\x+3, 3.75);
         
  \draw [fill=red!55, line width=0.3mm] (\x +2.25,1.5) rectangle (\x+3, 2.25);
            \draw [fill=red!55, line width=0.3mm] (\x+2.25,2.25) rectangle (\x+3, 3);         
         
            \draw [fill=red!55, line width=0.3mm] (\x+2.25,0.75) rectangle (\x+3, 1.5);
        }

    \draw [fill=red!55, line width=0.3mm] (3,0) rectangle (3.75, 0.75);

\end{scope}

\begin{scope}[shift={(17.5,-17.25)}]
\foreach \x in {0, 0.75, 1.5, 2.25} {
            \draw [fill=blue!55, line width=0.3mm] (\x,0) rectangle (\x+0.75, 0.75);

    }
    
    \foreach \x in {0, 0.75, 1.5} { 
            \draw [fill=blue!55, line width=0.3mm] (\x,0.75) rectangle (\x+0.75, 1.5);

        }
        
         \foreach \x in {0, 0.75} { 
            
\draw [fill=green!55, line width=0.3mm] (\x+2.25,3) rectangle (\x+3, 3.75);            
            
            \draw [fill=red!55, line width=0.3mm] (\x+2.25,0.75) rectangle (\x+3, 1.5);
            
              \draw [fill=red!55, line width=0.3mm] (\x +2.25,1.5) rectangle (\x+3, 2.25);
            \draw [fill=red!55, line width=0.3mm] (\x+2.25,2.25) rectangle (\x+3, 3);
        }

    \draw [fill=red!55, line width=0.3mm] (3,0) rectangle (3.75, 0.75);

     \foreach \x in {0, 0.75, 1.5} { 
     \foreach \y in {0, -0.75, -1.5}{
      \draw [fill=orange!55, line width=0.3mm] (\x,\y+3) rectangle (\x+0.75,\y+3.75);
      }
}

\begin{scope}[shift={(0,6.5)}]
\foreach \x in {0, 0.75, 1.5} { 
     \foreach \y in {0, -0.75, -1.5}{
      \draw [fill=green!55, line width=0.3mm] (\x,\y+3) rectangle (\x+0.75,\y+3.75);
      }
}
\end{scope}

\begin{scope}[shift={(0,13)}]
\foreach \x in {0, 0.75, 1.5} { 
     \foreach \y in {0, -0.75, -1.5}{
      \draw [fill=green!55, line width=0.3mm] (\x,\y+3) rectangle (\x+0.75,\y+3.75);
      }
}
\end{scope}

\end{scope}

\begin{scope}[shift={(17.5,-23.75)}]
\foreach \x in {0, 0.75, 1.5, 2.25} {
            \draw [fill=blue!55, line width=0.3mm] (\x,0) rectangle (\x+0.75, 0.75);

    }

         \foreach \x in {0.75} { 
             \draw [fill=green!55, line width=0.3mm] (\x+2.25,3) rectangle (\x+3, 3.75); 
             
            \draw [fill=red!55, line width=0.3mm] (\x+2.25,0.75) rectangle (\x+3, 1.5);
            
              \draw [fill=red!55, line width=0.3mm] (\x +2.25,1.5) rectangle (\x+3, 2.25);
            \draw [fill=red!55, line width=0.3mm] (\x+2.25,2.25) rectangle (\x+3, 3);
        }

    \draw [fill=red!55, line width=0.3mm] (3,0) rectangle (3.75, 0.75);

     \foreach \x in {0, 0.75, 1.5, 2.25} { 
     \foreach \y in {0, -0.75, -1.5, -2.25}{
      \draw [fill=orange!55, line width=0.3mm] (\x,\y+3) rectangle (\x+0.75,\y+3.75);
      }
}

\end{scope}

\end{scope}


\begin{scope}[shift={(21,0)}]


\begin{scope}[shift={(17.5,-4.25)}]
\foreach \x in {0, 0.75, 1.5, 2.25} {
            \draw [fill=blue!55, line width=0.3mm] (\x,0) rectangle (\x+0.75, 0.75);

   }
    
    \foreach \x in {0, 0.75, 1.5} { 

        }

\draw [fill=green!55, line width=0.3mm] (3,3) rectangle (3.75, 3.75);        
        
         \draw [fill=red!55, line width=0.3mm] (3,1.5) rectangle (3.75, 2.25);
            
             \draw [fill=red!55, line width=0.3mm] (3,2.25) rectangle (3.75, 3);
            
            \draw [fill=red!55, line width=0.3mm] (3,0.75) rectangle (3.75, 1.5);
 
    \draw [fill=red!55, line width=0.3mm] (3,0) rectangle (3.75, 0.75);

    

\end{scope}

\begin{scope}[shift={(17.5,-10.75)}]
\foreach \x in {0, 0.75, 1.5, 2.25} {
            \draw [fill=blue!55, line width=0.3mm] (\x,0) rectangle (\x+0.75, 0.75);

    }
    
   \draw [fill=green!55, line width=0.3mm] (3,3) rectangle (3.75, 3.75);
        
      \draw [fill=red!55, line width=0.3mm] (3,1.5) rectangle (3.75, 2.25);
            \draw [fill=red!55, line width=0.3mm] (3,2.25) rectangle (3.75, 3);

            \draw [fill=red!55, line width=0.3mm] (3,0.75) rectangle (3.75, 1.5);

    \draw [fill=red!55, line width=0.3mm] (3,0) rectangle (3.75, 0.75);

\end{scope}

\begin{scope}[shift={(17.5,-17.25)}]
\foreach \x in {0, 0.75, 1.5, 2.25} {
            \draw [fill=blue!55, line width=0.3mm] (\x,0) rectangle (\x+0.75, 0.75);

    }
    
    \draw [fill=green!55, line width=0.3mm] (3,3) rectangle (3.75, 3.75);
  
       \draw [fill=red!55, line width=0.3mm] (3,0.75) rectangle (3.75, 1.5);
            
              \draw [fill=red!55, line width=0.3mm] (3,1.5) rectangle (3.75, 2.25);
            \draw [fill=red!55, line width=0.3mm] (3,2.25) rectangle (3.75, 3);

    \draw [fill=red!55, line width=0.3mm] (3,0) rectangle (3.75, 0.75);

\end{scope}

\begin{scope}[shift={(17.5,-23.75)}]
\foreach \x in {0, 0.75, 1.5, 2.25} {
            \draw [fill=blue!55, line width=0.3mm] (\x,0) rectangle (\x+0.75, 0.75);

    }

         \foreach \x in {0.75} { 
            
\draw [fill=green!55, line width=0.3mm] (\x+2.25,3) rectangle (\x+3, 3.75);            
            
            \draw [fill=red!55, line width=0.3mm] (\x+2.25,0.75) rectangle (\x+3, 1.5);
            
              \draw [fill=red!55, line width=0.3mm] (\x +2.25,1.5) rectangle (\x+3, 2.25);
            \draw [fill=red!55, line width=0.3mm] (\x+2.25,2.25) rectangle (\x+3, 3);
        }

    \draw [fill=red!55, line width=0.3mm] (3,0) rectangle (3.75, 0.75);

     \foreach \x in {0, 0.75, 1.5, 2.25} { 
     \foreach \y in {0, -0.75, -1.5, -2.25}{
      \draw [fill=orange!55, line width=0.3mm] (\x,\y+3) rectangle (\x+0.75,\y+3.75);
      }
}


\begin{scope}[shift={(0,6.5)}]
 \foreach \x in {0, 0.75, 1.5, 2.25} { 
     \foreach \y in {0, -0.75, -1.5, -2.25}{
      \draw [fill=green!55, line width=0.3mm] (\x,\y+3) rectangle (\x+0.75,\y+3.75);
      }
}
\end{scope}

\begin{scope}[shift={(0,13)}]
 \foreach \x in {0, 0.75, 1.5, 2.25} { 
     \foreach \y in {0, -0.75, -1.5, -2.25}{
      \draw [fill=green!55, line width=0.3mm] (\x,\y+3) rectangle (\x+0.75,\y+3.75);
      }
}
\end{scope}

\begin{scope}[shift={(0,19.5)}]
 \foreach \x in {0, 0.75, 1.5, 2.25} { 
     \foreach \y in {0, -0.75, -1.5, -2.25}{
      \draw [fill=green!55, line width=0.3mm] (\x,\y+3) rectangle (\x+0.75,\y+3.75);
      }
}
\end{scope}

\end{scope}

\end{scope}


 \foreach \x in {0} { 
  \foreach \y in {0} { 
\begin{scope}[shift={(17.5+\x, -4.5-\y)}] 
   
\draw[<->] (0,0) -- (3.75,0) node[midway, below, yshift=0.5mm] {$n+1$};
\draw[<->] (4,0.2) -- (4,3.95) node[midway, right, xshift=-0.5mm] {$n+1$};

\end{scope}

}
}

\draw[<->] (17,-25) -- (42.5,-25) node[midway, below, yshift=-2mm] {\huge $n$};

\draw[<->] (43.7,-24) -- (43.7,-0.5) node[midway, right, xshift=1.7mm] {\huge $n$};

\end{tikzpicture}
\vskip .8cm

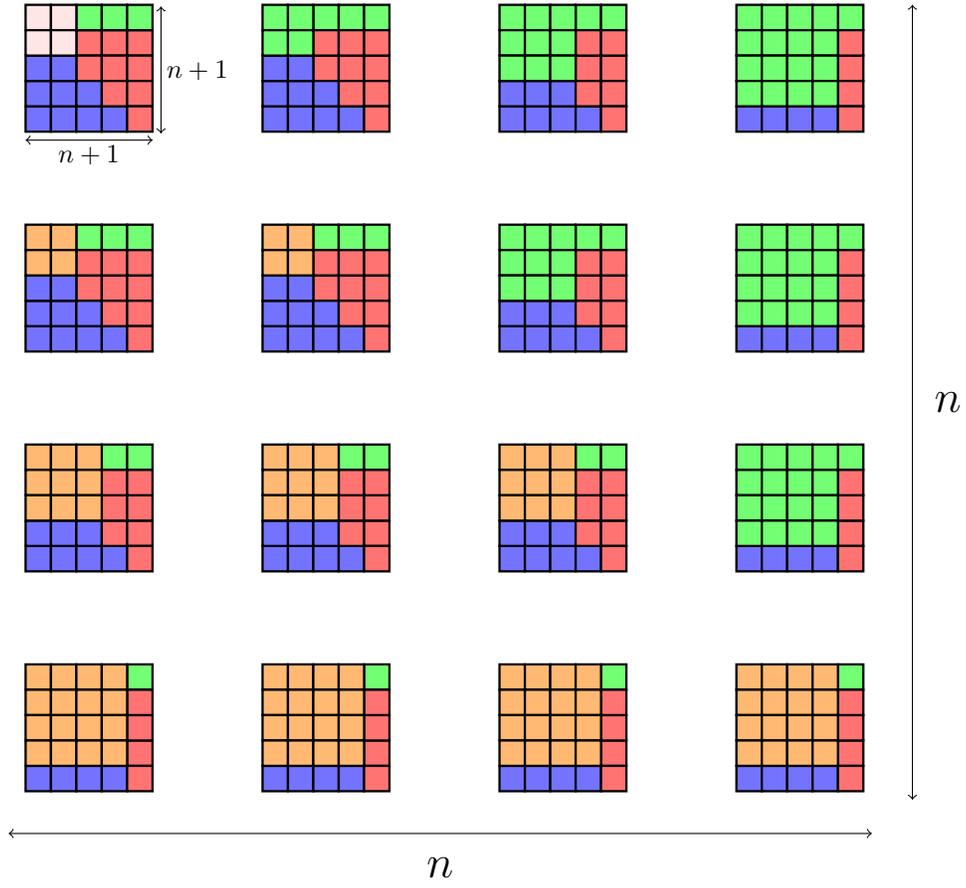
\captionof{figure}{Second section with five 5D pyramids}
\label{second with 5D}
\par
\vspace{5mm}
\noindent
\begin{tikzpicture}[scale=0.45]


\begin{scope}[shift={(0,0)}]

\begin{scope}[shift={(17.5,-4.25)}]
\foreach \x in {0, 0.75, 1.5, 2.25} {
            \draw [fill=blue!55, line width=0.3mm] (\x,0) rectangle (\x+0.75, 0.75);

    }
    
    \foreach \x in {0, 0.75, 1.5} {

            \draw [fill=blue!55, line width=0.3mm] (\x,0.75) rectangle (\x+0.75, 1.5);

        }
        
         \foreach \x in {0, 0.75} { 
         \draw [fill=green!55, line width=0.3mm] (\x+2.25,3) rectangle (\x+3, 3.75);
         
          \draw [fill=red!55, line width=0.3mm] (\x +2.25,1.5) rectangle (\x+3, 2.25);
            \draw [fill=red!55, line width=0.3mm] (\x+2.25,2.25) rectangle (\x+3, 3);
            
            \draw [fill=red!55, line width=0.3mm] (\x+2.25,0.75) rectangle (\x+3, 1.5);
        }

    \draw [fill=red!55, line width=0.3mm] (3,0) rectangle (3.75, 0.75);

\begin{scope}[shift={(0,0)}]
\foreach \x in {0, 0.75, 1.5} { 
     \foreach \y in {0, -0.75, -1.5}{
      \draw [fill=pink!40, line width=0.3mm] (\x,\y+3) rectangle (\x+0.75,\y+3.75);
      }
}
    \end{scope}

\end{scope}

\begin{scope}[shift={(17.5,-10.75)}]
\foreach \x in {0, 0.75, 1.5, 2.25} {
            \draw [fill=blue!55, line width=0.3mm] (\x,0) rectangle (\x+0.75, 0.75);

    }
    
    \foreach \x in {0, 0.75, 1.5} { 
            \draw [fill=blue!55, line width=0.3mm] (\x,0.75) rectangle (\x+0.75, 1.5);
            
        }
        
         \foreach \x in {0, 0.75} { 
\draw [fill=green!55, line width=0.3mm] (\x+2.25,3) rectangle (\x+3, 3.75);
            
            \draw [fill=red!55, line width=0.3mm] (\x +2.25,1.5) rectangle (\x+3, 2.25);
            \draw [fill=red!55, line width=0.3mm] (\x+2.25,2.25) rectangle (\x+3, 3);

            \draw [fill=red!55, line width=0.3mm] (\x+2.25,0.75) rectangle (\x+3, 1.5);
        }

    \draw [fill=red!55, line width=0.3mm] (3,0) rectangle (3.75, 0.75);

     \foreach \x in {0, 0.75, 1.5} { 
     \foreach \y in {0, -0.75, -1.5}{
      \draw [fill=pink!40, line width=0.3mm] (\x,\y+3) rectangle (\x+0.75,\y+3.75);
      }
}

\end{scope}

\begin{scope}[shift={(17.5,-17.25)}]
\foreach \x in {0, 0.75, 1.5, 2.25} {
            \draw [fill=blue!55, line width=0.3mm] (\x,0) rectangle (\x+0.75, 0.75);

    }
    
    \foreach \x in {0, 0.75, 1.5} { 
            \draw [fill=blue!55, line width=0.3mm] (\x,0.75) rectangle (\x+0.75, 1.5);

        }
        
         \foreach \x in {0, 0.75} { 
             \draw [fill=green!55, line width=0.3mm] (\x+2.25,3) rectangle (\x+3, 3.75);
            \draw [fill=red!55, line width=0.3mm] (\x+2.25,0.75) rectangle (\x+3, 1.5);
            
              \draw [fill=red!55, line width=0.3mm] (\x +2.25,1.5) rectangle (\x+3, 2.25);
            \draw [fill=red!55, line width=0.3mm] (\x+2.25,2.25) rectangle (\x+3, 3);
        }

    \draw [fill=red!55, line width=0.3mm] (3,0) rectangle (3.75, 0.75);

     \foreach \x in {0, 0.75, 1.5} { 
     \foreach \y in {0, -0.75, -1.5}{
      \draw [fill=orange!55, line width=0.3mm] (\x,\y+3) rectangle (\x+0.75,\y+3.75);
      }
}

\end{scope}

\begin{scope}[shift={(17.5,-23.75)}]
\foreach \x in {0, 0.75, 1.5, 2.25} {
            \draw [fill=blue!55, line width=0.3mm] (\x,0) rectangle (\x+0.75, 0.75);

    }

         \foreach \x in {0.75} { 
            
              \draw [fill=green!55, line width=0.3mm] (\x+2.25,3) rectangle (\x+3, 3.75);
            \draw [fill=red!55, line width=0.3mm] (\x+2.25,0.75) rectangle (\x+3, 1.5);
            
              \draw [fill=red!55, line width=0.3mm] (\x +2.25,1.5) rectangle (\x+3, 2.25);
            \draw [fill=red!55, line width=0.3mm] (\x+2.25,2.25) rectangle (\x+3, 3);
        }

    \draw [fill=red!55, line width=0.3mm] (3,0) rectangle (3.75, 0.75);

     \foreach \x in {0, 0.75, 1.5, 2.25} { 
     \foreach \y in {0, -0.75, -1.5, -2.25}{
      \draw [fill=orange!55, line width=0.3mm] (\x,\y+3) rectangle (\x+0.75,\y+3.75);
      }
}

\end{scope}

\end{scope}


\begin{scope}[shift={(7,0)}]

\begin{scope}[shift={(17.5,-4.25)}]
\foreach \x in {0, 0.75, 1.5, 2.25} {
            \draw [fill=blue!55, line width=0.3mm] (\x,0) rectangle (\x+0.75, 0.75);

    }
    
    \foreach \x in {0, 0.75, 1.5} { 
  
            \draw [fill=blue!55, line width=0.3mm] (\x,0.75) rectangle (\x+0.75, 1.5);

        }
        
         \foreach \x in {0, 0.75} { 
\draw [fill=green!55, line width=0.3mm] (\x+2.25,3) rectangle (\x+3, 3.75);

            \draw [fill=red!55, line width=0.3mm] (\x +2.25,1.5) rectangle (\x+3, 2.25);
            
             \draw [fill=red!55, line width=0.3mm] (\x+2.25,2.25) rectangle (\x+3, 3);

            \draw [fill=red!55, line width=0.3mm] (\x+2.25,0.75) rectangle (\x+3, 1.5);
        }
    
    
    \draw [fill=red!55, line width=0.3mm] (3,0) rectangle (3.75, 0.75);

    

\end{scope}

\begin{scope}[shift={(17.5,-10.75)}]
\foreach \x in {0, 0.75, 1.5, 2.25} {
            \draw [fill=blue!55, line width=0.3mm] (\x,0) rectangle (\x+0.75, 0.75);

    }
    
    \foreach \x in {0, 0.75, 1.5} {

            \draw [fill=blue!55, line width=0.3mm] (\x,0.75) rectangle (\x+0.75, 1.5);

        }
        
         \foreach \x in {0, 0.75} { 
         
         \draw [fill=green!55, line width=0.3mm] (\x+2.25,3) rectangle (\x+3, 3.75);
            \draw [fill=red!55, line width=0.3mm] (\x +2.25,1.5) rectangle (\x+3, 2.25);
            \draw [fill=red!55, line width=0.3mm] (\x+2.25,2.25) rectangle (\x+3, 3);

            \draw [fill=red!55, line width=0.3mm] (\x+2.25,0.75) rectangle (\x+3, 1.5);
        }

    \draw [fill=red!55, line width=0.3mm] (3,0) rectangle (3.75, 0.75);

     \foreach \x in {0, 0.75, 1.5} { 
     \foreach \y in {0, -0.75, -1.5}{
      \draw [fill=pink!40, line width=0.3mm] (\x,\y+3) rectangle (\x+0.75,\y+3.75);
      }
}

\begin{scope}[shift={(0,6.5)}]
 \foreach \x in {0, 0.75, 1.5} { 
     \foreach \y in {0, -0.75, -1.5}{
      \draw [fill=pink!40, line width=0.3mm] (\x,\y+3) rectangle (\x+0.75,\y+3.75);
      }
}
\end{scope}

\end{scope}

\begin{scope}[shift={(17.5,-17.25)}]
\foreach \x in {0, 0.75, 1.5, 2.25} {
            \draw [fill=blue!55, line width=0.3mm] (\x,0) rectangle (\x+0.75, 0.75);

    }
    
    \foreach \x in {0, 0.75, 1.5} { 
            \draw [fill=blue!55, line width=0.3mm] (\x,0.75) rectangle (\x+0.75, 1.5);

        }
        
         \foreach \x in {0, 0.75} { 
            \draw [fill=green!55, line width=0.3mm] (\x+2.25,3) rectangle (\x+3, 3.75);
            
            \draw [fill=red!55, line width=0.3mm] (\x+2.25,0.75) rectangle (\x+3, 1.5);
            
              \draw [fill=red!55, line width=0.3mm] (\x +2.25,1.5) rectangle (\x+3, 2.25);
            \draw [fill=red!55, line width=0.3mm] (\x+2.25,2.25) rectangle (\x+3, 3);
        }

    \draw [fill=red!55, line width=0.3mm] (3,0) rectangle (3.75, 0.75);

     \foreach \x in {0, 0.75, 1.5} { 
     \foreach \y in {0, -0.75, -1.5}{
      \draw [fill=orange!55, line width=0.3mm] (\x,\y+3) rectangle (\x+0.75,\y+3.75);
      }
}

\end{scope}

\begin{scope}[shift={(17.5,-23.75)}]
\foreach \x in {0, 0.75, 1.5, 2.25} {
            \draw [fill=blue!55, line width=0.3mm] (\x,0) rectangle (\x+0.75, 0.75);

    }

         \foreach \x in {0.75} { 
            
\draw [fill=green!55, line width=0.3mm] (\x+2.25,3) rectangle (\x+3, 3.75);            
            
            \draw [fill=red!55, line width=0.3mm] (\x+2.25,0.75) rectangle (\x+3, 1.5);
            
              \draw [fill=red!55, line width=0.3mm] (\x +2.25,1.5) rectangle (\x+3, 2.25);
            \draw [fill=red!55, line width=0.3mm] (\x+2.25,2.25) rectangle (\x+3, 3);
        }

    \draw [fill=red!55, line width=0.3mm] (3,0) rectangle (3.75, 0.75);

     \foreach \x in {0, 0.75, 1.5, 2.25} { 
     \foreach \y in {0, -0.75, -1.5, -2.25}{
      \draw [fill=orange!55, line width=0.3mm] (\x,\y+3) rectangle (\x+0.75,\y+3.75);
      }
}

\end{scope}

\end{scope}


\begin{scope}[shift={(14,0)}]


\begin{scope}[shift={(17.5,-4.25)}]
\foreach \x in {0, 0.75, 1.5, 2.25} {
            \draw [fill=blue!55, line width=0.3mm] (\x,0) rectangle (\x+0.75, 0.75);

   }
    
    \foreach \x in {0, 0.75, 1.5} { 
            \draw [fill=blue!55, line width=0.3mm] (\x,0.75) rectangle (\x+0.75, 1.5);

        }
        
         \foreach \x in {0, 0.75} { 
         
 \draw [fill=green!55, line width=0.3mm] (\x +2.25,3) rectangle (\x+3, 3.75);         
         
          \draw [fill=red!55, line width=0.3mm] (\x +2.25,1.5) rectangle (\x+3, 2.25);
            
             \draw [fill=red!55, line width=0.3mm] (\x+2.25,2.25) rectangle (\x+3, 3);
            \draw [fill=red!55, line width=0.3mm] (\x+2.25,0.75) rectangle (\x+3, 1.5);
        }
    
    
    \draw [fill=red!55, line width=0.3mm] (3,0) rectangle (3.75, 0.75);

    

\end{scope}

\begin{scope}[shift={(17.5,-10.75)}]
\foreach \x in {0, 0.75, 1.5, 2.25} {
            \draw [fill=blue!55, line width=0.3mm] (\x,0) rectangle (\x+0.75, 0.75);

    }
    
    \foreach \x in {0, 0.75, 1.5} { 
            \draw [fill=blue!55, line width=0.3mm] (\x,0.75) rectangle (\x+0.75, 1.5);

        }
        
         \foreach \x in {0, 0.75} { 
         \draw [fill=green!55, line width=0.3mm] (\x +2.25,3) rectangle (\x+3, 3.75);
         
  \draw [fill=red!55, line width=0.3mm] (\x +2.25,1.5) rectangle (\x+3, 2.25);
            \draw [fill=red!55, line width=0.3mm] (\x+2.25,2.25) rectangle (\x+3, 3);         
         
            \draw [fill=red!55, line width=0.3mm] (\x+2.25,0.75) rectangle (\x+3, 1.5);
        }

    \draw [fill=red!55, line width=0.3mm] (3,0) rectangle (3.75, 0.75);

\end{scope}

\begin{scope}[shift={(17.5,-17.25)}]
\foreach \x in {0, 0.75, 1.5, 2.25} {
            \draw [fill=blue!55, line width=0.3mm] (\x,0) rectangle (\x+0.75, 0.75);

    }
    
    \foreach \x in {0, 0.75, 1.5} { 
            \draw [fill=blue!55, line width=0.3mm] (\x,0.75) rectangle (\x+0.75, 1.5);

        }
        
         \foreach \x in {0, 0.75} { 
            
\draw [fill=green!55, line width=0.3mm] (\x+2.25,3) rectangle (\x+3, 3.75);            
            
            \draw [fill=red!55, line width=0.3mm] (\x+2.25,0.75) rectangle (\x+3, 1.5);
            
              \draw [fill=red!55, line width=0.3mm] (\x +2.25,1.5) rectangle (\x+3, 2.25);
            \draw [fill=red!55, line width=0.3mm] (\x+2.25,2.25) rectangle (\x+3, 3);
        }

    \draw [fill=red!55, line width=0.3mm] (3,0) rectangle (3.75, 0.75);

     \foreach \x in {0, 0.75, 1.5} { 
     \foreach \y in {0, -0.75, -1.5}{
      \draw [fill=orange!55, line width=0.3mm] (\x,\y+3) rectangle (\x+0.75,\y+3.75);
      }
}

\begin{scope}[shift={(0,6.5)}]
\foreach \x in {0, 0.75, 1.5} { 
     \foreach \y in {0, -0.75, -1.5}{
      \draw [fill=green!55, line width=0.3mm] (\x,\y+3) rectangle (\x+0.75,\y+3.75);
      }
}
\end{scope}

\begin{scope}[shift={(0,13)}]
\foreach \x in {0, 0.75, 1.5} { 
     \foreach \y in {0, -0.75, -1.5}{
      \draw [fill=green!55, line width=0.3mm] (\x,\y+3) rectangle (\x+0.75,\y+3.75);
      }
}
\end{scope}

\end{scope}

\begin{scope}[shift={(17.5,-23.75)}]
\foreach \x in {0, 0.75, 1.5, 2.25} {
            \draw [fill=blue!55, line width=0.3mm] (\x,0) rectangle (\x+0.75, 0.75);

    }

         \foreach \x in {0.75} { 
             \draw [fill=green!55, line width=0.3mm] (\x+2.25,3) rectangle (\x+3, 3.75); 
             
            \draw [fill=red!55, line width=0.3mm] (\x+2.25,0.75) rectangle (\x+3, 1.5);
            
              \draw [fill=red!55, line width=0.3mm] (\x +2.25,1.5) rectangle (\x+3, 2.25);
            \draw [fill=red!55, line width=0.3mm] (\x+2.25,2.25) rectangle (\x+3, 3);
        }

    \draw [fill=red!55, line width=0.3mm] (3,0) rectangle (3.75, 0.75);

     \foreach \x in {0, 0.75, 1.5, 2.25} { 
     \foreach \y in {0, -0.75, -1.5, -2.25}{
      \draw [fill=orange!55, line width=0.3mm] (\x,\y+3) rectangle (\x+0.75,\y+3.75);
      }
}

\end{scope}

\end{scope}


\begin{scope}[shift={(21,0)}]


\begin{scope}[shift={(17.5,-4.25)}]
\foreach \x in {0, 0.75, 1.5, 2.25} {
            \draw [fill=blue!55, line width=0.3mm] (\x,0) rectangle (\x+0.75, 0.75);

   }
    
    \foreach \x in {0, 0.75, 1.5} { 

        }

\draw [fill=green!55, line width=0.3mm] (3,3) rectangle (3.75, 3.75);        
        
         \draw [fill=red!55, line width=0.3mm] (3,1.5) rectangle (3.75, 2.25);
            
             \draw [fill=red!55, line width=0.3mm] (3,2.25) rectangle (3.75, 3);
            
            \draw [fill=red!55, line width=0.3mm] (3,0.75) rectangle (3.75, 1.5);
 
    \draw [fill=red!55, line width=0.3mm] (3,0) rectangle (3.75, 0.75);

    

\end{scope}

\begin{scope}[shift={(17.5,-10.75)}]
\foreach \x in {0, 0.75, 1.5, 2.25} {
            \draw [fill=blue!55, line width=0.3mm] (\x,0) rectangle (\x+0.75, 0.75);

    }
    
   \draw [fill=green!55, line width=0.3mm] (3,3) rectangle (3.75, 3.75);
        
      \draw [fill=red!55, line width=0.3mm] (3,1.5) rectangle (3.75, 2.25);
            \draw [fill=red!55, line width=0.3mm] (3,2.25) rectangle (3.75, 3);

            \draw [fill=red!55, line width=0.3mm] (3,0.75) rectangle (3.75, 1.5);

    \draw [fill=red!55, line width=0.3mm] (3,0) rectangle (3.75, 0.75);

\end{scope}

\begin{scope}[shift={(17.5,-17.25)}]
\foreach \x in {0, 0.75, 1.5, 2.25} {
            \draw [fill=blue!55, line width=0.3mm] (\x,0) rectangle (\x+0.75, 0.75);

    }
    
    \draw [fill=green!55, line width=0.3mm] (3,3) rectangle (3.75, 3.75);
  
       \draw [fill=red!55, line width=0.3mm] (3,0.75) rectangle (3.75, 1.5);
            
              \draw [fill=red!55, line width=0.3mm] (3,1.5) rectangle (3.75, 2.25);
            \draw [fill=red!55, line width=0.3mm] (3,2.25) rectangle (3.75, 3);

    \draw [fill=red!55, line width=0.3mm] (3,0) rectangle (3.75, 0.75);

\end{scope}

\begin{scope}[shift={(17.5,-23.75)}]
\foreach \x in {0, 0.75, 1.5, 2.25} {
            \draw [fill=blue!55, line width=0.3mm] (\x,0) rectangle (\x+0.75, 0.75);

    }

         \foreach \x in {0.75} { 
            
\draw [fill=green!55, line width=0.3mm] (\x+2.25,3) rectangle (\x+3, 3.75);            
            
            \draw [fill=red!55, line width=0.3mm] (\x+2.25,0.75) rectangle (\x+3, 1.5);
            
              \draw [fill=red!55, line width=0.3mm] (\x +2.25,1.5) rectangle (\x+3, 2.25);
            \draw [fill=red!55, line width=0.3mm] (\x+2.25,2.25) rectangle (\x+3, 3);
        }

    \draw [fill=red!55, line width=0.3mm] (3,0) rectangle (3.75, 0.75);

     \foreach \x in {0, 0.75, 1.5, 2.25} { 
     \foreach \y in {0, -0.75, -1.5, -2.25}{
      \draw [fill=orange!55, line width=0.3mm] (\x,\y+3) rectangle (\x+0.75,\y+3.75);
      }
}


\begin{scope}[shift={(0,6.5)}]
 \foreach \x in {0, 0.75, 1.5, 2.25} { 
     \foreach \y in {0, -0.75, -1.5, -2.25}{
      \draw [fill=green!55, line width=0.3mm] (\x,\y+3) rectangle (\x+0.75,\y+3.75);
      }
}
\end{scope}

\begin{scope}[shift={(0,13)}]
 \foreach \x in {0, 0.75, 1.5, 2.25} { 
     \foreach \y in {0, -0.75, -1.5, -2.25}{
      \draw [fill=green!55, line width=0.3mm] (\x,\y+3) rectangle (\x+0.75,\y+3.75);
      }
}
\end{scope}

\begin{scope}[shift={(0,19.5)}]
 \foreach \x in {0, 0.75, 1.5, 2.25} { 
     \foreach \y in {0, -0.75, -1.5, -2.25}{
      \draw [fill=green!55, line width=0.3mm] (\x,\y+3) rectangle (\x+0.75,\y+3.75);
      }
}
\end{scope}

\end{scope}

\end{scope}


 \foreach \x in {0} { 
  \foreach \y in {0} { 
\begin{scope}[shift={(17.5+\x, -4.5-\y)}] 
   
\draw[<->] (0,0) -- (3.75,0) node[midway, below, yshift=0.5mm] {$n+1$};
\draw[<->] (4,0.2) -- (4,3.95) node[midway, right, xshift=-0.5mm] {$n+1$};

\end{scope}

}
}

\draw[<->] (17,-25) -- (42.5,-25) node[midway, below, yshift=-2mm] {\huge $n$};

\draw[<->] (43.7,-24) -- (43.7,-0.5) node[midway, right, xshift=1.7mm] {\huge $n$};

\end{tikzpicture}
\vskip .8cm

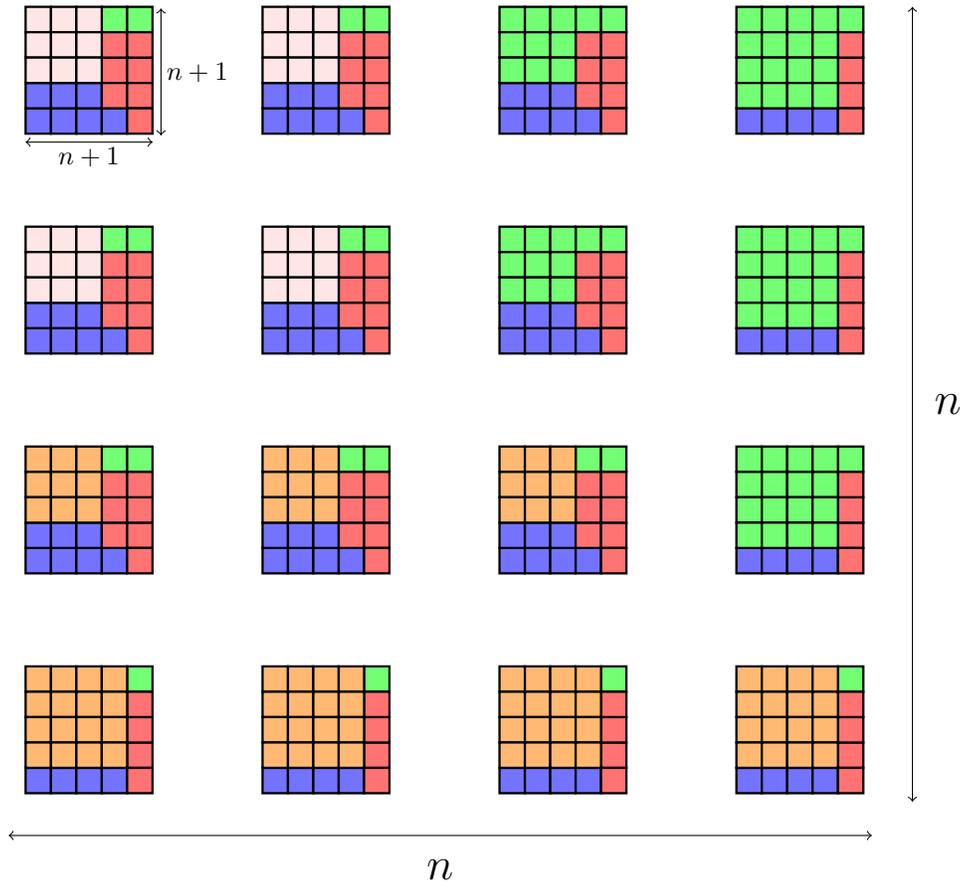
\captionof{figure}{Third section with five 5D pyramids}
\label{third with 5D}
\par
\vspace{5mm}
\noindent
\begin{tikzpicture}[scale=0.45]


\begin{scope}[shift={(0,0)}]

\begin{scope}[shift={(17.5,-4.25)}]
\foreach \x in {0, 0.75, 1.5, 2.25} {
            \draw [fill=blue!55, line width=0.3mm] (\x,0) rectangle (\x+0.75, 0.75);

    }

        \draw [fill=green!55, line width=0.3mm] (3,3) rectangle (3.75, 3.75);
         
          \draw [fill=red!55, line width=0.3mm] (3,1.5) rectangle (3.75, 2.25);
            \draw [fill=red!55, line width=0.3mm] (3,2.25) rectangle (3.75, 3);
            
            \draw [fill=red!55, line width=0.3mm] (3,0.75) rectangle (3.75, 1.5);

    \draw [fill=red!55, line width=0.3mm] (3,0) rectangle (3.75, 0.75);

\begin{scope}[shift={(0,0)}]
\foreach \x in {0, 0.75, 1.5, 2.25} { 
     \foreach \y in {0, -0.75, -1.5, -2.25}{
      \draw [fill=pink!40, line width=0.3mm] (\x,\y+3) rectangle (\x+0.75,\y+3.75);
      }
}
    \end{scope}

\end{scope}

\begin{scope}[shift={(17.5,-10.75)}]
\foreach \x in {0, 0.75, 1.5, 2.25} {
            \draw [fill=blue!55, line width=0.3mm] (\x,0) rectangle (\x+0.75, 0.75);

    }

    \draw [fill=green!55, line width=0.3mm] (3,3) rectangle (3.75, 3.75);
            
            \draw [fill=red!55, line width=0.3mm] (3,1.5) rectangle (3.75, 2.25);
            \draw [fill=red!55, line width=0.3mm] (3,2.25) rectangle (3.75, 3);

            \draw [fill=red!55, line width=0.3mm] (3,0.75) rectangle (3.75, 1.5);

    \draw [fill=red!55, line width=0.3mm] (3,0) rectangle (3.75, 0.75);

     \foreach \x in {0, 0.75, 1.5, 2.25} { 
     \foreach \y in {0, -0.75, -1.5, -2.25}{
      \draw [fill=pink!40, line width=0.3mm] (\x,\y+3) rectangle (\x+0.75,\y+3.75);
      }
}

\end{scope}

\begin{scope}[shift={(17.5,-17.25)}]
\foreach \x in {0, 0.75, 1.5, 2.25} {
            \draw [fill=blue!55, line width=0.3mm] (\x,0) rectangle (\x+0.75, 0.75);

    }

        \draw [fill=green!55, line width=0.3mm] (3,3) rectangle (3.75, 3.75);
            \draw [fill=red!55, line width=0.3mm] (3,0.75) rectangle (3.75, 1.5);
            
              \draw [fill=red!55, line width=0.3mm] (3,1.5) rectangle (3.75, 2.25);
            \draw [fill=red!55, line width=0.3mm] (3,2.25) rectangle (3.75, 3);

    \draw [fill=red!55, line width=0.3mm] (3,0) rectangle (3.75, 0.75);

     \foreach \x in {0, 0.75, 1.5, 2.25} { 
     \foreach \y in {0, -0.75, -1.5, -2.25}{
      \draw [fill=pink!40, line width=0.3mm] (\x,\y+3) rectangle (\x+0.75,\y+3.75);
      }
}

\end{scope}

\begin{scope}[shift={(17.5,-23.75)}]
\foreach \x in {0, 0.75, 1.5, 2.25} {
            \draw [fill=blue!55, line width=0.3mm] (\x,0) rectangle (\x+0.75, 0.75);

    }

         \foreach \x in {0.75} { 
            
              \draw [fill=green!55, line width=0.3mm] (\x+2.25,3) rectangle (\x+3, 3.75);
            \draw [fill=red!55, line width=0.3mm] (\x+2.25,0.75) rectangle (\x+3, 1.5);
            
              \draw [fill=red!55, line width=0.3mm] (\x +2.25,1.5) rectangle (\x+3, 2.25);
            \draw [fill=red!55, line width=0.3mm] (\x+2.25,2.25) rectangle (\x+3, 3);
        }

    \draw [fill=red!55, line width=0.3mm] (3,0) rectangle (3.75, 0.75);

     \foreach \x in {0, 0.75, 1.5, 2.25} { 
     \foreach \y in {0, -0.75, -1.5, -2.25}{
      \draw [fill=orange!55, line width=0.3mm] (\x,\y+3) rectangle (\x+0.75,\y+3.75);
      }
}

\end{scope}

\end{scope}


\begin{scope}[shift={(7,0)}]

\begin{scope}[shift={(17.5,-4.25)}]
\foreach \x in {0, 0.75, 1.5, 2.25} {
            \draw [fill=blue!55, line width=0.3mm] (\x,0) rectangle (\x+0.75, 0.75);

    }

      \draw [fill=green!55, line width=0.3mm] (3,3) rectangle (3.75, 3.75);    
            \draw [fill=red!55, line width=0.3mm] (3,1.5) rectangle (3.75, 2.25);
             \draw [fill=red!55, line width=0.3mm] (3,2.25) rectangle (3.75, 3);         
            \draw [fill=red!55, line width=0.3mm] (3,0.75) rectangle (3.75, 1.5);
    
    
    \draw [fill=red!55, line width=0.3mm] (3,0) rectangle (3.75, 0.75);

    

\end{scope}

\begin{scope}[shift={(17.5,-10.75)}]
\foreach \x in {0, 0.75, 1.5, 2.25} {
            \draw [fill=blue!55, line width=0.3mm] (\x,0) rectangle (\x+0.75, 0.75);

    }
    
    \foreach \x in {0, 0.75, 1.5} {

            \draw [fill=blue!55, line width=0.3mm] (\x,0.75) rectangle (\x+0.75, 1.5);

        }
        
         \foreach \x in {0, 0.75} { 
         
         \draw [fill=green!55, line width=0.3mm] (\x+2.25,3) rectangle (\x+3, 3.75);
            \draw [fill=red!55, line width=0.3mm] (\x +2.25,1.5) rectangle (\x+3, 2.25);
            \draw [fill=red!55, line width=0.3mm] (\x+2.25,2.25) rectangle (\x+3, 3);

            \draw [fill=red!55, line width=0.3mm] (\x+2.25,0.75) rectangle (\x+3, 1.5);
        }

    \draw [fill=red!55, line width=0.3mm] (3,0) rectangle (3.75, 0.75);

     \foreach \x in {0, 0.75, 1.5, 2.25} { 
     \foreach \y in {0, -0.75, -1.5, -2.25}{
      \draw [fill=pink!40, line width=0.3mm] (\x,\y+3) rectangle (\x+0.75,\y+3.75);
      }
}

\begin{scope}[shift={(0,6.5)}]
 \foreach \x in {0, 0.75, 1.5, 2.25} { 
     \foreach \y in {0, -0.75, -1.5, -2.25}{
      \draw [fill=pink!40, line width=0.3mm] (\x,\y+3) rectangle (\x+0.75,\y+3.75);
      }
}
\end{scope}

\end{scope}

\begin{scope}[shift={(17.5,-17.25)}]
\foreach \x in {0, 0.75, 1.5, 2.25} {
            \draw [fill=blue!55, line width=0.3mm] (\x,0) rectangle (\x+0.75, 0.75);

    }
    
    \foreach \x in {0, 0.75, 1.5} { 
            \draw [fill=blue!55, line width=0.3mm] (\x,0.75) rectangle (\x+0.75, 1.5);

        }
        
         \foreach \x in {0, 0.75} { 
            \draw [fill=green!55, line width=0.3mm] (\x+2.25,3) rectangle (\x+3, 3.75);
            
            \draw [fill=red!55, line width=0.3mm] (\x+2.25,0.75) rectangle (\x+3, 1.5);
            
              \draw [fill=red!55, line width=0.3mm] (\x +2.25,1.5) rectangle (\x+3, 2.25);
            \draw [fill=red!55, line width=0.3mm] (\x+2.25,2.25) rectangle (\x+3, 3);
        }

    \draw [fill=red!55, line width=0.3mm] (3,0) rectangle (3.75, 0.75);

     \foreach \x in {0, 0.75, 1.5, 2.25} { 
     \foreach \y in {0, -0.75, -1.5, -2.25}{
      \draw [fill=pink!40, line width=0.3mm] (\x,\y+3) rectangle (\x+0.75,\y+3.75);
      }
}

\end{scope}

\begin{scope}[shift={(17.5,-23.75)}]
\foreach \x in {0, 0.75, 1.5, 2.25} {
            \draw [fill=blue!55, line width=0.3mm] (\x,0) rectangle (\x+0.75, 0.75);

    }

         \foreach \x in {0.75} { 
            
\draw [fill=green!55, line width=0.3mm] (\x+2.25,3) rectangle (\x+3, 3.75);            
            
            \draw [fill=red!55, line width=0.3mm] (\x+2.25,0.75) rectangle (\x+3, 1.5);
            
              \draw [fill=red!55, line width=0.3mm] (\x +2.25,1.5) rectangle (\x+3, 2.25);
            \draw [fill=red!55, line width=0.3mm] (\x+2.25,2.25) rectangle (\x+3, 3);
        }

    \draw [fill=red!55, line width=0.3mm] (3,0) rectangle (3.75, 0.75);

     \foreach \x in {0, 0.75, 1.5, 2.25} { 
     \foreach \y in {0, -0.75, -1.5, -2.25}{
      \draw [fill=orange!55, line width=0.3mm] (\x,\y+3) rectangle (\x+0.75,\y+3.75);
      }
}

\end{scope}

\end{scope}


\begin{scope}[shift={(14,0)}]


\begin{scope}[shift={(17.5,-4.25)}]
\foreach \x in {0, 0.75, 1.5, 2.25} {
            \draw [fill=blue!55, line width=0.3mm] (\x,0) rectangle (\x+0.75, 0.75);

   }
    
    \foreach \x in {0, 0.75, 1.5} { 
            \draw [fill=blue!55, line width=0.3mm] (\x,0.75) rectangle (\x+0.75, 1.5);

        }
        
         \foreach \x in {0, 0.75} { 
         
 \draw [fill=green!55, line width=0.3mm] (\x +2.25,3) rectangle (\x+3, 3.75);         
         
          \draw [fill=red!55, line width=0.3mm] (\x +2.25,1.5) rectangle (\x+3, 2.25);
            
             \draw [fill=red!55, line width=0.3mm] (\x+2.25,2.25) rectangle (\x+3, 3);
            \draw [fill=red!55, line width=0.3mm] (\x+2.25,0.75) rectangle (\x+3, 1.5);
        }
    
    
    \draw [fill=red!55, line width=0.3mm] (3,0) rectangle (3.75, 0.75);

    

\end{scope}

\begin{scope}[shift={(17.5,-10.75)}]
\foreach \x in {0, 0.75, 1.5, 2.25} {
            \draw [fill=blue!55, line width=0.3mm] (\x,0) rectangle (\x+0.75, 0.75);

    }
    
    \foreach \x in {0, 0.75, 1.5} { 
            \draw [fill=blue!55, line width=0.3mm] (\x,0.75) rectangle (\x+0.75, 1.5);

        }
        
         \foreach \x in {0, 0.75} { 
         \draw [fill=green!55, line width=0.3mm] (\x +2.25,3) rectangle (\x+3, 3.75);
         
  \draw [fill=red!55, line width=0.3mm] (\x +2.25,1.5) rectangle (\x+3, 2.25);
            \draw [fill=red!55, line width=0.3mm] (\x+2.25,2.25) rectangle (\x+3, 3);         
         
            \draw [fill=red!55, line width=0.3mm] (\x+2.25,0.75) rectangle (\x+3, 1.5);
        }

    \draw [fill=red!55, line width=0.3mm] (3,0) rectangle (3.75, 0.75);

\end{scope}

\begin{scope}[shift={(17.5,-17.25)}]
\foreach \x in {0, 0.75, 1.5, 2.25} {
            \draw [fill=blue!55, line width=0.3mm] (\x,0) rectangle (\x+0.75, 0.75);

    }
    
    \foreach \x in {0, 0.75, 1.5} { 
            \draw [fill=blue!55, line width=0.3mm] (\x,0.75) rectangle (\x+0.75, 1.5);

        }
        
         \foreach \x in {0, 0.75} { 
            
\draw [fill=green!55, line width=0.3mm] (\x+2.25,3) rectangle (\x+3, 3.75);            
            
            \draw [fill=red!55, line width=0.3mm] (\x+2.25,0.75) rectangle (\x+3, 1.5);
            
              \draw [fill=red!55, line width=0.3mm] (\x +2.25,1.5) rectangle (\x+3, 2.25);
            \draw [fill=red!55, line width=0.3mm] (\x+2.25,2.25) rectangle (\x+3, 3);
        }

    \draw [fill=red!55, line width=0.3mm] (3,0) rectangle (3.75, 0.75);

     \foreach \x in {0, 0.75, 1.5, 2.25} { 
     \foreach \y in {0, -0.75, -1.5, -2.25}{
      \draw [fill=pink!40, line width=0.3mm] (\x,\y+3) rectangle (\x+0.75,\y+3.75);
      }
}

\begin{scope}[shift={(0,6.5)}]
\foreach \x in {0, 0.75, 1.5, 2.25} { 
     \foreach \y in {0, -0.75, -1.5, -2.25}{
      \draw [fill=pink!40, line width=0.3mm] (\x,\y+3) rectangle (\x+0.75,\y+3.75);
      }
}
\end{scope}

\begin{scope}[shift={(0,13)}]
\foreach \x in {0, 0.75, 1.5, 2.25} { 
     \foreach \y in {0, -0.75, -1.5, -2.25}{
      \draw [fill=pink!40, line width=0.3mm] (\x,\y+3) rectangle (\x+0.75,\y+3.75);
      }
}
\end{scope}

\end{scope}

\begin{scope}[shift={(17.5,-23.75)}]
\foreach \x in {0, 0.75, 1.5, 2.25} {
            \draw [fill=blue!55, line width=0.3mm] (\x,0) rectangle (\x+0.75, 0.75);

    }

         \foreach \x in {0.75} { 
             \draw [fill=green!55, line width=0.3mm] (\x+2.25,3) rectangle (\x+3, 3.75); 
             
            \draw [fill=red!55, line width=0.3mm] (\x+2.25,0.75) rectangle (\x+3, 1.5);
            
              \draw [fill=red!55, line width=0.3mm] (\x +2.25,1.5) rectangle (\x+3, 2.25);
            \draw [fill=red!55, line width=0.3mm] (\x+2.25,2.25) rectangle (\x+3, 3);
        }

    \draw [fill=red!55, line width=0.3mm] (3,0) rectangle (3.75, 0.75);

     \foreach \x in {0, 0.75, 1.5, 2.25} { 
     \foreach \y in {0, -0.75, -1.5, -2.25}{
      \draw [fill=orange!55, line width=0.3mm] (\x,\y+3) rectangle (\x+0.75,\y+3.75);
      }
}

\end{scope}

\end{scope}


\begin{scope}[shift={(21,0)}]


\begin{scope}[shift={(17.5,-4.25)}]
\foreach \x in {0, 0.75, 1.5, 2.25} {
            \draw [fill=blue!55, line width=0.3mm] (\x,0) rectangle (\x+0.75, 0.75);

   }
    
    \foreach \x in {0, 0.75, 1.5} { 

        }

\draw [fill=green!55, line width=0.3mm] (3,3) rectangle (3.75, 3.75);        
        
         \draw [fill=red!55, line width=0.3mm] (3,1.5) rectangle (3.75, 2.25);
            
             \draw [fill=red!55, line width=0.3mm] (3,2.25) rectangle (3.75, 3);
            
            \draw [fill=red!55, line width=0.3mm] (3,0.75) rectangle (3.75, 1.5);
 
    \draw [fill=red!55, line width=0.3mm] (3,0) rectangle (3.75, 0.75);

    

\end{scope}

\begin{scope}[shift={(17.5,-10.75)}]
\foreach \x in {0, 0.75, 1.5, 2.25} {
            \draw [fill=blue!55, line width=0.3mm] (\x,0) rectangle (\x+0.75, 0.75);

    }
    
   \draw [fill=green!55, line width=0.3mm] (3,3) rectangle (3.75, 3.75);
        
      \draw [fill=red!55, line width=0.3mm] (3,1.5) rectangle (3.75, 2.25);
            \draw [fill=red!55, line width=0.3mm] (3,2.25) rectangle (3.75, 3);

            \draw [fill=red!55, line width=0.3mm] (3,0.75) rectangle (3.75, 1.5);

    \draw [fill=red!55, line width=0.3mm] (3,0) rectangle (3.75, 0.75);

\end{scope}

\begin{scope}[shift={(17.5,-17.25)}]
\foreach \x in {0, 0.75, 1.5, 2.25} {
            \draw [fill=blue!55, line width=0.3mm] (\x,0) rectangle (\x+0.75, 0.75);

    }
    
    \draw [fill=green!55, line width=0.3mm] (3,3) rectangle (3.75, 3.75);
  
       \draw [fill=red!55, line width=0.3mm] (3,0.75) rectangle (3.75, 1.5);
            
              \draw [fill=red!55, line width=0.3mm] (3,1.5) rectangle (3.75, 2.25);
            \draw [fill=red!55, line width=0.3mm] (3,2.25) rectangle (3.75, 3);

    \draw [fill=red!55, line width=0.3mm] (3,0) rectangle (3.75, 0.75);

\end{scope}

\begin{scope}[shift={(17.5,-23.75)}]
\foreach \x in {0, 0.75, 1.5, 2.25} {
            \draw [fill=blue!55, line width=0.3mm] (\x,0) rectangle (\x+0.75, 0.75);

    }

         \foreach \x in {0.75} { 
            
\draw [fill=green!55, line width=0.3mm] (\x+2.25,3) rectangle (\x+3, 3.75);            
            
            \draw [fill=red!55, line width=0.3mm] (\x+2.25,0.75) rectangle (\x+3, 1.5);
            
              \draw [fill=red!55, line width=0.3mm] (\x +2.25,1.5) rectangle (\x+3, 2.25);
            \draw [fill=red!55, line width=0.3mm] (\x+2.25,2.25) rectangle (\x+3, 3);
        }

    \draw [fill=red!55, line width=0.3mm] (3,0) rectangle (3.75, 0.75);

     \foreach \x in {0, 0.75, 1.5, 2.25} { 
     \foreach \y in {0, -0.75, -1.5, -2.25}{
      \draw [fill=orange!55, line width=0.3mm] (\x,\y+3) rectangle (\x+0.75,\y+3.75);
      }
}


\begin{scope}[shift={(0,6.5)}]
 \foreach \x in {0, 0.75, 1.5, 2.25} { 
     \foreach \y in {0, -0.75, -1.5, -2.25}{
      \draw [fill=green!55, line width=0.3mm] (\x,\y+3) rectangle (\x+0.75,\y+3.75);
      }
}
\end{scope}

\begin{scope}[shift={(0,13)}]
 \foreach \x in {0, 0.75, 1.5, 2.25} { 
     \foreach \y in {0, -0.75, -1.5, -2.25}{
      \draw [fill=green!55, line width=0.3mm] (\x,\y+3) rectangle (\x+0.75,\y+3.75);
      }
}
\end{scope}

\begin{scope}[shift={(0,19.5)}]
 \foreach \x in {0, 0.75, 1.5, 2.25} { 
     \foreach \y in {0, -0.75, -1.5, -2.25}{
      \draw [fill=green!55, line width=0.3mm] (\x,\y+3) rectangle (\x+0.75,\y+3.75);
      }
}
\end{scope}

\end{scope}

\end{scope}


 \foreach \x in {0} { 
  \foreach \y in {0} { 
\begin{scope}[shift={(17.5+\x, -4.5-\y)}] 
   
\draw[<->] (0,0) -- (3.75,0) node[midway, below, yshift=0.5mm] {$n+1$};
\draw[<->] (4,0.2) -- (4,3.95) node[midway, right, xshift=-0.5mm] {$n+1$};

\end{scope}

}
}

\draw[<->] (17,-25) -- (42.5,-25) node[midway, below, yshift=-2mm] {\huge $n$};

\draw[<->] (43.7,-24) -- (43.7,-0.5) node[midway, right, xshift=1.7mm] {\huge $n$};

\end{tikzpicture}
\vskip .8cm
\captionof{figure}{$n$-th section with 5D pyramids}
\label{n with 5D}
\par
\vspace{5mm}
\noindent
\begin{tikzpicture}[scale=0.47]

\node at (15,-11) {\hskip 1.5cm \huge $1\cdot 1^2 + 3\cdot 2^2+ 5\cdot 3^2+\cdots +(2n-1)n^2$};

\begin{scope}[shift={(6,0)}] 

\foreach \p in {0, 5, 10, 15} {
\foreach \q in {0, 5, 10, 15} {
\begin{scope}[shift={(\p, \q-33)}]
 \foreach \x in {0, 0.75, 1.5, 2.25} { 
     \foreach \y in {0, -0.75, -1.5, -2.25}{
      \draw [fill=pink!40, line width=0.3mm] (\x,\y+3) rectangle (\x+0.75,\y+3.75);
      }
}
\end{scope}

}
}


\begin{scope}[shift={(13, -18)}] 
   
\draw[<->] (2,0.4) -- (5,0.4) node[midway, below, yshift=0.5mm] {$n$};
\draw[<->] (5.4,1-0.25) -- (5.4,4-0.25) node[midway, right, xshift=-0.5mm] {$n$};

\end{scope}

\begin{scope}[shift={(-15, -8)}] 

\draw[<->] (15,-25.3) -- (33,-25.3) node[midway, below, yshift=-2mm] {\Huge $n$};

\draw[<->] (34.2,-24.3) -- (34.2,-6.3) node[midway, right, xshift=1.7mm] {\Huge $n$};


\fill[white] (14.5,-5.75) rectangle (28.5,-19.75);

\end{scope}

\begin{scope}[shift={(0.7, -14.7)}] 
\foreach \p in {0, 5, 10} {
\foreach \q in {0, 5, 10} {
\begin{scope}[shift={(\p, \q-14)}]
 \foreach \x in {0, 0.75, 1.5} { 
     \foreach \y in {0, -0.75, -1.5}{
      \draw [fill=pink!40, line width=0.3mm] (\x,\y+3) rectangle (\x+0.75,\y+3.75);
      }
}
\end{scope}

}
}
\end{scope}

\fill[white] (0,-13) rectangle (10,-23);

\begin{scope}[shift={(1.5, -20.4)}] 
\foreach \p in {0, 5} {
\foreach \q in {0, 5} {
\begin{scope}[shift={(\p, \q-4)}]
 \foreach \x in {0, 0.75} { 
     \foreach \y in {0, -0.75}{
      \draw [fill=pink!40, line width=0.3mm] (\x,\y+3) rectangle (\x+0.75,\y+3.75);
      }
}
\end{scope}

\fill[white] (-0.2,3) rectangle (1.8, 5) ;
}
}
\end{scope}

\draw [fill=pink!40, line width=0.3mm] (0+2.2,3-20.2) rectangle (0.75+2.2,3.75-20.2);

\end{scope}
\end{tikzpicture}
\vskip .8cm
\captionof{figure}{Resulting excess from assembling five 5D pyramids}
\label{Convolution}

\restoreparindent As a result, we obtain
$n$ sections, each formed by an
$n\times n$ array of squares of side length 
$(n+1)$ and one final section corresponding to what remains of the last 5D pyramid after subtracting the squares used to fill the gaps, as shown in Figure~\ref{Convolution}. Since in total we have five 5D pyramids, that is, 
$5(1^4+2^4+ 3^4+\cdots + n^4)$, we have thus given a visual proof of the following identity
\begin{equation}\label{Archimedes generalisation}
\begin{aligned}
&5(1^4+2^4+ 3^4+\cdots + n^4) = \\
&n^3(n+1)^2 + \Bigl( 1\cdot 1^2 + 3 \cdot 2^2+ 5 \cdot 3^2 +\cdots +  (2n-1)n^2\Bigr).
\end{aligned}
\end{equation}

\subsection{Step 2: Starting the DIY. The Factors $n(n+1)$}

We might think that most of the work is already done, since we now have 
$n$ sections, each section given by an $n\times n$ array of identical rectangles with dimensions 
$(n+1)\times (n+1)$, which looks very convenient. However, we still have a final layer in our construction where we have 
$1\cdot 1^2 + 3 \cdot 2^2+ 5 \cdot 3^2 +\cdots +  (2n-1)n^2$ unit squares and, even worse, we have no clue where the factor $(n^2+n-\frac{1}{3})$ might appear. The adjustment we will perform in this section may look naive, but it is important to complete our 2D puzzle.

It is based on the trivial identity
$$
n^2(n+1)^2 = n(n+1)\,n(n+1).
$$

Although this may seem obvious algebraically, it is visually significant. What we want to achieve is to transform layers with $n\times n$ rectangles of dimensions
$(n+1)\times (n+1)$ into layers with $n\times (n+1)$ rectangles of dimensions $n\times (n+1)$.

As indicated in Figure~\ref{Factors n(n+1)} (where we have removed the colors, since the first
$n$ layers are made up of identical rectangles), this can be done by removing one row of unit squares from each rectangle to form a new row of rectangles.\par
\vspace{5mm}
\noindent
\begin{tikzpicture}[scale=0.40]

\foreach \s in {6, 12, 18, 24}{
\foreach \t in {0, 7, 14, 21} {
\begin{scope}[shift={(\t,\s)}]
\foreach \z in {0, 0.75, 1.5, 2.25, 3} {
 \draw [fill=white!20, dash pattern=on 1pt off 1pt, line width=0.3mm] (\z,-0.75) rectangle (\z+0.75, 0);
}
\end{scope}
}
}

\foreach \p in {0, 7, 14, 21}{ 
\foreach \q in {0, 6, 12, 18, 24}{
\begin{scope}[shift={(\p,\q)}]

\foreach \x in {0, 0.75, 1.5, 2.25, 3} {
\foreach \y in {0, 0.75, 1.5, 2.25}{
            \draw [fill=gray!20, line width=0.3mm] (\x,\y) rectangle (\x+0.75,\y + 0.75);
}
}
\end{scope}
}
}

\draw [dash pattern=on 1pt off 1pt, ultra thick] (-1,-1) rectangle (25.7,4);

\draw[<->] (0,-1.7) -- (25,-1.7) node[midway, below, yshift=-2mm] {\Large $n$};

\draw[<->] (26.3,0) -- (26.3,27) node[midway, right, xshift=1.7mm] {\Large $n+1$};


\begin{scope}[shift={(19, 25)}] 
   
\draw[<->] (2,2.5) -- (5.75,2.5) node[midway, above, yshift=0.5mm] {$n+1$};
\draw[<->] (6.3,-1) -- (6.3,2) node[midway, right, xshift=-0.5mm] {$n$};

\end{scope}
\end{tikzpicture}
\vskip .8cm

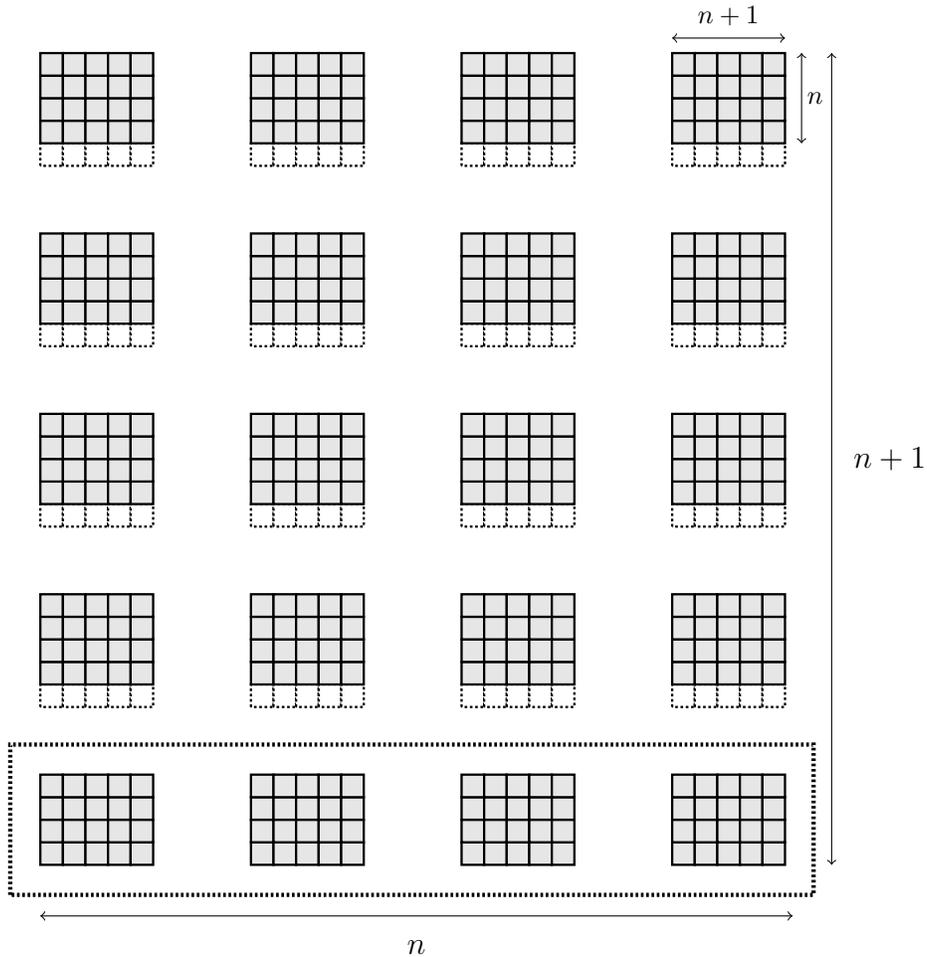
\captionof{figure}{Visual explanation of the factors $n(n+1)$}
\label{Factors n(n+1)}

The reason for this modification is that we want the factors 
$n(n+1)$ to come from the number of rectangles in each layer, so that the mysterious factor $n^2 + n - \frac{1}{3}$ arises from the adjustments we perform on each individual rectangle, as we will see in the next step.

\subsection{Step 3: Taking Out the Scissors. The Irreducible Factor $(n^2 + n -\frac{1}{3})$}

The issue posed by the factor $n^2 + n - \frac{1}{3}$ is that it is
irreducible over the field of rational numbers, although it factors over
the field of real numbers. In fact, we can write this factor as

$$n^2 + n - \frac{1}{3} = (n - x_1)(n - x_2),$$
where $x_1 = -\frac{1}{2} + \sqrt{\frac{7}{12}}$ and
$x_2 = -\frac{1}{2} - \sqrt{\frac{7}{12}}$ are its real roots. Thus,
$x_2 = -1 - x_1$, and by setting $x = x_1$ we can write
\begin{equation} \label{factorization}
n^2 + n - \frac{1}{3} = (n - x)(n + 1 + x).
\end{equation}

\restoreparindent This tells us exactly what kind of DIY operation we need to perform on each rectangle with dimensions $(n+1)\times n$ in the first $n$ layers of our 2D puzzle. Specifically, we cut a thin strip of width $x$ along the length of each rectangle, and in the last rectangle we make an additional vertical cut at a distance $x$ from the edge. These pieces can then be rearranged vertically at the right end of each rectangle, forming a new rectangle with dimensions $(n+1+x)\times (n-x)$.
\vskip .2cm
\begin{tikzpicture}[scale=0.9]
 \draw[line width=0mm] (0,0) -- (0,0);
\begin{scope}[shift={(3,0)}]

\foreach \x in {0, 0.75, 1.5, 2.25, 3} {
\foreach \y in {0, 0.75, 1.5, 2.25}{
            \draw [fill=gray!20, line width=0.3mm] (\x,\y) rectangle (\x+0.75,\y + 0.75);
}
}
\end{scope}


\begin{scope}[shift={(0, 0)}] 
   
\draw[<->] (3,-0.2) -- (6.75,-0.2) node[midway, below] {\Large $n+1$};
\draw[<->] (2.8,0) -- (2.8,3) node[midway, left] {\Large $n$};
\draw[<->] (7,2.75) -- (7,3) node[midway, right] {\Large $x$};
\draw[<->] (6.5, 3.2) -- (6.75,3.2) node[midway, above] {\Large $x$};

\end{scope}

\draw[dashed] (2.75, 2.75) -- (7,2.75);
\draw[dashed] (6.5, 2.75) -- (6.5,3.25);
\draw[dashed] (0,-1) -- (11,-1);
\end{tikzpicture}
\vskip .3cm
\begin{tikzpicture}[scale=0.9]
 \draw[line width=0mm] (0,0) -- (0,0);
\begin{scope}[shift={(3,0)}]

\foreach \x in {0, 0.75, 1.5, 2.25, 3} {
\foreach \y in {0, 0.75, 1.5, 2.25}{
            \draw [fill=gray!20, line width=0.3mm] (\x,\y) rectangle (\x+0.75,\y + 0.75);
}
}

\foreach \x in {3.75} {
\foreach \y in {0, 0.75, 1.5, 2.25}{
            \draw [fill=gray!20, line width=0.3mm] (\x,\y) rectangle (\x+0.25,\y + 0.75);
}
}

\foreach \x in {0, 0.75, 1.5, 2.25, 3, 3.75} {
\foreach \y in {3}{
            \draw [draw=none, fill=white] (\x-0.2,\y+0.2) rectangle (\x+0.8,\y -0.25);
}
}


 \draw [fill=gray!20, line width=0.3mm] (5,2.75) rectangle (5.75,3);

 \draw [fill=gray!20, line width=0.3mm] (6,2.75) rectangle (6.25,3);
 
\end{scope}


\begin{scope}[shift={(0, 0)}] 
   
\draw[<->] (3,-0.2) -- (7,-0.2) node[midway, below] {\Large $n+1+x$};
\draw[<->] (2.8,0) -- (2.8,2.75) node[midway, left] {\Large $n-x$};

\end{scope}

\draw [line width=0.3mm](3, 2.75) -- (7,2.75);

\end{tikzpicture}
\vskip .2cm
\captionof{figure}{Scissor DIY explanation of the factor $(n^2 + n -\frac{1}{3})$.}
\label{Scissors}

\restoreparindent In this process, two extra pieces are left over, which we place in the last layer that was still to be completed (see Figure~\ref{Convolution}).

\restoreparindent Note that the operation carried out in this section can be performed for any $0 \le x \le 1$. 
However, the reason why $x$ must be chosen to be exactly the positive root of the polynomial 
$n^2 + n - \tfrac{1}{3}$ will become clear in the next section.

\subsection{Step 4: Completing the Top Layer. The Factor $(n+\tfrac{1}{2})$}

The explanation of the factor $(n+\tfrac{1}{2})$ is similar to the DIY operation performed in Diagram~\ref{DIY} for the sum of squares. In other words, the top layer of the block we want to construct must have an area equal to half of the area of the first $n$ layers.

\restoreparindent However, unlike the operation carried out in Diagram~\ref{DIY}, in our situation an adjustment has already been performed in the first $n$ layers. Indeed, in each rectangle of area $n(n+1)$ in Figure~\ref{Factors n(n+1)}, a leftover of area $\mathcal{R} = x^2 + x$ has been removed (see Figure~\ref{Leftover}). 

\vskip .5cm

\begin{tikzpicture}[scale=2]
\draw[line width=0mm] (0,0) -- (0,0); 

\begin{scope}[shift={(-2.75, -3)}] 

\draw [fill=gray!20, line width=0.3mm] (5,2.75) rectangle (5.75,3);

 \draw [fill=gray!20, line width=0.3mm] (6.5,2.75) rectangle (6.75,3);


\draw[<->] (6.5,2.66) -- (6.75,2.66) node[midway, below ] {\Large $x$};

\draw[<->] (6.4,2.75) -- (6.4,3) node[midway, left] {\Large $x$};

\draw[<->] (5,2.66) -- (5.75,2.66) node[midway, below ] {\Large $1$};

\draw[<->] (4.9,2.75) -- (4.9,3) node[midway, left] {\Large $x$};
 
\end{scope}
\end{tikzpicture}
\captionof{figure}{Leftover pieces from each rectangle.}
\label{Leftover}

\restoreparindent  As a result, each of the first $n$ layers has area
\[
n^2(n+1)^2 - \mathcal{R}\, n(n+1).
\]

\restoreparindent We must now check that this area coincides with the result of adding the leftovers from the first $n$ layers, that is, $\mathcal{R}\, n^2(n+1)$, to the last layer (Figure~\ref{Convolution}) and then doubling this quantity.

\restoreparindent In order to verify this equality in a visual way, we need a geometric argument that can be regarded as a kind of commutativity for figurate numbers. 

\begin{tikzpicture}[scale=0.45]

 \draw (0,0) -- (0,0);

\begin{scope}[shift={(3.5,14)}] 

\foreach \p in {0, 5, 10, 15} {
\foreach \q in {0, 5, 10, 15} {
\begin{scope}[shift={(\p, \q-33)}]
 \foreach \x in {0, 0.75, 1.5, 2.25} { 
     \foreach \y in {0, -0.75, -1.5, -2.25}{
      \draw [fill=pink!40, line width=0.3mm] (\x,\y+3) rectangle (\x+0.75,\y+3.75);
      }
}
\end{scope}

}
}


\begin{scope}[shift={(13, -18)}] 
   
\draw[<->] (2,0.4) -- (5,0.4) node[midway, below, yshift=0.5mm] {$n$};
\draw[<->] (5.4,1-0.25) -- (5.4,4-0.25) node[midway, right, xshift=-0.5mm] {$n$};

\end{scope}

\begin{scope}[shift={(-15, -8)}] 

\draw[<->] (15,-25.3) -- (33,-25.3) node[midway, below, yshift=-2mm] {\Huge $n$};

\draw[<->] (34.2,-24.3) -- (34.2,-6.3) node[midway, right, xshift=1.7mm] {\Huge $n$};


\fill[white] (14.5,-5.75) rectangle (28.5,-19.75);

\end{scope}

\fill[white] (0,-13) rectangle (10,-23);

\end{scope}

\node at (8,-8) {\hskip 1.5cm \huge $n^2(2n-1)$};
\end{tikzpicture}
\vskip .3cm
\captionof{figure}{A corner-shaped arrangement formed by squares}
\label{Odd Square}
\restoreparindent Suppose we start with an arrangement of a given shape (of size $n$) and place, in each “entry” of this shape, a figurate pattern of another type but with the same size $n$. The total number of unit squares is the same as if we had started with an arrangement of the second shape and, in each of its entries, placed the figurate pattern of the first shape.

\restoreparindent For example, the number of unit squares in a corner-shaped arrangement whose entries are squares (see Figure~\ref{Odd Square}) is the same as in a square arrangement whose entries are corner-shaped (see Figure~\ref{Square Odd}). In both cases, the total number of unit squares is $n^2(2n-1)$.

\begin{tikzpicture}[scale=0.45]

 \draw (0,0) -- (0,0);

\node at (11,0) {\hskip 1.5cm \huge $(2n-1)n^2$};

\begin{scope}[shift={(3,12)}] 

\foreach \p in {0, 5, 10, 15} {
\foreach \q in {0, 5, 10, 15} {
\begin{scope}[shift={(\p, \q-33)}]
 \foreach \x in {2.25} { 
     \foreach \y in {0, -0.75, -1.5, -2.25}{
      \draw [fill=pink!40, line width=0.3mm] (\x,\y+3) rectangle (\x+0.75,\y+3.75);

      }
}

  \foreach \x in {0, 0.75, 1.5, 2.25} { 
     \foreach \y in {-2.25}{
      \draw [fill=pink!40, line width=0.3mm] (\x,\y+3) rectangle (\x+0.75,\y+3.75);
      }
      }
\end{scope}

}
}


\begin{scope}[shift={(-2, -18)}] 
   
\draw[<->] (2,0.4) -- (5,0.4) node[midway, below, yshift=0.5mm] {$n$};
\draw[<->] (5.4,1-0.25) -- (5.4,4-0.25) node[midway, right, xshift=-0.5mm] {$n$};

\end{scope}

\begin{scope}[shift={(-15, -8)}] 

\draw[<->] (15,-25.3) -- (33,-25.3) node[midway, below, yshift=-2mm] {\Huge $n$};

\draw[<->] (34.2,-24.3) -- (34.2,-6.3) node[midway, right, xshift=1.7mm] {\Huge $n$};


\end{scope}

\end{scope}
\end{tikzpicture}
\vskip .3cm
\captionof{figure}{A square arrangement formed by corner-shaped figures}
\label{Square Odd}

\restoreparindent Thus, the final layer, consisting of 
$1\cdot 1^2 + 3\cdot 2^2+5\cdot 3^2+\cdots +(2n-1)n^2$ unit squares, can be organized into square arrangements made of corner-shaped figures, as illustrated in Figure \ref{Top Dual}.

\vskip .3cm
\begin{tikzpicture}[scale=0.47]

\draw [fill=pink!300, line width=0.3mm] (0,3) rectangle (0.75,3.75);

\foreach \p in {0, 2} {
\foreach \q in {0, 2} {
\begin{scope}[shift={(\p, \q-4)}]
 \foreach \x in {0.75} { 
     \foreach \y in {0, -0.75}{
      \draw [fill=pink!200, line width=0.3mm] (\x,\y+3) rectangle (\x+0.75,\y+3.75);
      }
}

\foreach \x in {0, 0.75} { 
     \foreach \y in {-0.75}{
      \draw [fill=pink!200, line width=0.3mm] (\x,\y+3) rectangle (\x+0.75,\y+3.75);
      }
}
\end{scope}

}
}

\foreach \p in {0, 3.5, 7} {
\foreach \q in {0, 3.5, 7} {
\begin{scope}[shift={(\p, \q-14)}]
 \foreach \x in {1.5} { 
     \foreach \y in {0, -0.75, -1.5}{
      \draw [fill=pink!120, line width=0.3mm] (\x,\y+3) rectangle (\x+0.75,\y+3.75);
      }
}

\foreach \x in {0, 0.75, 1.5} { 
     \foreach \y in {-1.5}{
      \draw [fill=pink!120, line width=0.3mm] (\x,\y+3) rectangle (\x+0.75,\y+3.75);
      }
}
\end{scope}

}
}

\foreach \p in {0, 5, 10, 15} {
\foreach \q in {0, 5, 10, 15} {
\begin{scope}[shift={(\p, \q-33)}]
 \foreach \x in { 2.25} { 
     \foreach \y in {0, -0.75, -1.5, -2.25}{
      \draw [fill=pink!40, line width=0.3mm] (\x,\y+3) rectangle (\x+0.75,\y+3.75);
      }
}

 \foreach \x in {0, 0.75, 1.5, 2.25} { 
     \foreach \y in {-2.25}{
      \draw [fill=pink!40, line width=0.3mm] (\x,\y+3) rectangle (\x+0.75,\y+3.75);
      }
}

\end{scope}

}
}

\node at (15.5,0) {\LARGE $1\cdot 1^2 +3\cdot 2^2+5\cdot 3^2+\cdots +(2n-1)n^2$};


\begin{scope}[shift={(-2, -18)}] 
   
\draw[<->] (2,0.4) -- (5,0.4) node[midway, below, yshift=0.5mm] {$n$};
\draw[<->] (5.4,1-0.25) -- (5.4,4-0.25) node[midway, right, xshift=-0.5mm] {$n$};

\end{scope}

\begin{scope}[shift={(-15, -8)}] 

\draw[<->] (15,-25.3) -- (33,-25.3) node[midway, below, yshift=-2mm] {\Huge $n$};

\draw[<->] (34.2,-24.3) -- (34.2,-6.3) node[midway, right, xshift=1.7mm] {\Huge $n$};



\end{scope}
\end{tikzpicture}

\vskip .8cm

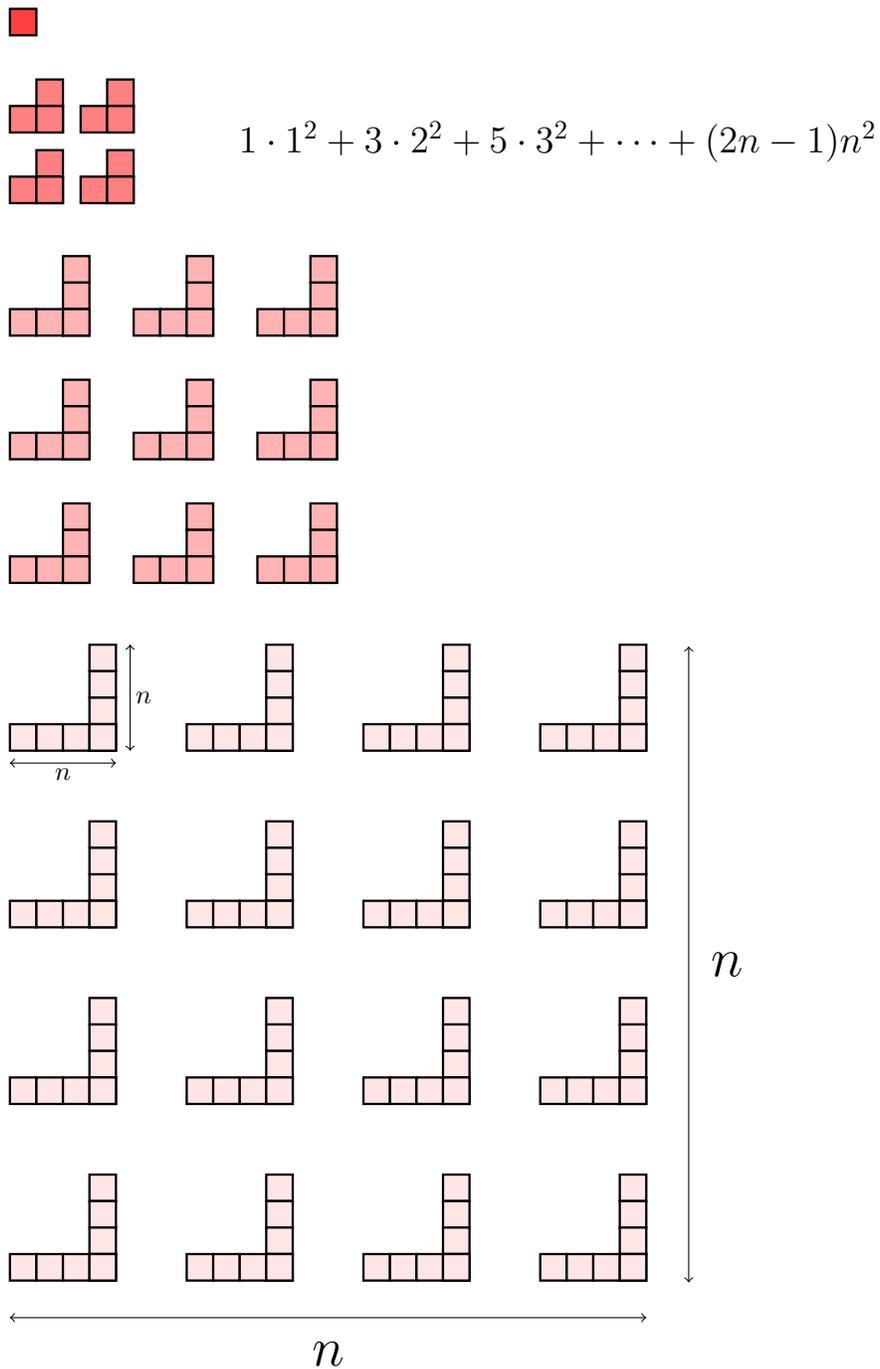
\captionof{figure}{Final layer in the dual square–corner representation}
\label{Top Dual}

\restoreparindent Using the fact that the sum of consecutive odd numbers is a square (see Figure \ref{Odd numbers}), we can represent the dual square–corner decomposition of the final layer in a more condensed form, as shown in Figure \ref{Top Dual condensed}.

\vskip .3cm
\begin{tikzpicture}[scale=0.47]
 \draw (0,0) -- (0,0);
\begin{scope}[shift={(2.5, 0)}] 
\draw [fill=pink!300, line width=0.3mm] (0,3) rectangle (0.75,3.75);

\foreach \p in {0, 5} {
\foreach \q in {0, 5} {
\begin{scope}[shift={(\p, \q-5)}]
 \foreach \x in {0.75} { 
     \foreach \y in {0, -0.75}{
      \draw [fill=pink!200, line width=0.3mm] (\x,\y+3) rectangle (\x+0.75,\y+3.75);
      }
}

\foreach \x in {0, 0.75} { 
     \foreach \y in {-0.75}{
      \draw [fill=pink!200, line width=0.3mm] (\x,\y+3) rectangle (\x+0.75,\y+3.75);
      }
}
\end{scope}

}
}

\foreach \p in {0, 5, 10} {
\foreach \q in {0, 5, 10} {
\begin{scope}[shift={(\p, \q-10)}]
 \foreach \x in {1.5} { 
     \foreach \y in {0, -0.75, -1.5}{
      \draw [fill=pink!120, line width=0.3mm] (\x,\y+3) rectangle (\x+0.75,\y+3.75);
      }
}

\foreach \x in {0, 0.75, 1.5} { 
     \foreach \y in {-1.5}{
      \draw [fill=pink!120, line width=0.3mm] (\x,\y+3) rectangle (\x+0.75,\y+3.75);
      }
}
\end{scope}

}
}

\foreach \p in {0, 5, 10, 15} {
\foreach \q in {0, 5, 10, 15} {
\begin{scope}[shift={(\p, \q-15)}]
 \foreach \x in { 2.25} { 
     \foreach \y in {0, -0.75, -1.5, -2.25}{
      \draw [fill=pink!40, line width=0.3mm] (\x,\y+3) rectangle (\x+0.75,\y+3.75);
      }
}

 \foreach \x in {0, 0.75, 1.5, 2.25} { 
     \foreach \y in {-2.25}{
      \draw [fill=pink!40, line width=0.3mm] (\x,\y+3) rectangle (\x+0.75,\y+3.75);
      }
}

\end{scope}

}
}

\node at (9.75,6) {\huge $1\cdot 1^2 +3\cdot 2^2+5\cdot 3^2+\cdots +(2n-1)n^2$};


\begin{scope}[shift={(-2, 0)}] 
   
\draw[<->] (2,0.4) -- (5,0.4) node[midway, below, yshift=0.5mm] {$n$};
\draw[<->] (5.4,1-0.25) -- (5.4,4-0.25) node[midway, right, xshift=-0.5mm] {$n$};

\end{scope}

\begin{scope}[shift={(-15, 10)}] 

\draw[<->] (15,-25.3) -- (33,-25.3) node[midway, below, yshift=-2mm] {\Huge $n$};

\draw[<->] (34.2,-24.3) -- (34.2,-6.3) node[midway, right, xshift=1.7mm] {\Huge $n$};


\draw [dashed, ultra thick] (19.5,-10.75) rectangle (33.5,-24.75);

\end{scope}

\end{scope}
\end{tikzpicture}

\vskip .8cm

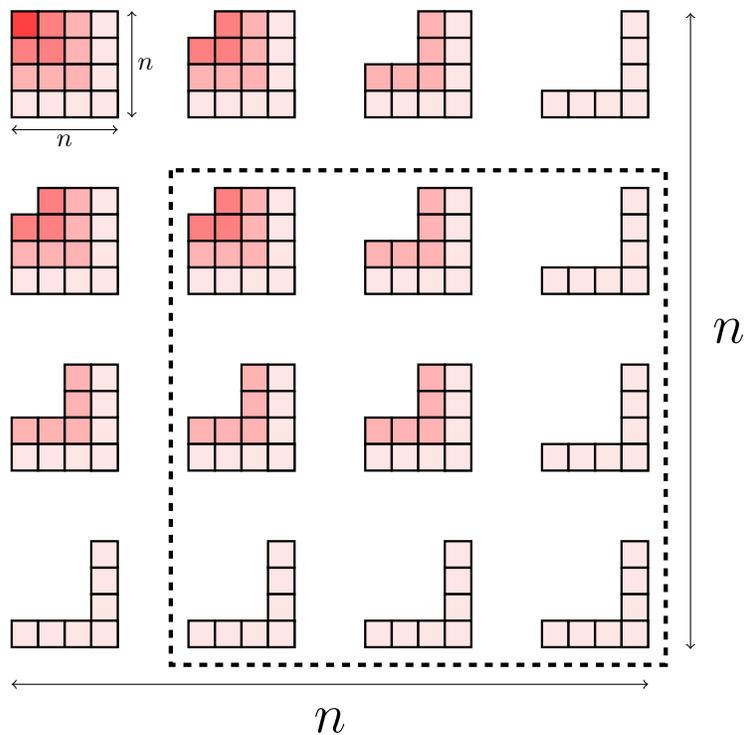
\captionof{figure}{Condensed dual form of the top layer}
\label{Top Dual condensed}

\restoreparindent In the previous diagram, a square has been highlighted with a dashed outline. Since the final layer must be duplicated, we overlap this configuration with the same layer in its original corner–square form, shown again in Figure \ref{Convolution 2}. The dashed square marks the region to be overlapped and helps visualize the argument.

\vskip .3cm
\begin{tikzpicture}[scale=0.47]

\node at (15,-11) {\hskip 1cm \huge $1\cdot 1^2 + 3\cdot 2^2+ 5\cdot 3^2+\cdots +(2n-1)n^2$};

\begin{scope}[shift={(6,0)}] 

\foreach \p in {0, 5, 10, 15} {
\foreach \q in {0, 5, 10, 15} {
\begin{scope}[shift={(\p, \q-33)}]
 \foreach \x in {0, 0.75, 1.5, 2.25} { 
     \foreach \y in {0, -0.75, -1.5, -2.25}{
      \draw [fill=gray!40, line width=0.3mm] (\x,\y+3) rectangle (\x+0.75,\y+3.75);
      }
}
\end{scope}

}
}


\begin{scope}[shift={(13, -18)}] 
   
\draw[<->] (2,0.4) -- (5,0.4) node[midway, below, yshift=0.5mm] {$n$};
\draw[<->] (5.4,1-0.25) -- (5.4,4-0.25) node[midway, right, xshift=-0.5mm] {$n$};

\end{scope}

\begin{scope}[shift={(-15, -8)}] 

\draw[<->] (15,-25.3) -- (33,-25.3) node[midway, below, yshift=-2mm] {\Huge $n$};

\draw[<->] (34.2,-24.3) -- (34.2,-6.3) node[midway, right, xshift=1.7mm] {\Huge $n$};


\fill[white] (14,-5.75) rectangle (28.5,-19.75);

\end{scope}

\begin{scope}[shift={(0, -14)}] 
\foreach \p in {0, 5, 10} {
\foreach \q in {0, 5, 10} {
\begin{scope}[shift={(\p, \q-14)}]
 \foreach \x in {0, 0.75, 1.5} { 
     \foreach \y in {0, -0.75, -1.5}{
      \draw [fill=gray!40, line width=0.3mm] (\x,\y+3) rectangle (\x+0.75,\y+3.75);
      }
}
\end{scope}

}
}
\end{scope}

\fill[white] (-0.2,-13) rectangle (10,-23);

\begin{scope}[shift={(0, -19)}] 
\foreach \p in {0, 5} {
\foreach \q in {0, 5} {
\begin{scope}[shift={(\p, \q-4)}]
 \foreach \x in {0, 0.75} { 
     \foreach \y in {0, -0.75}{
      \draw [fill=gray!40, line width=0.3mm] (\x,\y+3) rectangle (\x+0.75,\y+3.75);
      }
}
\end{scope}

\fill[white] (-0.2,3) rectangle (1.8, 5) ;
}
}
\end{scope}

\draw [fill=gray!40, line width=0.3mm] (0,3-18) rectangle (0.75,3.75-18);

\draw [dashed, ultra thick] (19.5-20,-10.75-3) rectangle (19.5-6,-24.75-3);

\end{scope}
\end{tikzpicture}
\vskip .8cm

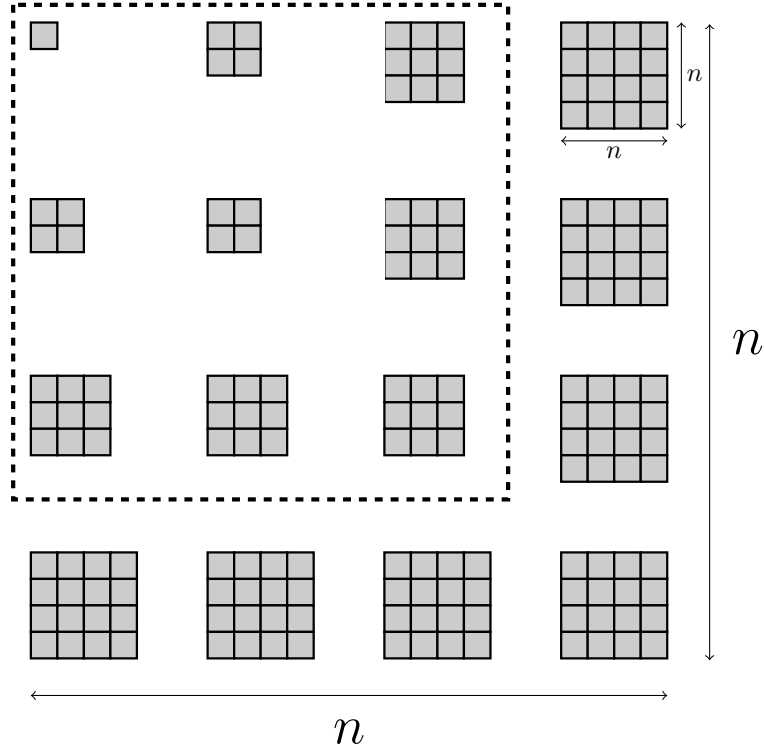
\captionof{figure}{Top layer revisited}
\label{Convolution 2}

\restoreparindent In this way, the duplicated final layer forms a square arrangement of side 
$n+1$, composed of squares of side 
$n$, from which the sum of squares 
$1^2+2^2+3^2+\cdots +n^2$
 has been subtracted twice, as can be seen in Figure \ref{Two copies}.

Now, since the leftover pieces of area 
$\mathcal{R}n(n+1)$ from each of the $n$ layers must be added to the top layer before duplicating it, and taking into account formula $(\ref{Squares})$ for the sum of squares, we obtain:

\begin{align*}
  &2\mathcal{R}n^2(n+1)+n^2(n+1)^2-2\Bigl(1^2+2^2+3^2+\cdots +n^2\Bigr)\\
  &= 2\mathcal{R}n^2(n+1)+n^2(n+1)^2-\frac{n(n+1)(2n+1)}{3} \\&=n(n+1)\Bigl[ 2\mathcal{R} n+n(n+1)-\frac{2n+1}{3}\Bigr]
\end{align*}

We now consider whether the final layer, after adding the leftover pieces and then duplicating it, is equal to each of the first $n$ layers, each of which has area
$$n^2(n+1)^2-\mathcal{R}n(n+1)=n(n+1)\Bigl[ n(n+1)-\mathcal{R}\Bigr].$$

This is the case if the following equality holds:

$$\frac{2n+1}{3}=\mathcal{R}(2n+1),$$
which is equivalent to 
$$\mathcal{R}=\frac{1}{3}.$$

Since $\mathcal{R}=x^2+x$, where $x^2+x=\frac{1}{3}$, it follows that the number of unit squares in each layer must be multiplied by $n+\frac{1}{2}$. Moreover, the above computation shows that the value $x$, which determined the width of the strip that had to be removed from each of the rectangles in the first $n$ layers, must be the positive solution of the equation $x^2+x-\frac{1}{3}=0$, in order to obtain exactly the factor $(n+\frac{1}{2})$.\par
\vspace{5mm}
\noindent
\begin{tikzpicture}[scale=0.44]
 \draw (0,0) -- (0,0);
\begin{scope}[shift={(1.5, 0)}] 
\draw [fill=pink!300, line width=0.3mm] (0,3) rectangle (0.75,3.75);

\foreach \p in {0, 5} {
\foreach \q in {0, 5} {
\begin{scope}[shift={(\p, \q-5)}]
 \foreach \x in {0.75} { 
     \foreach \y in {0, -0.75}{
      \draw [fill=pink!200, line width=0.3mm] (\x,\y+3) rectangle (\x+0.75,\y+3.75);
      }
}

\foreach \x in {0, 0.75} { 
     \foreach \y in {-0.75}{
      \draw [fill=pink!200, line width=0.3mm] (\x,\y+3) rectangle (\x+0.75,\y+3.75);
      }
}
\end{scope}

}
}

\foreach \p in {0, 5, 10} {
\foreach \q in {0, 5, 10} {
\begin{scope}[shift={(\p, \q-10)}]
 \foreach \x in {1.5} { 
     \foreach \y in {0, -0.75, -1.5}{
      \draw [fill=pink!120, line width=0.3mm] (\x,\y+3) rectangle (\x+0.75,\y+3.75);
      }
}

\foreach \x in {0, 0.75, 1.5} { 
     \foreach \y in {-1.5}{
      \draw [fill=pink!120, line width=0.3mm] (\x,\y+3) rectangle (\x+0.75,\y+3.75);
      }
}
\end{scope}

}
}

\foreach \p in {0, 5, 10, 15} {
\foreach \q in {0, 5, 10, 15} {
\begin{scope}[shift={(\p, \q-15)}]
 \foreach \x in { 2.25} { 
     \foreach \y in {0, -0.75, -1.5, -2.25}{
      \draw [fill=pink!40, line width=0.3mm] (\x,\y+3) rectangle (\x+0.75,\y+3.75);
      }
}

 \foreach \x in {0, 0.75, 1.5, 2.25} { 
     \foreach \y in {-2.25}{
      \draw [fill=pink!40, line width=0.3mm] (\x,\y+3) rectangle (\x+0.75,\y+3.75);
      }
}

\end{scope}

}
}

\foreach \p in {20} {
\foreach \q in {18, 23, 28} {
\begin{scope}[shift={(\p, \q-33)}]
 \foreach \x in {0, 0.75, 1.5, 2.25} { 
     \foreach \y in {0, -0.75, -1.5, -2.25}{
      \draw [fill=gray!40, line width=0.3mm] (\x,\y+3) rectangle (\x+0.75,\y+3.75);
      }
}
\end{scope}

}
}

\foreach \p in {5, 10, 15, 20} {
\foreach \q in {13} {
\begin{scope}[shift={(\p, \q-33)}]
 \foreach \x in {0, 0.75, 1.5, 2.25} { 
     \foreach \y in {0, -0.75, -1.5, -2.25}{
      \draw [fill=gray!40, line width=0.3mm] (\x,\y+3) rectangle (\x+0.75,\y+3.75);
      }
}
\end{scope}

}
}

\foreach \p in {15} {
\foreach \q in {23, 28} {
\begin{scope}[shift={(\p, \q-33)}]
 \foreach \x in {0, 0.75, 1.5} { 
     \foreach \y in {0, -0.75, -1.5}{
      \draw [fill=gray!40, line width=0.3mm] (\x,\y+3) rectangle (\x+0.75,\y+3.75);
      }
}
\end{scope}

}
}

\foreach \p in {5, 10, 15} {
\foreach \q in {18} {
\begin{scope}[shift={(\p, \q-33)}]
 \foreach \x in {0, 0.75, 1.5} { 
     \foreach \y in {0, -0.75, -1.5}{
      \draw [fill=gray!40, line width=0.3mm] (\x,\y+3) rectangle (\x+0.75,\y+3.75);
      }
}
\end{scope}

}
}

\foreach \p in {10} {
\foreach \q in {28} {
\begin{scope}[shift={(\p, \q-33)}]
 \foreach \x in {0, 0.75} { 
     \foreach \y in {0, -0.75}{
      \draw [fill=gray!40, line width=0.3mm] (\x,\y+3) rectangle (\x+0.75,\y+3.75);
      }
}
\end{scope}

}
}

\foreach \p in {5, 10} {
\foreach \q in {23} {
\begin{scope}[shift={(\p, \q-33)}]
 \foreach \x in {0, 0.75} { 
     \foreach \y in {0, -0.75}{
      \draw [fill=gray!40, line width=0.3mm] (\x,\y+3) rectangle (\x+0.75,\y+3.75);
      }
}
\end{scope}

}
}

\foreach \p in {5} {
\foreach \q in {28} {
\begin{scope}[shift={(\p, \q-33)}]
 \foreach \x in {0} { 
     \foreach \y in {0}{
      \draw [fill=gray!40, line width=0.3mm] (\x,\y+3) rectangle (\x+0.75,\y+3.75);
      }
}
\end{scope}

}
}

\node at (12,7) {\huge $n^2 (n+1)^2 -2(1^2+2^2+3^2+\cdots + n^2)$};


\begin{scope}[shift={(-2, 0)}] 
   
\draw[<->] (2,0.4) -- (5,0.4) node[midway, below, yshift=0.5mm] {$n$};
\draw[<->] (5.4,1-0.25) -- (5.4,4-0.25) node[midway, right, xshift=-0.5mm] {$n$};

\end{scope}

\begin{scope}[shift={(-15, 10)}] 

\draw[<->] (15,-30) -- (38.2,-30) node[midway, below, yshift=-2mm] {\LARGE $n+1$};

\draw[<->] (38.7,-29.3) -- (38.7,-6.3) node[midway, right, xshift=1.7mm] {\LARGE $n+1$};

\draw [dashed, ultra thick] (34.75,-6) rectangle (38.25,-9.5);

\draw [dashed, ultra thick] (15,-26) rectangle (18.5,-29.5);

\draw [dashed, ultra thick] (19.5,-10.75) rectangle (33.5,-24.75);

\end{scope}

\end{scope}
\end{tikzpicture}

\vskip .8cm

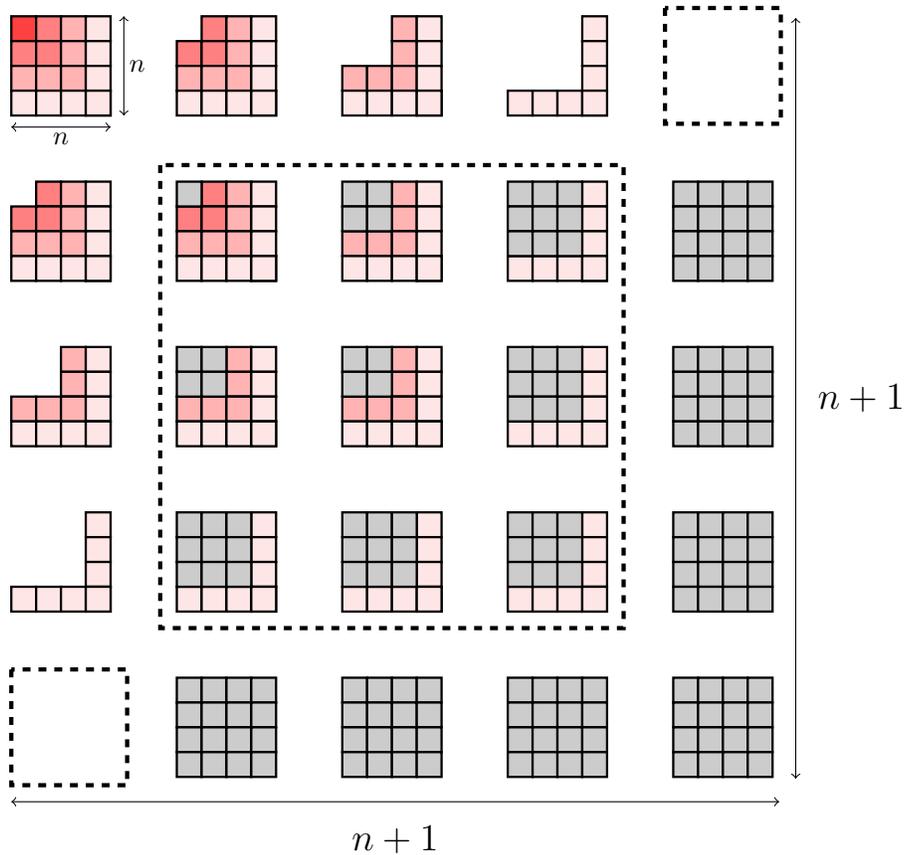
\captionof{figure}{Two copies of the top layer}
\label{Two copies}


\begin{thebibliography}{99}
\bibitem{Abbott} E. A. Abbott, {\em Flatland: A Romance of Many Dimensions}, Seeley \& Co. (1884).

\bibitem{Arakawa} T. Arakawa, T. Ibukiyama, M. Kaneko, D. Zagier, {\em Bernoulli numbers and Zeta Functions}, Springer, 2014.

\bibitem{Heat2} Archimedes, T. L. Heath, The Works of Archimedes, Cambridge University Press (Wellesley College Library) (1897).

\bibitem{Aryabhata} Aryabhata, W. E. Clark, The Aryabhatiya of Aryabhata: An Ancient Indian Work on Mathematics and Astronomy, University of Chicago Press,(1930) ; Kessinger Publishing, (2006).

\bibitem{BeeryFermat}
J.~Beery,
Sums of powers of positive integers: Pierre de Fermat (1601--1665), France,
\emph{Convergence},
Mathematical Association of America (2010),
\url{https://old.maa.org/press/periodicals/convergence/sums-of-powers-of-positive-integers-pierre-de-fermat-1601-1665-france}.


\bibitem{BernoulliArs}
J.~Bernoulli,
\emph{Ars Conjectandi},
translated by E.~D.~Sylla,
Johns Hopkins University Press, Baltimore (2006).

\bibitem{Conway} J. Conway,  R. K. Guy, The Book of Numbers. New York, NY: Springer (1996)

\bibitem{Nicomachus} M. L. D'Ooge, F. E. Robbins , L.C. Karpinski, Nicomachus' Introduction to Arithmetic. Macmillan (1926).

\bibitem{Faulhaber1631}
J.~Faulhaber,
\emph{Academia Algebrae},
Johann Ulrich Sch{\"o}nigs Erben, Ulm (1631).


\bibitem{Gulley} N. Gulley, L. Shure, Nicomachus's Theorem, Matlab Central (https://blogs.mathworks.com/loren/2010/03/04/nichomachuss-theorem/).


\bibitem{Heat} T. L. Heath, A History of Greek Mathematics: Volume I (From Thales to Euclid), Oxford University Press (1921).

\bibitem{Sasho} S. Kalajdzievski, {\em Some evident Summation formulas}, The Mathematical Intelligencer, ({\bf 22}), 47--49, (2000). 

\bibitem{Katz2009}
V.~J.~Katz,
\emph{A History of Mathematics: An Introduction},
3rd ed., Addison--Wesley, Boston (2009).


\bibitem{Knuth1993}
D.~E.~Knuth,
Johann Faulhaber and sums of powers,
\emph{Mathematics of Computation} \textbf{61} (1993), 277--294.

\bibitem{Mahoney1973}
M.~S.~Mahoney,
\emph{The Mathematical Career of Pierre de Fermat},
Princeton University Press, Princeton (1973).


\bibitem{Nelsen} R. B. Nelsen, {\em Proof without words: Exercises in visual thinking}. Mathematical Association of America (1993).

\bibitem{Pascal1654}
B.~Pascal,
\emph{Trait{\'e} du triangle arithm{\'e}tique},
Guillaume Desprez, Paris (1654).

\bibitem{Pengelley} D. Pengelley, Sums of Powers in Discrete Mathematics: Archimedes Sums Squares in the Sand, Convergence (2013), DOI:10.4169/loci003986

\bibitem{Siu} M. K. Siu, Proof without words: Sum of squares, {\em Math. Mag.} 57 no. 2 (1984)

\end{thebibliography}
\end{document}